\documentclass[11pt,leqno]{amsart}

\usepackage[top=2.5 cm,bottom=2 cm,left=2.5 cm,right=2 cm]{geometry}
\usepackage[utf8]{inputenc}

\usepackage{bbm}
\usepackage[mathscr]{eucal}
\usepackage{esint}
\usepackage{graphicx}
\usepackage{hyperref}
\usepackage{color}
\usepackage{amsfonts}
\usepackage{psfrag}
\newcounter{stepnb}

\usepackage{tikz}
\usetikzlibrary{shapes.geometric,calc,decorations.markings,math}

\tikzstyle{nodo}=[circle,draw,fill,inner sep=0pt,minimum size=%
1.5mm]
\tikzstyle{infinito}=[circle,inner sep=0pt,minimum size=0mm]

\newtheorem{theorem}{Theorem}[section]
\newtheorem{lemma}[theorem]{Lemma}
\newtheorem{proposition}[theorem]{Proposition}
\newtheorem{corol}[theorem]{Corollary}

\theoremstyle{remark}
\newtheorem{remark}[theorem]{Remark}

\newtheorem{definition}[theorem]{Definition}
\numberwithin{equation}{section}

\newcommand{\R}{\mathbb{R}}

\newcommand{\Leb}{{\mathscr L}}

\newcommand{\ee}{\varepsilon}

\newcommand{\be}{\begin{equation}}
\newcommand{\eq}{\end{equation}}

\begin{document}

\title[One--dimensional IBVP]{Initial--boundary value problems for merely bounded nearly incompressible vector fields in one space dimension}

\author[S.~Dovetta]{Simone Dovetta}
\address{S.D. Dipartimento di Scienze di Base e Applicate per l'Ingegneria, Università degli Studi di Roma "La Sapienza", via Antonio Scarpa 14, 00161 Roma, Italy.}
\email{simone.dovetta@uniroma1.it}
\author[E. Marconi]{Elio Marconi}
\address{E.M. EPFL B, Station 8, CH-1015 Lausanne, Switzerland.}
\email{elio.marconi@epfl.ch}
\author[L.~V.~Spinolo]{Laura V.~Spinolo}
\address{L.V.S. IMATI-CNR, via Ferrata 5, I-27100 Pavia, Italy.}
\email{spinolo@imati.cnr.it}
\maketitle
{
\rightskip .85 cm
\leftskip .85 cm
\parindent 0 pt
\begin{footnotesize}

{\sc Abstract.}  We establish existence and uniqueness results for initial-boundary value problems for transport equations in one space dimension with nearly incompressible velocity fields, under the sole assumption that the fields are bounded. In the case where the velocity field is either nonnegative or nonpositive, one can rely on similar techniques as in the case of the Cauchy problem. Conversely, in the general case we introduce a new and more technically demanding construction, which heuristically speaking relies on a ``lagrangian formulation" of the problem, albeit in a highly irregular setting. We also establish stability of the solution in weak and strong topologies, and propagation of the $BV$ regularity. In the case of  either nonnegative or nonpositive velocity fields we also establish a $BV$-in-time regularity result, and we exhibit a counterexample showing that the result is false in the case of sign-changing vector fields. To conclude, we establish a trace renormalization property.  
\\



\end{footnotesize}
}

\section{Introduction and main results}
The pioneering works by DiPerna--Lions~\cite{DiPernaLions} and Ambrosio~\cite{Ambrosio} establish existence and uniqueness results for the Cauchy problem for transport and continuity equations with weakly differentiable velocity 
fields: in particular, Ambrosio~\cite{Ambrosio} deals with $BV$ (bounded total variation) vector fields satisfying suitable bounds on the space divergence. In view of the applications to conservation laws, the requirement on the space divergence in~\cite{Ambrosio} can be conveniently replaced by the notion of \emph{nearly incompressible vector field}, which is extensively discussed by De Lellis in~\cite{DL07}. 

The analysis in the present paper is originally motivated by the applications to a family of conservation laws modeling traffic flows on road networks discussed in the companion paper~\cite{DMS}. Here we focus on initial-boundary value problems for transport equations with nearly incompressible vector fields in one space dimension and, under the sole assumption that the vector field is bounded, establish existence and uniqueness results, stability, and propagation of $BV$ regularity. Initial-boundary value problems for transport equations with low regularity coefficients have 
been previously studied for Sobolev~\cite{Boyer} or $BV$~\cite{CCS17,CrippaDonadelloSpinolo}  vector fields. Note that the counterexamples in~\cite{CrippaDonadelloSpinolo} show that, in several space dimensions, as soon as the $BV$ regularity deteriorates at the domain boundary, uniqueness may be lost. Compared to the analysis in~\cite{Boyer,CCS17,CrippaDonadelloSpinolo} we restrict to the one-dimensional case, but impose much weaker assumptions on the velocity field as we remove the requirement that it has Sobolev or $BV$ regularity. Well-posedness results for the Cauchy problem for merely bounded nearly incompressible vector fields have been established by Panov~\cite{Panov}, see also Gusev~\cite{Gusev}. We were able to extend the techniques in~\cite{Gusev,Panov} to tackle the initial-boundary value problem in the case the velocity field is either nonpositive or nonegative. However, as pointed out in Remark~\ref{r:rosso} below, there is apparently a fairly severe obstruction to the extension of these techniques to the case of sign--changing vector fields. To tackle this case, we introduce a new and more technically demanding construction, which we comment upon later on in the introduction. 

We now provide the definition of nearly incompressible vector field in one space dimension (we refer to~\cite{DL07} for the multi-dimensional case). 
\begin{definition}
\label{d:ni1d}
Consider an open interval $]\alpha, \beta[$ and a time $T>0$.  We term a vector field $b \in L^\infty (]0, T[ \times ]\alpha, \beta[; \R)$  \emph{nearly incompressible} if there is a density function $\rho \in  L^\infty (]0, T[ \times ]\alpha, \beta[)$ such that
\begin{itemize}
\item[i)] $\rho\geq0$ a.e. on $]0, T[\times ]\alpha, \beta[$\footnote{Note that some authors give a more restrictive definition and require that $\rho$ is bounded away from $0$, see for instance the recent groundbreaking work~\cite{BianchiniBonicatto}};
\item[ii)] $\rho$ is a distributional solution of the equation
\be \label{e:density1d}
    \partial_t \rho + \partial_x [b \rho] =0 \quad \text{on $]0, T[ \times ]\alpha, \beta[$.}
\eq
\end{itemize}
\end{definition}
To simplify the exposition, in the following we assume that $]\alpha, \beta[$ is a \emph{bounded} interval, but our analysis extends to the cases $\alpha = - \infty$ and $\beta = + \infty$.  We consider the following initial-boundary value problem
\begin{equation}
			\label{IVPtheta2}
			\begin{cases}
			\partial_t\left[ \rho\theta\right]+\partial_x [ b \rho\theta ]=0 
           &  \text{on }  ]0,T[ \times]\alpha, \beta[ \\
               \theta( t, \alpha) = \bar \theta (t) & \text{if $b   (t, \alpha) > 0$} \\
               \theta( t, \beta) = \underline \theta (t) & \text{if $b  (t, \beta) < 0$}   \\  
			\theta(0,\cdot)= \theta_0
			\end{cases}
		\end{equation} 
where $\bar \theta, \underline \theta$ and $\theta_0$ are assigned functions. Some remarks are here in order. First, by combining the equation at the first line of~\eqref{IVPtheta2} with~\eqref{e:density1d} we \emph{formally} obtain the transport equation
\be \label{e:arancione}
    \partial_t \theta+ b \ \partial_x \theta =0.
\eq
Note, however, that we are interested in the case where $b$ is a bounded function and $\partial_x \theta$ is a distribution, and therefore the product $b \ \partial_x \theta$ is not well-defined. Second, in~\eqref{IVPtheta2} we assign the values of $\theta(\cdot, \alpha)$ and $\theta(\cdot, \beta)$ on the set where $b(\cdot, \alpha)>0$ and $b(\cdot, \beta)<0$, respectively. This is consistent with~\eqref{e:arancione}, but in principle nonsensical since $b$ is only an $L^\infty$ function and hence the values $b(\cdot, \alpha)$ and $b(\cdot, \beta)$, which are the values on a zero-measure set, are not well-defined. Following~\cite{CCS17,CrippaDonadelloSpinolo}, we tackle this problem and provide a rigorous formulation of~\eqref{IVPtheta2} by relying on the theory of normal traces for measure-divergence vector fields developed in~\cite{AmbrosioCrippaManiglia,Anzellotti,ChenZiemerTorres}, see~\S\ref{s:distform} and in particular Definition~\ref{d:ibvp}. As a third remark we point out that, by combining the rigorous formulation of the boundary condition given in Definition~\ref{d:ibvp} with Lemma~\ref{L_estimate_traces}, we obtain that in~\eqref{IVPtheta2} one could equivalently assign the boundary condition on the set where $b(\cdot, \alpha) \ge 0$ and $b(\cdot, \beta) \leq 0$, see Remark~\ref{r:libri}. In particular, if $b \ge 0$\footnote{In the following, to simplify the notation we will  explicitely discuss the case $b\ge0$ only, however our results straightforwardly extend to the case $b \leq 0$} then we can consider the initial-boundary value problem
\begin{equation}
	\label{IVPtheta}
	\begin{cases}
	\partial_t\left[ \rho\theta\right]+\partial_x [ b \rho\theta ]=0 
	\quad  \text{on }  ]0,T[ \times]\alpha, \beta[ \\
	\theta( \cdot, \alpha) = \bar \theta  
	\qquad   
	\theta(0,\cdot)= \theta_0.
	\end{cases}
	\end{equation} 
As a matter of fact, all the results contained in this work concerning~\eqref{IVPtheta2} extend to the case of~\eqref{IVPtheta}, provided $b \ge 0$. We now state our existence and uniqueness result. 
\begin{theorem} \label{t:genex}
Fix $T>0$ and an open and bounded interval $]\alpha, \beta[ \subseteq \R$. Assume that $b \in L^\infty (]0, T[ \times ]\alpha, \beta[)$ is a nearly incompressible vector field with density $\rho$. Then for every $\bar \theta, \underline \theta \in L^\infty(]0, T[)$, and $\theta_0 \in L^\infty (]\alpha, \beta[)$ there is 
$\theta \in L^\infty   (]0, T[ \times ]\alpha, \beta[)$ that is solution, in the sense of Definition~\ref{d:ibvp}, of the initial-boundary value problem~\eqref{IVPtheta2}
and satisfies
\be \label{e:linftysol2}
    \| \theta \|_{L^\infty} \leq \max \{ \| \bar \theta \|_{L^\infty} , \|  \theta_0 \|_{L^\infty}, \| \underline \theta \|_{L^\infty}\}. 
\eq
Also, 
if $\theta_1, \theta_2  \in L^\infty   (]0, T[ \times ]\alpha, \beta[)$ are two solutions of~\eqref{IVPtheta2}, then $\rho \theta_1 =\rho \theta_2$ a.e. on $]0, T[ \times ]\alpha, \beta[$. 
\end{theorem} 
Some remarks are again in order. First, the equality $\rho \theta_1 =\rho \theta_2$ is the best 
uniqueness result one can hope for, since the equation at the first line of~\eqref{IVPtheta2} does not provide any information on $\theta$ on the set where $\rho$ vanishes. Second, 
the proof of the uniqueness result heavily relies on the fact that we are focusing 
on the one-dimensional case. Note that the counterexamples in~\cite{CrippaDonadelloSpinolo}  concern zero-divergence and henceforth (it suffices to take $\rho \equiv 1$) nearly incompressible vector fields and show that, in several space dimensions, uniqueness fails as soon as the $BV$ regularity deteriorates at the domain boundary. Third, the most interesting feature of the above result is that we establish uniqueness for sign-changing coefficients $b$, as in this case the techniques developed for the Cauchy problem in~\cite{Gusev,Panov} do not extend to initial-boundary value problems. Fourth, as a byproduct of the uniqueness proof in the general case we also establish a trace renormalization result, see Theorem~\ref{t:fumetti}.
Fifth, we also have stability with respect to weak and strong convergence, see Proposition~\ref{p:caffe}, and a comparison principle given by the following result. 
\begin{corol} \label{c:comparison} Under the same assumptions as in the statement of Theorem~\ref{t:genex}, assume furthermore that $\bar \theta_a \ge \bar \theta_b$, $\underline{\theta}_a\ge \underline{\theta}_b$ and $\theta_{0a} \ge \theta_{0b}$, then the corresponding solutions of~\eqref{IVPtheta2} satisfy $\rho \theta_a \ge \rho \theta_b$ a.e. on $]0, T[ \times ]\alpha, \beta[$.
\end{corol}
The following result establishes propagation of the $BV$ regularity. 
\begin{proposition}\label{p:the}
Under the same assumptions as in the statement of Theorem~\ref{t:genex}, 
assume furthermore that $\bar \theta, \underline \theta \in {BV}(]0,T[)$ and $\theta_0 \in {BV}(]\alpha,\beta[)$, then there is a solution $\theta$ of~\eqref{IVPtheta2}
such that
\begin{equation}\label{e:the}
\begin{split}
\mathrm{TotVar}_{]\alpha,\beta[} \theta(t,\cdot) & \le \mathrm{TotVar} \bar \theta + |\bar \theta(0^+)-\theta_0(\alpha^+)| + \mathrm{TotVar} \theta_0 + |\theta_0(\beta^-)-\underline \theta(0^+)| \\ & \quad + \mathrm{TotVar} \underline \theta 
 \qquad \text{for a.e. $t \in ]0, T[$}.
\end{split}
\eq
Also,  $\theta \in BV (]0, T[ \times ]\alpha, \beta[)$ and 
\be \label{e:caffelatte}
    |D \theta| ( ]0, T[ \times ]\alpha, \beta [)\leq C
    \Big( \| b \|_{L^\infty}, T,  \mathrm{TotVar}\ \bar \theta, |\bar \theta(0^+)-\theta_0(\alpha^+)|, \mathrm{TotVar}\ \theta_0, |\theta_0(\beta^-)-\underline \theta(0^+)| , \mathrm{TotVar} \ \underline \theta \Big),
\eq
where $|D \theta| ( ]0, T[ \times ]\alpha, \beta [)$ denotes the total variation of the distributional gradient of $\theta$.  
\end{proposition}
In~\eqref{e:the} and~\eqref{e:caffelatte}, $\theta_0(\alpha^+)$ and $\theta_0(\beta^-)$ denote the right and the left limit of $\theta_0$ at $\alpha$ and $\beta$, respectively: they are well defined since $\theta_0 \in BV(]\alpha, \beta[)$. We analogously define $\bar \theta (0^+)$ and $\underline \theta (0^+)$. The proof of Proposition~\ref{p:the} highly relies on the fact that the space variable is one-dimensional. Note that, as pointed out in~\cite{ColombiniLuoRauch} the propagation of the $BV$ regularity fails already in two space dimensions, see also~\cite{AlbertiCrippaMazzucato,BrueNguyen,Jabin} and~\cite{Marconi} for the autonomous case. In the case where $b$ is either nonpositive or nonnegative, we also have a $BV$-in-time regularity result. 
\begin{theorem}
\label{t:fanta}
Fix $T>0$ and a bounded, open interval $]\alpha, \beta[ \subseteq \R$ and assume that $b \in L^\infty (]0, T[ \times ]\alpha, \beta[)$ is a nearly incompressible vector field with density $\rho$. Assume furthermore that $b \ge 0$ and that  $\theta_0 \in BV (]\alpha, \beta[)$ and $\bar \theta \in BV (]0, T[)$. Then there is $\theta \in L^\infty (]0, T[ \times ]\alpha, \beta[)$ which  is a solution of~\eqref{IVPtheta} in the sense of Definition~\ref{d:ibvp} and satisfies the following. For every $x \in ]\alpha, \beta[$, there is $\tilde \theta_x  \in L^\infty (]0, T[)$ such that 
\be 
\label{e:truguale}
     \mathrm{Tr}[ b \rho \theta] (\cdot, x)  = \mathrm{Tr}[ b \rho]  (\cdot, x) \tilde \theta_x \quad \text{a.e. on $]0,T[$}
\eq 
and 
\be 
\label{e:bvatx}
 \mathrm{TotVar} \, \tilde \theta_x \leq C(  \mathrm{Tot.Var.} \,  
\theta_0,  \mathrm{Tot Var} \, \bar \theta, |\bar \theta(0^+)-\theta_0(\alpha^+)|).
\eq
Also, if $\kappa \leq \theta_0, \bar \theta \leq K$, then 
\be \label{e:traghetto}
   \kappa \leq  \tilde \theta_x  \leq K \; \text{a.e. on $]0, T[$}.
\eq
The above results extend to the extremum $x = \beta$ provided one replaces 
  $\mathrm{Tr}[ b \rho \theta] (\cdot, x)$ and $ \mathrm{Tr}[ b \rho]  (\cdot, x)$ with 
  $\mathrm{Tr}[ b \rho \theta] (\cdot, \beta^-)$  and $ \mathrm{Tr}[ b \rho]  (\cdot, \beta^-)$, respectively. 
\end{theorem}
Note that, in~\eqref{e:truguale}, $\mathrm{Tr}[ b \rho \theta] (\cdot, x)$ and  $\mathrm{Tr}[ b \rho]  (\cdot, x) $ denote the normal trace of the functions $b \rho \theta$ and $b \rho$, respectively, at $x$: the normal trace is rigorously defined in Remark~\ref{r:blu}. Very loosely speaking, the normal trace can be regarded as a way to define the values of $b \rho \theta$ and $b \rho$ at the point $x$, which a priori is not possible because the segment $\{ (t, x): t \in ]0, T[ \}$ is negligible. To get an heuristic idea of the meaning of~\eqref{e:truguale} one can replace  $\mathrm{Tr}[ b \rho \theta] (\cdot, x)$ and  $\mathrm{Tr}[ b \rho]  (\cdot, x) $ with the pointwise values $b \rho \theta (\cdot, x)$ and  $ b \rho  (\cdot, x) $, respectively. Note furthermore that, if $b$ changes sign, then~\eqref{e:bvatx} fails, as the counterexample discussed in \S\ref{sss:sprite} shows.   

To conclude this introduction we provide some handwaving remarks on the main ideas underpinning the proof of the main results. As pointed out before, we could not extend the techniques in~\cite{Gusev,Panov} to establish the proof of Theorem~\ref{t:genex} in the general case. Instead,   we rely  on a different approach, which is based on a ``lagrangian formulation" of the problem, albeit in a very irregular setting. This approach is inspired by recent works 
on the transport equation with highly irregular, zero divergence vector fields in two space dimensions, see in particular~\cite{AlbertiBianchiniCrippaJEMS} and also~\cite{BonicattoMarconi,Marconi}. Very loosely speaking, the basic idea is the following. Assume that $\theta$ and $b$ are both smooth functions, then the transport equation~\eqref{e:arancione} is well-defined and we can apply the classic method of characteristics: the solution $\theta$ is simply transported along the curves in the $(t, x)$ plane tangent to the vector $(1, b(t, x))$. Assume furthermore that $\rho$ is also a smooth function, and that it is bounded away from $0$. Consider the potential function $Q:[0, T] \times [\alpha, \beta] \to \R$ defined by setting 
\be \label{e:verde}
      \partial_t Q = - b \rho, \quad \partial_x Q = \rho, \quad Q(0, \alpha) =0.
\eq 
Since $\rho$ is bounded away from $0$, we can apply the Implicit Function Theorem and conclude that the level sets of $Q$ are curves in the $(t, x)$ plane tangent to the vector $(1, b(t, x))$, i.e. they are characteristic lines for~\eqref{e:arancione}. In other words, the solution $\theta$ of the transport equation~\eqref{e:arancione} is transported (and henceforth constant) along the level sets of $Q$. In a nonsmooth setting (say $b, \rho \in L^\infty$) the equation $\dot{X} = b(t, X)$ defining the characteristic curves is highly ill posed since both existence and uniqueness may fail. On the other hand, the potential function satisfying~\eqref{e:verde} is Lipschitz continuous, hence its level sets are well-defined and there is hope of showing that $\theta$ is transported, in some weak sense, along the level sets of $Q$. Note, however, that in the present paper we consider the case where $\rho$ can attain the value zero: in this case the level sets of $Q$ may have a nontrivial structure and this is why we have to rely on the more technical definition given by~\eqref{e:vaniglia}. See also \S\ref{ss:roadmap} for some further details about the main ideas underpinning the construction given in the proof of Theorem~\ref{t:genex}.

\subsection*{Outline}
The exposition is organized as follows. In \S\ref{s:overview} we recall some preliminary results, in \S\ref{s:main} we introduce our main argument and provide the proof of Theorem~\ref{t:genex}, and in \S\ref{s:main2} we establish Corollary~\ref{c:comparison} and the trace renormalization property.  In \S\ref{s:confetto} we establish the proof of Proposition~\ref{p:the} and with Proposition~\ref{p:caffe} we discuss the stability of the solution of~\eqref{IVPtheta2} . As a byproduct we obtain an alternative proof of the existence statement in Theorem~\ref{t:genex}, see Remark~\ref{r:cicerchie}. Finally, in \S\ref{s:pos} we establish the proof of Theorem~\ref{t:fanta} and in \S\ref{sss:sprite} we discuss the counterexample showing that Theorem~\ref{t:fanta} fails if $b$ changes sign. We also provide a self-contained proof of Theorem~\ref{t:genex} in the case $b\ge 0$ (or $b \leq 0$). The reader who is only interested in this case can skip \S\ref{s:main} and focus on \S\ref{s:pos}. Note however that \S\ref{s:main} is the most interesting part of the present paper. For the reader's convenience we conclude the introduction by recalling the main notation used in the present paper. 
\subsection*{Notation}
We denote by $C(a_1, \dots, a_\ell)$ a constant only depending on the quantities $a_1, \dots, a_\ell$. Its precise value can vary from occurrence to occurrence. 
\subsubsection*{General mathematical symbols} 
\begin{itemize}
\item $\R_+: =[0, + \infty[$;
\item $\Leb^d$: the $d$-dimensional Lebesgue measure;
\item $\mathrm{Tot Var} \ u$: the total variation of the function $u$;
\item a.e., for a.e. $x$: almost everywhere, for almost every $x$. Unless otherwise specified, it means with respect to the Lebesgue measure;
\item $dx$: integration with respect to the standard Lebesgue measure;
\item $f \Leb^d$, where $f \in L^1 (\R^d)$: the measure on $\R^d$ defined by setting 
$$
    f \Leb^d (E) : = \int_E f (x) dx, 
$$
for every measurable set $E \subseteq \R^d$;
\item $BV$: the space of functions with bounded total variation; 
\item $\mathrm{Tr}[B, \partial \Lambda]$: the normal trace of $B$ on $\partial \Lambda$, see Lemma~\ref{l:p32acm};
\item $\mathcal H^d$: the $d$-dimensional Hausdorff measure;
\item $f\llcorner A$: the restriction of a function $f$ to a set $A$;
\item $\mathrm{Tr}[b \rho \theta] (\cdot, \alpha^+)$, $\mathrm{Tr}[b \rho \theta] (\cdot, \beta^-)$: the normal trace of the function $b \rho \theta$ at the points $\alpha$ and $\beta$, see \S\ref{s:distform};
\item $\mathrm{Tr}[b \rho \theta](\cdot, x)$, $\mathrm{Tr}[b \rho \theta] (\cdot, x)$: the normal trace of the functions $b \rho \theta$ and $b \rho$ at the point $x \in ]\alpha, \beta[$, see Remark~\ref{r:blu};
\item $\mathrm{Tr}[(\rho\theta,\rho b \theta)](t,\gamma(t)^\pm)$: the left and right normal trace of the vector field $(\rho \theta, b \rho \theta)$ along the curve $\gamma$, see the discussion before the statement of Lemma~\ref{L_1};
\item $[\rho \theta]_0$: the initial datum of the function $\rho \theta$, see
\S\ref{s:distform};
\item $\nabla u$: the gradient of the function $u \in C^1(\R^d)$;
\item $Du$: the distributional gradient of the function $u \in L^1_{\mathrm{loc}}(\R^d)$;
\item $|\mu |$: the total variation of the measure $\mu$;
\item $U^-(f,[a,b])$, $L^+(f,[a,b])$: the upper decreasing and the lower increasing envelopes of the function $f$ on $[a, b]$, see~\eqref{e:fragola1} and~\eqref{e:fragola2};
\item $f(]a, b[):$ the image of the function $f$, i.e. 
$$
    f(]a, b[) : = \{ f(x): \; x \in ]a, b[ \}
$$
\item $\bar E$: the closure of the set $E\subseteq \R^d$ with respect to the standard Euclidean topology;
\item $(v_1, v_2)^\perp:=(-v_2, v_1)$: the vector perpendicular to the vector $(v_1, v_2)$ in $\R^2$;
\item $B_r (t, x)$: the open ball of radius $r$ and center at $(t, x)$.
\end{itemize}

\subsubsection*{Symbols introduced in the present paper}
\begin{itemize}
\item $E^-_f$: see~\eqref{e:fragola3};
\item $Q$: the potential function defined by~\eqref{e:verde};
\item $Q_\theta$: the potential function defined by~\eqref{E_def_Q};
\item $\gamma_{\bar t, \bar x}:$ the curve defined by~\eqref{e:vaniglia}; 
\item $t^*(\gamma_{\bar t, \bar x})$, $t_*(\gamma_{\bar t, \bar x})$: see~\eqref{e:cioccolato};
\item $E^-_{\bar t}$, $E^+_{\bar t}$: see~\eqref{E_def_E^-};
\item $x_\alpha (\bar t), x_\beta(\bar t)$: see~\eqref{e:gelato};
\item $t_{\min}, \bar y $: see~\eqref{e:limone}.
\item $\tilde t_*$: see~\eqref{e:bluenavy}
\end{itemize}
\section{Preliminary results} \label{s:overview}
\subsection{A regularity result for zero-divergence vector fields} \label{ss:daf}
We quote a very special case of Lemma 1.3.3 in \cite{Dafermos:book}.
\begin{lemma} \label{l:dafermos}
	Fix an interval $]a, b[ \subseteq \R$. Assume that $u, z \in L^\infty (\R_+ \times ]a, b[ )$ satisfy
	$$
	\partial_t u + \partial_x z =0.
	$$
	Then  $u$ has a representative such that the map $\R_+ \to  L^\infty (]a, b[)$, $t \mapsto u(t, \cdot)$ is continuous with respect to 
	the weak$^\ast$ topology. Also,  $z$  has a representative such that the map 
	$]a, b[  \to  L^\infty (\R_+)$, $x \mapsto z(\cdot, x)$ is continuous with respect to the weak$^\ast$ topology.
\end{lemma}
\begin{remark} \label{r:pointwise}
	In the following, we always use the continuous representative of the maps $t \mapsto u(t, \cdot)$ and $x \mapsto z(\cdot, x)$. In this way, the values $u(t, \cdot)$ and $z(\cdot, x)$ are well defined \emph{for every $t$ and $x$}, respectively. 
\end{remark}
\subsection{Normal traces for measure-divergence vector fields}
\label{ss:acm}
We quote~\cite[Proposition 3.2]{AmbrosioCrippaManiglia}.
\begin{lemma}
\label{l:p32acm}
Let $\Omega \subseteq \R^d$ be an open set and
assume that the distributional divergence of the vector field $C \in L^\infty (\Omega; \R^d)$ is a locally finite Radon measure. Let $\Lambda \subseteq \R^d$ be an open set with $C^1$ boundary, compactly contained in $\Omega$. Then there is a unique function $\mathrm{Tr}[C, \partial \Lambda] \in L^\infty (\partial \Lambda)$ such that 
\be 
\label{e:acmnt}
       \int_\Lambda \mathrm{Div} C \psi dx + \int_\Lambda C \cdot \nabla \psi dx =
       \int_{\partial \Lambda} \mathrm{Tr}[C, \partial \Lambda] \psi d \mathcal H^{d-1}, \quad \text{for every $\psi \in C^\infty_c (\Omega)$} .
\eq
Also, 
$$
    \| \mathrm{Tr}[C, \partial \Lambda]  \|_{L^\infty} \leq \| C \|_{L^\infty}. 
$$
\end{lemma}
If $C$ is a smooth function, then the Gauss-Green formula yields $\mathrm{Tr}[C, \partial \Lambda] = C \cdot \vec n$, where $\vec n$ is the outward pointing, unit normal vector to $\partial \Lambda$. This is the reason why we term $\mathrm{Tr}[C, \partial \Lambda]$ the \emph{normal trace} of  $C$ on $\partial \Lambda$. 
\begin{remark}
\label{r:azzurro} The regularity hypotheses on the open set $\Lambda$ in the statement of Lemma~\ref{l:p32acm} can be considerably weakened. In particular, in the present work we will sistematically apply Lemma~\ref{l:p32acm} to sets with lower regularity: in each case, the extension of formula~\eqref{e:acmnt} can be achieved through an approximation argument with a sequence $\{ \Lambda_n \}$ of regular sets invading $\Lambda$. 
\end{remark}

\subsection{Distributional formulation of the initial-boundary value problem~\eqref{IVPtheta}}
\label{s:distform}
In~\cite{CCS17,CrippaDonadelloSpinolo} the definition of boundary conditions for initial-boundary value problems for transport equations in several space dimensions is provided  by relying on the theory of normal traces and in particular on Lemma~\ref{l:p32acm}. 

Fix $T >0$, $\alpha< \beta \in \R $ and consider a nearly incompressible vector field $b \in L^\infty (]0, T[ \times ]\alpha, \beta[)$ with density ${\rho \in L^\infty (]0, T[ \times ]\alpha, \beta[)}$. Assume $\theta \in L^\infty (]0, T[ \times ]\alpha, \beta[)$ satisfies 
\be
\label{e:suppin2}
    \int_0^T \int_\alpha^\beta \rho \theta (\partial_t \phi + b \partial_x \phi) dx dt =0, 
    \quad \text{for every $\phi \in C^\infty_c (]0, T[ \times ]\alpha, \beta[)$,}
\eq
 then by applying~\cite[Lemma 3.1]{CCS17} with $\Omega= ]\alpha, \beta[$, $d=1$ we get that there are unique functions $\mathrm{Tr}[b \rho \theta](\cdot, \alpha^+)$, $\mathrm{Tr}[b \rho \theta](\cdot, \beta^-)$ and 
$[\rho \theta]_0$ such that  
\be
\label{e:suppout}
    \begin{split}
    \int_0^T \! \! \int_\alpha^\beta \rho \theta (\partial_t \psi + b \partial_x \psi) dx dt & = 
    \int_0^T \! \!  \psi (t, \alpha) \mathrm{Tr} [b \rho \theta] (t, \alpha^+)  dt + \int_0^T \! \!   \psi (t, \beta) \mathrm{Tr} [b \rho \theta] (t, \beta^-)  dt \\
    & \quad -
    \int_\alpha^\beta [\rho \theta]_0 \psi(0, \cdot) dx , 
    \quad \text{for every $\psi \in C^\infty_c (]- \infty, T[ \times \R)$.}
\end{split}
\eq
Note that, if $\rho, b$ and $\theta$ are all smooth functions, then 
$$
   \mathrm{Tr} [b \rho \theta] (\cdot, \alpha^+) = - b \rho \theta (\cdot, \alpha), \qquad 
    \mathrm{Tr} [b \rho \theta] (\cdot, \beta^-) =  b \rho \theta (\cdot, \beta), \qquad  [\rho \theta]_0 = \rho \theta (0, \cdot) 
$$
\begin{remark} \label{r:carta}
Assume that $b \in L^\infty (]0, T[ \times ]\alpha, \beta[)$ is a nearly incompressible vector field with density $\rho$. Then $\theta \equiv 1$ satisfies~\eqref{e:suppin2} and hence by applying the above argument we define the functions $\mathrm{Tr} [b \rho] (\cdot, \alpha^+)$, 
   $ \mathrm{Tr} [b \rho ] (\cdot, \beta^-)$ and  $ [\rho]_0$. Note that they satisfy 
\be
\label{e:suppout2}
    \begin{split}
    \int_0^T \! \! \int_\alpha^\beta \rho (\partial_t \psi + b \partial_x \psi) dx dt & = 
    \int_0^T \! \!  \psi (t, \alpha) \mathrm{Tr} [b \rho] (t, \alpha^+)  dt + \int_0^T \! \!  \psi  (t, \beta) \mathrm{Tr} [b \rho ] (t, \beta^-)  dt \\
    & \quad -
    \int_\alpha^\beta [\rho]_0 \psi(0, \cdot) dx , 
    \quad \text{for every $\psi \in C^\infty_c (]- \infty, T[ \times \R)$.}
\end{split}
\eq
\end{remark}
\begin{remark}\label{r:blu}
Fix $x \in ]\alpha, \beta[$, then by applying~\cite[Lemma 3.1]{CCS17} to the open sets $\Omega= ]\alpha, x[$ and $]x, \beta[$ we recover the definition of the normal trace $\mathrm{Tr}[\rho b \theta](\cdot, x^-)$, $\mathrm{Tr}[\rho b \theta](\cdot, x^+)$, $\mathrm{Tr}[\rho b](\cdot, x^-)$ and $\mathrm{Tr}[\rho b](\cdot, x^+)$. By using the fact that the equations $\partial_t \rho + \partial_x [b \rho]=0$ and $\partial_t [ \rho \theta] + \partial_x [b \rho \theta]=0$ are satisfied on the whole set $]0, T[ \times ]\alpha, \beta[$ one can then  show that 
\be \label{e:blue}
     \mathrm{Tr}[\rho b \theta](\cdot, x^-)= \mathrm{Tr}[\rho b \theta](\cdot, x^+), 
     \qquad 
      \mathrm{Tr}[\rho b](\cdot, x^-)= \mathrm{Tr}[\rho b](\cdot, x^+)
     \qquad \text{a.e. on $]0, T[$}.
\eq
In the following and in the statement of Theorem~\ref{t:fanta} we denote the common values by $ \mathrm{Tr}[\rho b \theta](\cdot, x)$ and $\mathrm{Tr}[\rho b](\cdot, x)$. 
\end{remark}
The following definition is analogous to~\cite[Definition 3.3]{CCS17}. 
\begin{definition}
\label{d:ibvp}
Fix $\alpha< \beta \in \R$, $T>0$ and a nearly incompressible vector field $b \in L^\infty (]0, T[ \times ]\alpha, \beta[)$ with density $\rho \in L^\infty (]0, T[ \times ]\alpha, \beta[)$.  A distributional solution of~\eqref{IVPtheta2} is a function $\theta \in L^\infty (]0, T[ \times ]\alpha, \beta[)$ satisfying~\eqref{e:suppin2} and the equalities 
\begin{equation} \label{e:dataibvp2}
    \left.
    \begin{array}{ccc}
   \mathrm{Tr}  [b \rho \theta] (t, \alpha^+) =  
   \mathrm{Tr}  [b \rho ] (t, \alpha^+) \bar \theta (t) \quad  \text{for a.e. $t: \; \mathrm{Tr}[b \rho](t, \alpha^+)<0$,} \\ 
     \mathrm{Tr}  [b \rho \theta] (t, \beta^-) =  
   \mathrm{Tr}  [b \rho ] (t, \beta^-) \underline{\theta}(t) \quad \text{for a.e. $t: \; \mathrm{Tr}[b \rho](t, \beta^-)<0$}, \\
    {[\rho \theta]_0} (x) = [\rho]_0 (x) \theta_0 (x) \quad 
     \text{for a.e. $x \in ]\alpha, \beta[$}. 
\end{array}
\right. 
\end{equation}
A distributional solution of~\eqref{IVPtheta} is a function $\theta \in L^\infty (]0, T[ \times ]\alpha, \beta[)$ satisfying~\eqref{e:suppin2} and the equalities 
\be \label{e:dataibvp}
   \mathrm{Tr}  [b \rho \theta] (\cdot, \alpha^+) =  
   \mathrm{Tr}  [b \rho ] (\cdot, \alpha^+) \bar \theta \; \text{a.e. in $]0, T[$}, \qquad  
    [\rho \theta]_0 = [\rho]_0 \theta_0 \; \text{a.e. in $]\alpha, \beta[$}. 
\eq
\end{definition}
\subsection{Monotone envelopes}
Fix $T>0$ and let  $f: [0,T] \to \R$ be a Lipschitz continuous function. Fix $[a,b] \subseteq [0,T]$: we term 
$U^-(f,[a,b])$ the upper decreasing envelope of $f$ on $[a,b]$ defined by setting 
\begin{equation} \label{e:fragola1}
U^-(f,[a,b]):= \inf \{ g:[a,b]\to \R : g\ge f\llcorner [a,b], \mbox{ $g$ is monotone non-increasing} \}.
\end{equation}
Similarly we term $L^+(f,[a,b])$ the lower increasing envelope of $f$ on $[a,b]$ defined by setting 
\begin{equation}  \label{e:fragola2}
L^+(f,[a,b]):= \sup \{ g:[a,b]\to \R : g\le f\llcorner [a,b], \mbox{ $g$ is monotone non-decreasing} \}.
\end{equation}
The following properties are known and we report the proof for the sake of completeness. 
\begin{figure}
\centering 
\includegraphics[width=0.7\columnwidth]{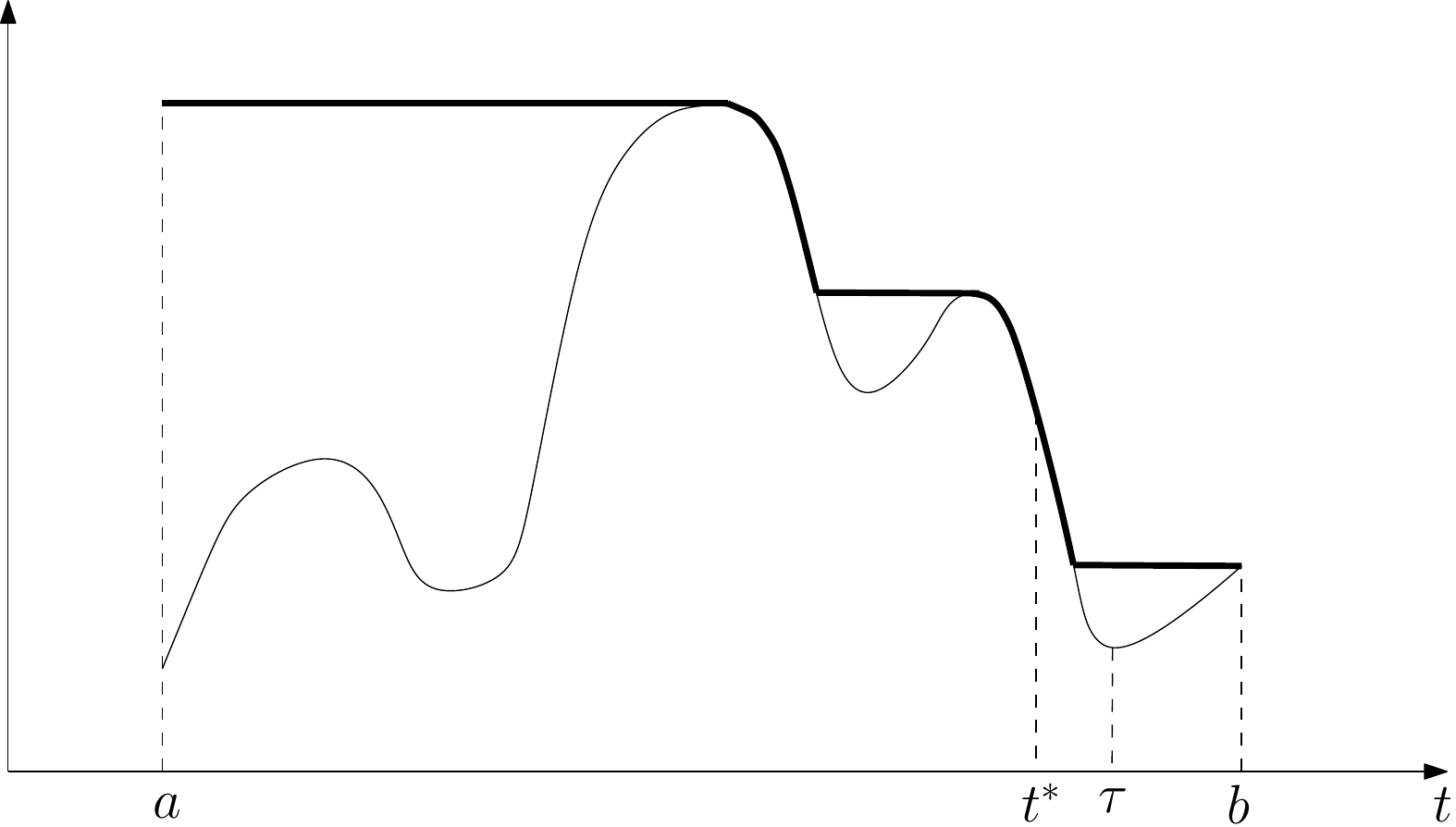}
\label{f:invilupp}
\caption{The upper decreasing envelope of the function $f$}
\end{figure}
\begin{lemma}\label{L_envelope}
Let $f: [0,T] \to \R$ be a Lipschitz continuous function  and let $U^-(f,[a,b])$ be as in~\eqref{e:fragola1} with $[a,b] \subseteq [0,T]$. Then $U^-(f,[a,b])$ is a monotone non-increasing and Lipschitz continuous function. More precisely, set
\begin{equation} \label{e:fragola3}
E^-_f ([a, b])  := \{ x \in [a,b]: f(x) = U^-(f,[a,b])(x)\},
\end{equation}
then  
\begin{equation}\label{E_derivatives_envelope}
\begin{split}
&\frac{d}{dx}U^-(f,[a,b])(x)=\frac{d}{dx}f(x) \quad \mbox{for a.e. }x \in E^-_f ([a, b]), \\
&\frac{d}{dx}U^-(f,[a,b])(x)=0 \quad \mbox{for a.e. }x \in [a,b]\setminus E^-_f([a, b]).
\end{split}
\end{equation}
Analogous statements hold for $L^+(f,[a,b])$.
\end{lemma}
\begin{proof}
We only establish the claims concerning $U^-(f,[a,b])$, those for $L^+(f,[a,b])$ can be recovered by observing that $L^+(f,[a,b])= - U^-(-f, [a, b])$.
Also, to simplify the notation we write $E^-_f$ instead of $E^-_f ([a, b])$. 
 \\
{\sc Step 1:} the function $U^-(f,[a,b])$ is monotone non-increasing because it is the infimum of monotone non-increasing functions. We now establish the Lipschitz regularity of $U^-(f,[a,b])$. We denote by $L$ the Lipschitz constant of $f$ and we fix $x_1<x_2$ in $[a,b]$. Since  $f$ is a $L$-Lipschitz continuous function and $U^-(f,[a,b])\ge f$,
 then the function $g(x):= \max\{ U^-(f,[a,b])(x_2), U^-(f,[a,b])(x_2) - L (x-x_2)\}$ is monotone 
non-increasing and satisfies $g \ge f$. 
This yields $U^-(f,[a,b])(x) \le g(x)$ for every $x \in [a,b]$, and in particular $U^-(f,[a,b])(x_1) \le U^-(f,[a,b])(x_2) + L(x_2-x_1)$. 
Since $U^-(f,[a,b])$ is a monotone non-increasing function, this yields
\begin{equation*}
|U^-(f,[a,b])(x_2)-U^-(f,[a,b])(x_1)|\le L|x_2-x_1|
\end{equation*}
and shows that $U^-(f,[a,b])$ is a $L$-Lipschitz continuous function on $[a,b]$.

We now establish the first equation in~\eqref{E_derivatives_envelope}. We point out that the first equation in~\eqref{E_derivatives_envelope} is satisfied at every accumulation point of $E^-_f$ where both $f$ and  $U^-(f,[a,b])$ are differentiable. Since the set of points that are not accumulation points of $E^-_f$ is negligible, this establishes the first equation in~\eqref{E_derivatives_envelope}.  

We now establish the second equation in \eqref{E_derivatives_envelope}. More precisely, we
show that, if $]x_1,x_2[\subseteq [a,b]\setminus E^-_f$, then $U^-(f,[a,b])$ is constant on $]x_1,x_2[$.
It suffices to show that, for every given $x\in ]x_1,x_2[$, there is $\varepsilon>0$ such that $U^-(f,[a,b])$ is constant on $]x-\varepsilon,x+\varepsilon[$. To this end, we recall that $f$ and $U^-(f,[a,b])$ are both $L$-Lipschitz continous functions and set  
\begin{equation*}
\varepsilon := \frac{U^-(f,[a,b])(x)-f(x)}{2L},
\end{equation*}
then $ U^-(f,[a,b])(x+\varepsilon) \ge f(y)$ for every $y \in ]x -\varepsilon,x+\varepsilon[$. This implies that the function $g:[a, b] \to \R$ defined by setting
\begin{equation*}
g(y): = \begin{cases}
 U^-(f,[a,b])(x+\varepsilon) & \mbox{if }y \in  ]x -\varepsilon,x+\varepsilon[, \\
  U^-(f,[a,b])(y) & \mbox{elsewhere on $[a, b]$}
\end{cases}
\end{equation*}
is a monotone non-increasing function satisfying $g \ge f$ and hence $g\ge U^-(f,[a,b])$. Since $U^-(f,[a,b])$ is a monotone non-increasing function, then 
$U^-(f,[a,b])(y)=U^-(f,[a,b])(x+\varepsilon)$ for every $y \in ]x -\varepsilon,x+\varepsilon[$ and this concludes the proof.
\end{proof} 
\begin{lemma} \label{l:cialda}
Under the same assumptions as in Lemma~\ref{L_envelope}, fix $t_* \in E^-_f ([a, b])$, see Figure~\ref{f:invilupp}. Then for every $\tau \in [t_*, b]$ we have 
\be \label{e:cialda}
      U^-(f, [a, b]) (t) =  U^-(f, [a, \tau ]) (t), \quad \text{for every $t \in [a, t_*]$.} 
\eq 
\end{lemma}
\begin{proof}
We fix $\tau \in [t_\ast, b]$ and proceed according to the following steps.\\
{\sc Step 1:} we show that $t_\ast \in E^-_f ([a, \tau])$.  Since $U^-(f, [a, b])$ is a monotone 
non-increasing function such that $U^-(f, [a, b]) \ge f$ on $[a, \tau] \subseteq [a, b]$, then 
\be \label{e:cialda2}
    f \leq U^-(f, [a, \tau]) \leq U^-(f, [a, b])  \quad \text{on $[0, \tau]$.}
\eq Since $f(t_*) = U^-(f, [a, b])(t_*)$, then $f(t_*) = U^-(f, [a, \tau])(t_*)$. \\
{\sc Step 2:} we conclude the proof. We consider the function $g: [a, b] \to  \R$ defined by setting
$$
    g(x) : = \left\{ 
    \begin{array}{ll}
     U^-(f, [a, \tau])(t) & t \in [a, t_*] \\
     U^- (f, [a, b])(t) & t \in ]t_*, b ] \\
    \end{array}
    \right.
$$
Since $  U^-(f, [a, b])(t_\ast)=f(t_\ast) =   U^-(f, [a,\tau])(t_\ast)$, then $g$ is a monotone non-increasing function satisfying $g \ge f$ on $[a, b]$ and hence 
$g \ge  U^-(f, [a, b])$. Owing to~\eqref{e:cialda2}, this yields~\eqref{e:cialda}.  
\end{proof}
\subsection{An elementary lemma}
We need the following well-known result. We provide the proof for the sake of completeness. 
\begin{lemma}\label{l:pastiera}
Let $[a, b] \subseteq \R$ be a bounded interval. Assume that $\ell, f: [a, b] \to \R$ are two functions such that $f$ is continuous and 
\be \label{e:pastiera}
     |\ell (x_1) - \ell (x_2) | \leq L |f(x_1) - f(x_2)|, \quad \text{for every $x_1, x_2 \in [a, b]$} 
\eq
and for some constant $L>0$. Then there is a $L$-Lipschitz continuous function $g: f([a, b]) \to \R$ such that $\ell = g \circ f$. 
\end{lemma} 
\begin{proof}
We set 
\be \label{e:pastiera2}
      g(y) : = \ell \big( \min \{ x \in [a, b]: \; f(x) = y \} \big). 
\eq
To verify that $\ell = g \circ f$, we fix $\xi \in [a, b]$ and point out that 
$$
    g\circ f (\xi) \stackrel{\eqref{e:pastiera2}}{=}
     \ell \big( \min \{ x \in [a, b]: \; f(x) = f(\xi) \} \big) \stackrel{\eqref{e:pastiera}}{=}
    \ell (\xi).
$$
To establish the Lipschitz continuity, we fix $y_1, y_2  \in f([a, b])$ and set 
$$
    x_1 : = \min \{ x \in [a, b]: \; f(x) = y_1 \}, \qquad 
      x_2 : = \min \{ x \in [a, b]: \; f(x) = y_2 \}. 
$$
Then
\begin{equation*}
\begin{split}
    |g(y_1) - g(y_2)| & \stackrel{\eqref{e:pastiera2}}{=}
    \big| \ell(x_1) - \ell(x_2) \big| 
     \stackrel{\eqref{e:pastiera}}{\leq} L |f(x_1) - f(x_2)| = L |y_1 - y_2|. \qedhere
\end{split}
\end{equation*}
\end{proof}

\section{Proof of Theorem~\ref{t:genex}} \label{s:main}
\subsection{Proof roadmap and outline} \label{ss:roadmap}
Since the proof of Theorem~\ref{t:genex} is rather involved, we first provide some insights about the most important ideas underpinning the main argument. As we mentioned at the end of the introduction, the starting point is the intuition that $\theta$ should be transported, in some weak sense, along the level sets of the potential function $Q$ defined by~\eqref{e:verde}, which as a first approximation should act as ``characteristic curves". Since $\rho$ can attain the value $0$, in general the level sets of $Q$ are not curves and may have a more complicated structure. However, in~\eqref{e:vaniglia} we define the curve $\gamma_{t, x}$ passing through the point $(t, x)$ and this provides a weak notion of ``characteristic curve". We explain in Remark~\ref{r:vaniglia} the heuristic behind~\eqref{e:vaniglia}, here we just point out that, if $\rho$ is bounded away from $0$, then~\eqref{e:vaniglia} defines exactly the level sets of $Q$. Once we have introduced a weak notion of characteristic curve, we have to give a rigorous meaning to the intuition that $\theta$ is constant along the characteristic curves. Since the pointwise values of $\theta$ are not well-defined, we rather look for some more regular function 
which should be constant along our weak characteristic curves. Consider the Lipschitz continuous potential function $Q_\theta: [0, T] \times [\alpha, \beta] \to \R$ defined by setting 
\begin{equation}\label{E_def_Q}
\partial_x Q_\theta = \rho \theta, \qquad \partial_t Q_\theta = -\rho b \theta, \qquad \mbox{and} \qquad Q(0,\alpha)=0.
\end{equation}
In the smooth setting, $Q_\theta$ is constant along the characteristic curves with slope $(1,b)$, i.e. on the level sets of $Q$. One of the outcomes 
of our construction is that $Q_\theta$ is actually constant along the weak characteristic curves $\gamma_{t,x}$. To see this, we rely on a rather indirect argument: first, in \S\ref{ss:giallo} we construct a Lipschitz continuous function $\tilde Q$ which is i) transported along the weak characteristic curves and ii) satisfies the equality $\partial_x \tilde Q= \rho \theta$ at $t=0$ and the equality $\partial_t \tilde Q = - \rho b \theta$ on the part of the boundary where the boundary condition is assigned. Next, we 
recover from $\tilde Q$ a solution of the initial-boundary value problem~\eqref{IVPtheta2} and this establishes the existence part of Theorem~\ref{t:genex}, see Proposition~\ref{P_existence}. Finally, we show that, if $\theta$ is a solution of the initial-boundary value problem~\eqref{IVPtheta2} and $Q_\theta$ is defined by~\eqref{E_def_Q}, then $\tilde Q= Q_\theta$, see Proposition~\ref{P_formula}, and from this the uniqueness statement in Theorem~\ref{t:genex} quite easily follows, see Corollary~\ref{c:bigne}.  To conclude, let us briefly explain how we recover the function $\theta$. What we actually do is we prove that locally in space time $\tilde Q = g \circ Q$ for some suitable Lipschitz continuous function $g: I \subset \R \to \R$, see Lemma~\ref{L_tildeQ} for the exact statement. Next, we show that the function $g'\circ Q$ provides a solution of~\eqref{IVPtheta2}. To understand why this is true one should  keep in mind that, as a matter of fact, $\tilde Q= Q_\theta$ and hence by comparing~\eqref{e:verde} and~\eqref{E_def_Q} we get $\partial_x \tilde Q = \theta \partial_x Q$ and, owing to the equality $\tilde Q = g \circ Q$, this yields $\theta \partial_x Q= g' (Q) \partial_x Q$.

The exposition is organized as follows. In \S~\ref{ss:Q} we discuss some properties of the functions $Q$ and $Q_\theta$.  In~\S\ref{ss:marrone} we introduce our weak notion of characteristic curve and establish its main properties. In \S\ref{ss:giallo} we introduce the above mentioned function $\tilde Q$ and study its main properties, in~\S\ref{ss:indaco} we establish the existence part of Theorem~\ref{t:genex} and in \S\ref{ss:lavanda} the uniqueness. 

To conclude, we point out that to simplify some points of the argument in this section and in the following one we always assume
 \begin{equation}\label{E_ass}
T \le \bar T:=\frac{\beta - \alpha}{2\|b\|_{L^\infty}}.
\end{equation}
Up to repeating the argument on each time interval of the form $[k \bar T, (k+1)\bar T]$ with $k\in \mathbb{N}$, this is not restrictive.
\subsection{The potential function $Q$}
\label{ss:Q}
We fix $\theta \in  L^\infty (]0, T[ \times ]\alpha, \beta[)$ solving \eqref{e:suppin2} and we denote by 
$Q$ and $Q_\theta$  the Lipschitz continuous potential functions defined by~\eqref{e:verde} and~\eqref{E_def_Q}, respectively. Let $t^-, t^+ \in ]0, T[$ and consider a Lipschitz continuous curve $\gamma:[t^-,t^+] \to ]\alpha,\beta[$. We
now want to define the left and right normal traces 
$\mathrm{Tr}[(\rho\theta,\rho b \theta)](t,\gamma(t)^-)$ and $\mathrm{Tr}[(\rho\theta,\rho b \theta)](t,\gamma(t)^+)$. To this end, we fix $\bar x \in ]\alpha, \inf \gamma [$, we  set 
\be
\label{e:pistacchio}
       \Lambda: = \{ (t, x):~\; t \in ]t^-, t^+[, \; 
        \bar x < x < \gamma (t) \},
\eq
we apply Lemma~\ref{l:p32acm} with $\Omega= ]0, T[ \times ]\alpha, \beta[$, $C= (\rho \theta, b \rho \theta)$
and we denote the normal trace $\mathrm{Tr}[C, \partial \Lambda](t)$ by $\mathrm{Tr}[(\rho\theta,\rho b \theta)](t,\gamma(t)^-)$. To define $\mathrm{Tr}[(\rho\theta,\rho b \theta)](t,\gamma(t)^+)$ we fix $\underline x \in ] \sup \gamma, \beta[$ and we apply Lemma~\ref{l:p32acm} with  $\Lambda: = \{ (t, x):~\; t \in ]t^-, t^+[, \; \gamma (t) x < x < \underline x \}$.   
\begin{lemma}\label{L_1}
	Let $t^-, t^+ \in ]0, T[$ and $\gamma:[t^-,t^+] \to ]\alpha,\beta[$ as before.  
	Then 
	\begin{equation}\label{E_trace_Q_theta}
	\frac{d}{dt} Q_\theta(t,\gamma(t))= \pm \sqrt{1+\dot\gamma(t)^2} \mathrm{Tr}[(\rho\theta,\rho b \theta)](t,\gamma(t)^\pm), \;  \text{for a.e. $t \in ]t^-, t^+[$}.
	\end{equation}
\end{lemma}
\begin{proof}
	We establish \eqref{E_trace_Q_theta} for $\mathrm{Tr}[(\rho\theta,\rho b \theta)](t,\gamma(t)^-)$, the proof for 
   $\mathrm{Tr}[(\rho\theta,\rho b \theta)](t,\gamma(t)^+)$ is analogous.
	We fix $t_1<t_2$ in $]t^-,t^+[$, $\bar x \in ]\alpha, \inf \gamma[$ and consider the set~\eqref{e:pistacchio} 
and the vector field $C: = (\rho \theta, b \rho \theta)$. By relying on a standard approximation argument one can show that, for a.e. $t_1, t_2 \in ]t^-, t^+[$ and 
$\bar x \in ]\alpha, \inf \gamma[$ we have the equalities 
\be \label{e:tiramisu}
     \left.\mathrm{Tr} [C, \partial \Lambda] \right|_{t=t_1} = - \rho \theta(t_1, \cdot), \left. \quad 
     \mathrm{Tr} [C, \partial \Lambda] \right|_{t=t_2} =  
     \rho \theta(t_2, \cdot),  \left. \quad 
     \mathrm{Tr} [C, \partial \Lambda] \right|_{x=\bar x} = - b  \rho \theta(\cdot, \bar x) 
\eq
Also, for a.e. $t_1, t_2 \in ]t^-, t^+[$ and a.e. $\bar x \in ]\alpha, \inf \gamma[$ we have  
\be \label{e:crostata}
       \partial_x Q_\theta (t_1,\cdot)= \rho\theta(t_1,\cdot),
      \; 
      \partial_x Q_\theta(t_2,\cdot)= \rho\theta(t_2,\cdot) \;
      \text{a.e. in $]\alpha, \beta[$ and}
    \; \partial_t Q_\theta (\cdot,\bar x)= -\rho b\theta (\cdot,\bar x)
    \; \text{a.e. in $]0, T[$.}
\eq
We now fix $t_1, t_2 \in ]t^-, t^+[$, $t_1<t_2$ and $\bar x \in ]\alpha, \inf \gamma[$ in such a way that~\eqref{e:tiramisu} and~\eqref{e:crostata} are both satisfied. We set 
$
    \Lambda : = \{ (t, x): \ t \in ]t_1, t_2[, \ \bar x < x< \gamma(t) \}
$
and we apply formula~\eqref{e:acmnt} with $C = (\rho \theta, b \rho \theta)$ and a test function $\psi$ such that $\psi \equiv 1$ on $\Lambda$. We recall that $\mathrm{Div} \ C =0$ owing to~\eqref{e:suppin2} and that the restriction of the Hausdorff measure $\mathcal H^1$ to the graph of $\gamma$ is defined by the formula 
$$
    \int_\gamma f ds = \int_{t_-}^{t^+} f (t, \gamma(t)) 
     \sqrt{1+\dot\gamma(t)^2} dt, \quad \text{for every $f: \R^2 \to \R$ bounded Borel function}. 
$$
We eventually arrive at 
	\begin{equation*}
	\begin{split}
	-\int_{t_1}^{t_2}&\mathrm{Tr}[(\rho\theta,\rho b\theta)](t,\gamma(t)^-)  \sqrt{1+\dot\gamma(t)^2}dt
  \\
  & \stackrel{\eqref{e:acmnt},\eqref{e:tiramisu}}{=} 
   \int_{\bar x}^{\gamma(t_2)} \rho\theta(t_2,x)dx - \int_{t_1}^{t_2} \rho b \theta(t,\bar x) dt  -  \int_{\bar x}^{\gamma(t_1)} \rho\theta(t_1,x)dx \\
	& \stackrel{\eqref{e:crostata}}{=}
    ~ \big( Q_\theta(t_2,\gamma(t_2)) - Q_\theta(t_2,\bar x)\big) + \big( Q_\theta (t_2,\bar x) - Q_\theta(t_1,\bar x) \big) + \big( Q_\theta(t_1,\bar x) - Q_\theta(t_1,\gamma(t_1))\big) \phantom{\int}\\
	&=~  Q_\theta(t_2,\gamma(t_2)) - Q_\theta(t_1,\gamma(t_1)) \phantom{\int}
	\end{split}
	\end{equation*}
	and by the arbitrariness of $t_1$ and $t_2$ this establishes \eqref{E_trace_Q_theta}.
\end{proof}
By relying on an argument similar to the one in the proof of Lemma~\ref{L_1} one can establish the following equalities. 
\begin{lemma} \label{l:crema}
     We have 
	\begin{equation}\label{E_trace_boundary}
	\partial_t Q_\theta(t,\alpha)= \mathrm{Tr}[\rho b \theta](t,\alpha^+) \qquad \mbox{and} \qquad
	\partial_t Q_\theta(t,\beta)= -\mathrm{Tr}[\rho b \theta](t,\beta^-),  \;  \text{for a.e. $t \in ]t^-, t^+[$}
	\end{equation}
	and 
	\begin{equation}\label{E_trace_initial}
	\partial_x Q_\theta (0,x) = [\rho\theta]_0(x)
      ,  \;  \text{for a.e. $x \in ]\alpha, \beta[$}.
	\end{equation}
\end{lemma}
\begin{remark}\label{R_potential}
By applying Lemma \ref{l:crema} we can express the boundary conditions \eqref{e:dataibvp2} in terms of  $Q_\theta$ and $Q$, that is 
\begin{equation} \label{e:dataibvp2pot}
    \left.
    \begin{array}{ccc}
   \partial_t Q_\theta(t,\alpha) =  
  \partial_t Q (t, \alpha) \bar \theta (t) \quad  \text{for a.e. $t$ such that $ \partial_t Q(t,\alpha)>0$,} \\ 
    \partial_t Q_\theta(t,\beta) =  
   \partial_t Q(t,\beta)  \underline{\theta}(t) \quad \text{for a.e. $t$ such that $ \partial_t Q(t,\beta)<0$}, \\
    \partial_x Q_\theta(0,x) = \partial_xQ(0,x) \theta_0 (x) \quad 
     \text{for a.e. $x \in ]\alpha, \beta[$}. 
\end{array}
\right. 
\end{equation}
\end{remark}
\subsection{Weak characteristic curves} \label{ss:marrone}
In this paragraph we introduce a monotone family of Lipschitz continuous curves along which the potential $Q$ is transported. By a slight abuse of notation, we term them ``weak characteristic curves" because, as pointed out in \S\ref{ss:roadmap}, we regard them as the nonsmooth counterpart of the classical characteristic curves defined in a smooth setting. In Remark~\ref{r:vaniglia} we comment about definition~\ref{e:vaniglia}. 
\begin{lemma}\label{L_2}
	Fix $(\bar t,\bar x)\in ]0,T[\times ]\alpha,\beta[$ and set $\bar h= Q(\bar t, \bar x)$. We define the \emph{characteristic curve}
	$\gamma_{\bar t,\bar x}:[0,T] \to [\alpha,\beta]$ by
     setting 
	\begin{equation} \label{e:vaniglia}
	\gamma_{\bar t, \bar x}(t) :
     = \max \Bigg\{ \alpha \ ;  \
     \min \Big\{ \beta \ ; \ 
                \inf \{ x \in ]\alpha,\beta[ :
                Q(t,x)>\bar h \} \ ; \ 
                \bar x - \|b\|_{L^\infty}(t-\bar t) 
     \Big\}
     \Bigg\},
	\end{equation}
	where we have used the convention $\inf \emptyset = +\infty$.	Moreover, we set
	\begin{equation} \label{e:cioccolato}
	\begin{split}
	t^*(\gamma_{\bar t, \bar x})&:=\sup\{ t \in [\bar t,T]: \gamma_{\bar t,\bar x}(s) \in ]\alpha,\beta[ \mbox{ for all }s \in [\bar t,t[\}, \\
	t_*(\gamma_{\bar t, \bar x})&:=\inf\{ t \in [0,\bar t]: \gamma_{\bar t,\bar x}(s) \in ]\alpha,\beta[ \mbox{ for all }s \in ]t,\bar t]\}.
	\end{split}
	\end{equation}
	Then for every $(\bar t,\bar x) \in ]0,T[\times ]\alpha,\beta[$ we have the following properties:
	\begin{enumerate}
		\item[\emph{(a)}] $\gamma_{\bar t,\bar x}$ is a $\|b\|_{L^\infty}$-Lipschitz continuous curve;
		\item[\emph{(b)}] $\gamma_{\bar t,\bar x}(\bar t)=\bar x$; 
		\item[\emph{(c)}] $t_*(\gamma_{\bar t,\bar x}) \in [0, \bar t[$ and  $t^*(\gamma_{\bar t,\bar x}) \in ]\bar t,T]$;
		\item[\emph{(d)}] if $t_*(\gamma_{\bar t,\bar x})>0$, then $\gamma_{\bar t,\bar x}(t_*(\gamma_{\bar t,\bar x}))\in \{\alpha, \beta\}$;
		\item[\emph{(e)}] if $t^*(\gamma_{\bar t,\bar x})<T$, then $\gamma_{\bar t,\bar x}(t^*(\gamma_{\bar t,\bar x}))\in \{\alpha, \beta\}$;
		\item[\emph{(f)}] $Q(t,\gamma_{\bar t,\bar x}(t))=\bar h$ for every $t \in [t_*(\gamma_{\bar t,\bar x}), t^*(\gamma_{\bar t,\bar x})]$;
		\item[\emph{(g)}] for a.e. $t \in [t_*(\gamma_{\bar t,\bar x}), t^*(\gamma_{\bar t,\bar x})]$ we have 
		\begin{equation*}
		\mathrm{Tr}[(\rho,\rho b)](t,\gamma_{\bar t,\bar x}(t)^-)=\mathrm{Tr}[(\rho,\rho b)](t,\gamma_{\bar t,\bar x}(t)^+)=0;
		\end{equation*}
		\item[\emph{(h)}] for every $\alpha <\bar x_1 \le \bar x_2 < \beta$ and every $ t \in [0,T]$ we have  $\gamma_{\bar t,\bar x_1}(t) \le \gamma_{\bar t,\bar x_2}(t)$.
	\end{enumerate}
\end{lemma}
\begin{remark}\label{r:vaniglia}
The basic idea underpinning definition~\eqref{e:vaniglia} is very loosely speaking the following.  
As pointed out in \S\ref{ss:roadmap}, as a very first approximation we would like to define $\gamma_{\bar t,\bar x}$ as a level set of $Q$, namely the set of points such that $Q= Q(\bar t, \bar x) = \bar h$. However, since $\rho$ can attain the value $0$, the level set of $Q$ is in general not a curve. To tackle this problem we introduce the term $\inf \{ x \in ]\alpha,\beta[ :
                Q(t,x)>\bar h \}$. The term 
$\bar x - \|b\|_{L^\infty}(t-\bar t) $ makes sure that $\gamma_{\bar t,\bar x}$ passes through the point $(\bar t, \bar x)$, whereas the remaining terms enforce the fact that $\gamma_{\bar t,\bar x}$ is confined between $\alpha$ and $\beta$. 
\end{remark}
\begin{proof}[Proof of Lemma~\ref{L_2}]
We first establish property {(a)}.  Since $t \mapsto \bar x - (t-\bar t)\|b\|_{L^\infty}$ is a $\|b\|_{L^\infty}$-Lipschitz continuous curve, it suffices to show that 
	\begin{equation}\label{E_tilde_gamma}
	\tilde \gamma_{\bar t, \bar x}(t) : = 
    \max \Bigg\{ \alpha \ ;  \
     \min \Big\{ \beta \ ; \ 
                \inf \{ x \in ]\alpha,\beta[ :
                Q(t,x)>\bar h \} 
     \Big\}
     \Bigg\} 
	\end{equation}
	is a $\|b\|_{L^\infty}$-Lipschitz continuous curve on  $[0,T]$.
	Assume by contradiction that there are $t_1, t_2 \in [0, T]$, $t_1< t_2$, such that 
	$\tilde \gamma_{\bar t,\bar x}(t_1) < \tilde \gamma_{\bar t,\bar x}(t_2) - \|b\|_{L^\infty}(t_2-t_1)$. Consider the curve 
	$\gamma(t)= \tilde \gamma_{\bar t,\bar x}(t_2) + \|b\|_{L^\infty}(t-t_2)$,  $t \in [t_1,t_2]$.
	Since $\dot \gamma(t)= \|b\|_{L^\infty}$ and 
	\begin{equation*}
	(\rho(t,x), \rho(t,x) b(t,x)) \cdot (1, \|b\|_{L^\infty})^\perp =  \rho b -\rho \|b\|_{L^\infty}   \leq 0 \qquad \mbox{for a.e. }(t,x)\in ]0,T[\times ]\alpha,\beta[,
	\end{equation*}
	then $\mathrm{Tr}[(\rho,\rho b)](t,\gamma(t)^+)\ge 0$.  By applying Lemma \ref{L_1} with $\theta \equiv 1$ we get 
	\begin{equation}\label{E_ineq_Q}
	\frac{d}{dt}Q(t,\gamma(t))\ge 0 \qquad \Longrightarrow \qquad Q(\gamma(t_2))\ge Q(\gamma(t_1)).
	\end{equation} 
	We now show that this contradicts the definition of $\tilde \gamma_{\bar t,\bar x}$.
Since by construction $\gamma(t_2) \in ]\alpha,\beta]$, then owing to the continuity of $Q$ we have $Q(t_2,\gamma(t_2))\le \bar h$ and hence the inequality~\eqref{E_ineq_Q} yields $Q(t_1,\gamma(t_1))\le \bar h$.
	Since $\tilde \gamma_{\bar t,\bar x}(t_1)<\gamma(t_1)$ this contradicts the minimality in \eqref{E_tilde_gamma}.
	
	Similarly, assume by contradiction that there is $t_1<t_2$ such that 
	$\tilde \gamma_{\bar t,\bar x}(t_1) > \tilde \gamma_{\bar t,\bar x}(t_2) + \|b\|_{L^\infty}(t_2-t_1)$ and consider the curve 
	$\gamma(t)= \tilde \gamma_{\bar t,\bar x}(t_2) + \|b\|_{L^\infty}(t_2-t)$ for $t \in [t_1,t_2]$.
	The same argument as before yields 
	$
	Q(t_2,\gamma(t_2))\le Q(t_1,\gamma(t_1))
	$ and, since 
	by construction $\gamma(t_2) \in [\alpha, \beta[$, then owing to the continuity of $Q$ we have $Q(t_2,\gamma(t_2))\ge \bar h$ and this in turn yields $Q(t_1,\gamma(t_1))\ge \bar h$.
	Since $\tilde \gamma_{\bar t,\bar x}(t_1)>\gamma(t_1)$ this contradicts the minimality in \eqref{E_tilde_gamma}.
	This completes the proof of property {(a)}. 

To establish property {(b)}, we recall that, since $\rho \ge 0$, then for every $t \in [0, T]$ the function $Q(t, \cdot)$ is monotone non-decreasing. This implies that $ \inf \{ x \in ]\alpha,\beta[ :
                Q(\bar t,x)>\bar h \} \ge \bar x$  and by using formula~\eqref{e:vaniglia} this yields property {(b)}. 

To establish property {(c)} we point out that, by definition, $t^*(\gamma_{\bar t,\bar x})\in [\bar t,T]$. Since $\gamma_{\bar t,\bar x}(\bar t)=\bar x \in ]\alpha,\beta[$ and 
	$\gamma_{\bar t,\bar x}$ is a Lipschitz continuous curve, then $t^*(\gamma_{\bar t,\bar x})> \bar t$. An analogous argument shows that $t_*(\gamma_{\bar t,\bar x})< \bar t$.

Properties {(d)} and {(e)} both follow from the definition of $t_*(\gamma_{\bar t,\bar x})$ and 
$t^*(\gamma_{\bar t,\bar x})$ and the Lipschitz regularity of  $\gamma_{\bar t,\bar x}$.
 	
	We now establish property {(f)}. Assume that $t \in [0, T]$ satisfies $\gamma_{\bar t,\bar x}(t)= \tilde \gamma_{\bar t,\bar x}(t)$, where $\tilde \gamma$ is the same curve as in~\eqref{E_tilde_gamma}. 
The continuity of $Q$ implies that, if  $t \in ]t_*(\gamma_{\bar t,\bar x}), t^*(\gamma_{\bar t,\bar x})[$, then $Q(t,\gamma_{\bar t,\bar x}(t))=\bar h$. We are thus left to show that $Q(t,\gamma_{\bar t,\bar x}(t))=\bar h$ for every 
	$t \in]t_*(\gamma_{\bar t,\bar x}), t^*(\gamma_{\bar t,\bar x})[$ such that $\gamma_{\bar t,\bar x}(t)= \bar x - \|b\|_{L^\infty}(t- \bar t)$.
	We prove it for $t \in ]t_*(\gamma_{\bar t,\bar x}),\bar t]$, the case $t \in [\bar t, t^*(\gamma_{\bar t,\bar x})[$ is analogous.
	The same argument as in the proof of property {(a)} yields 
	\begin{equation}\label{E_barh>}
	Q(t, \bar x - \|b\|_{L^\infty}(t-\bar t)) \ge Q(\bar t,\bar x) = \bar h.
	\end{equation}
	Owing to the definition of $\gamma_{\bar t,\bar x}$ and the monotonicity of $Q$ with respect to $x$, the condition $\gamma_{\bar t,\bar x}(t)= \bar x - \|b\|_{L^\infty}(t- \bar t)$ holds only if 
	\begin{equation}\label{E_barh<}
	Q(t, \bar x - \|b\|_{L^\infty}(t- \bar t)) \le \bar h.
	\end{equation}
	By combining \eqref{E_barh>} and \eqref{E_barh<} we complete the proof of property {(f)}.
	
	Property {(g)} follows from property {(f)} and Lemma \ref{L_1}. Finally, by combining definition~\eqref{e:vaniglia} and the monotonicity of $Q$ with respect to $x$ we establish property {(h)}.
\end{proof}
\subsection{The function $\tilde Q$}\label{ss:giallo}
\subsubsection{Definition of $\tilde Q$} \label{sss:tildeq}
We first provide the formal definition of $\tilde Q$ and then in Remark~\ref{r:glicine} discuss the basic ideas underpinning its construction. Let $Q$ be as before the potential function defined by~\eqref{e:verde}. Given $\bar t \in ]0,T[$, we set 
\begin{equation}\label{E_def_E^-}
\begin{split}
E^-_{\bar t}:=&~ \left\{ t \in [0,\bar t]: U^-(Q(\cdot,\alpha),[0,\bar t])(t)=Q(t,\alpha)\right\} \\
E^+_{\bar t}:= &~\left\{ t \in [0,\bar t]: L^+(Q(\cdot,\beta),[0,\bar t])(t)=Q(t,\beta)\right\}.
\end{split}
\end{equation}
and 
\begin{equation} \label{e:gelato}
x_\alpha(\bar t):=\inf\{x \in ]\alpha,\beta[: t_*(\gamma_{\bar t, x})=0\} \qquad \mbox{and} \qquad x_\beta (\bar t):=  \sup\{x \in ]\alpha,\beta[: t_*(\gamma_{\bar t, x})=0\},
\end{equation}
where $t_*$ and $t^*$ are the same as in~\eqref{e:cioccolato}. Note that $ \{x \in ]\alpha,\beta[: t_*(\gamma_{\bar t, x})=0\} \neq \emptyset$ owing to~\eqref{E_ass} and to the fact that $\gamma_{\bar t, x}$ is a $\| b \|_{L^\infty}$-Lipschitz continuous curve, see Lemma~\ref{L_2}. This yields $x_\alpha(\bar t) \in [\alpha,\beta[$, $x_\beta(\bar t) \in ]\alpha,\beta]$ and $x_\alpha (\bar t) \leq x_\beta (\bar t)$. 

We now fix $\theta_0 \in L^\infty (]\alpha, \beta[)$ and $\bar \theta, \underline\theta \in L^\infty (]0, T[)$ and define the function $\tilde Q:]0,T[ \times ]\alpha,\beta[ \to \R$. Given $(\bar t, \bar x) \in ]0, T[ \times ]\alpha, \beta[$,  we define the value $Q(\bar t, \bar x)$ by separately considering the following cases: \\
{\sc Case 1:} $x_\alpha(\bar t)< \bar x < x_\beta(\bar t)$. We set
\begin{equation} \label{e:sacher}
\tilde Q (\bar t,\bar x) = \int_\alpha^{\gamma_{\bar t,\bar x}(0)}[\rho]_0 (x)\theta_0(x)dx.
\end{equation}
{\sc Case 2:} $\alpha < \bar x< x_\alpha(\bar t)$. We set 
\begin{equation} \label{e:limone}
t_{\min}:=\min E^-_{\bar t} ,
\qquad \qquad \qquad 
\bar y :=\begin{cases}
\alpha & \text{if }t_{\min} =0\\
\inf_{\varepsilon>0}\gamma_{t_{\min},\alpha+ \varepsilon}(0) & \text{otherwise} 
\end{cases}
\end{equation}
and 
\begin{equation}\label{E_formula2}
\tilde Q(\bar t,\bar x)= \int_\alpha^{\bar y}[\rho]_0 (x)\theta_0(x)dx + \int_{E^-_{\bar t} \cap ]0,t_*(\gamma_{\bar t,\bar x})[} \mathrm{Tr}[\rho b](t,\alpha^+) \bar \theta(t)dt. 
\end{equation}
{\sc Case 3:} $x_\beta(\bar t) < \bar x< \beta$. We set 
\begin{equation*}
\underline t_{\min}:=\min E^+_{\bar t}, \qquad \qquad \qquad 
\underline y :=\begin{cases}
\beta & \text{if }\underline t_{\min} =0\\
\sup_{\varepsilon>0}\gamma_{\underline t_{\min},\beta- \varepsilon}(0) & \text{otherwise}
\end{cases}
\end{equation*}
and
\begin{equation}\label{E_formula3}
\tilde Q(\bar t,\bar x)= \int_{\alpha}^{\underline y }[\rho]_0 (x)\theta_0(x)dx + \int_{E^+_{\bar t} \cap ]0,t_*(\gamma_{\bar t,\bar x})[} -\mathrm{Tr}[\rho b](t,\beta^-) \underline \theta(t)dt.
\end{equation}
Note that $\tilde Q$ is not defined at the points $(t,x_\alpha(t))$ and $(t,x_\beta(t))$, however Lemma~\ref{L_tildeQ} yields that $\tilde Q$ is a Lipschitz continuous function and therefore we can extend it by continuity to the whole set 
$[0,T]\times [\alpha,\beta]$. 
\begin{remark}\label{r:glicine}
We now  provide an heuristic explanation of the definition of the function $\tilde Q$. First, we recall that we want to define $\tilde Q$ in such a way that i) it is constant along the connected components of the weak characteristic curves $\gamma_{t,x}$ inside $]0, T[ \times ]\alpha, \beta[$, ii) it satisfies the equality $\partial_x \tilde Q (0, x)= [\rho]_0 \theta_0$ and the equalities $\partial_t \tilde Q (t, \alpha) = \mathrm{Tr}[\rho b](t,\alpha^+) \bar \theta$ and $\partial_t \tilde Q (t, \beta) = - \mathrm{Tr}[\rho b](t,\beta^-) \underline \theta$ where $ \mathrm{Tr}[\rho b](t,\alpha^+) < 0$ and $\mathrm{Tr}[\rho b](t,\beta^-)<0$, respectively and iii) is uniquely determined by the initial data $[\rho]_0 \theta_0$ and by the assigned boundary data. With this in mind, let us fix a point $(\bar t, \bar x)$ and backwardly follow the weak characteristic line $\gamma_{\bar t, \bar x}$ to define the value $\tilde Q(\bar t, \bar x)$. If $x_\alpha (\bar t) < \bar x < x_\beta (\bar t)$, this means that the backward weak characteristic curve $\gamma_{\bar t, \bar x}$ survives up to time $t=0$ and by recalling i) and ii) above this yields~\eqref{e:sacher}. If $\bar x < x_\alpha (\bar t)$, then we follow the weak characteristic line $\gamma_{\bar t, \bar x}$ up to the time
$t_* (\gamma_{\bar t, \bar x})$, when it collides with the boundary. We then take into account the boundary contributions up to the time $t_{\min}$ and in the meanwhile gain the second term on the right hand side of~\eqref{E_formula2}, which only depends on the boundary values attained on the set where we effectively assign them. 
At the time $t_{\min}$, heuristically speaking (the situation is actually slightly more complicated as the definition of $\bar y$ shows) we follow the backward weak characteristic curve detatching from the boundary, which crosses the $t=0$ axis at the point $\bar y$, and here we gain the first term in~\eqref{E_formula2}. 
\end{remark}
We first establish a few preliminary results. 
\begin{lemma}
\label{l:strudel}
Fix $\bar t \in ]0, T[$ and $x \in ]\alpha, \beta[$ and assume that $t_* (\gamma_{\bar t, x})>0$. If $ \gamma_{\bar t, x}( t_* (\gamma_{\bar t, x}))= \alpha$, then  $t_* (\gamma_{\bar t, x}) \in E^-_{\bar t}$. If $ \gamma_{\bar t, x}( t_* (\gamma_{\bar t, x}))= \beta$, then  $t_* (\gamma_{\bar t, x}) \in E^+_{\bar t}$. 
\end{lemma}
\begin{proof}
We only establish the implication $ \gamma_{\bar t, x}( t_* (\gamma_{\bar t , x}))= \alpha$ then  $t_* (\gamma_{\bar t, x}) \in E^-_{\bar t}$ , the proof of the other one is analogous. \\
{\sc Step 1:} we show that 
\be \label{e:ribes}
Q(s, \alpha) \leq Q(t_* (\gamma_{\bar t, x}), \alpha)
\quad \text{for every $s \in [t_* (\gamma_{\bar t, x}), \bar t].$}
\eq
 Assume by contradiction that there is $s \in [t_* (\gamma_{\bar t, x}), \bar t]$ such that $Q(s, \alpha) > Q(t_* (\gamma_{\bar t, x}, \alpha)$, then by definition of $t_* (\gamma_{\bar t, x})$ we have $\gamma_{\bar t, x} (s) \in ]\alpha, \beta[$ and owing to the fact that $Q(s, \cdot)$ is a monotone 
non-decreasing function we have $Q(s, \gamma_{\bar t, x} (s)) \ge Q(s, \alpha)  > Q(t_* (\gamma_{\bar t, x}), \alpha)$. On the other hand, owing to property (f) in the statement of Lemma~\ref{L_2} we have $Q(s, \gamma_{\bar t, x} (s)) =Q(t_* (\gamma_{\bar t, x}), \alpha)$ and this yields a contradiction. \\
{\sc Step 2:} we conclude the proof . We define the function $g: [0, \bar t] \to \R$ by setting
$$
    g(t) : = \left\{
    \begin{array}{ll}
               U^-(Q(\cdot,\alpha),[0,\bar t])(t) & t \in [0, t_* (\gamma_{\bar t, x})[ \\
                Q(t_* (\gamma_{\bar t, x}), \alpha) & t \in [t_* (\gamma_{\bar t, x}), \bar t],
    \end{array}
   \right.
$$
 then $g$ is a monotone non-increasing function and, owing to~\eqref{e:ribes}, $g \ge f$. This implies that $g\ge U^-(Q(\cdot,\alpha),[0,\bar t])$ and hence that $U^-(Q(\cdot,\alpha),[0,\bar t])( t_* (\gamma_{\bar t, x}))= Q(\alpha,  t_* (\gamma_{\bar t, x}))$.  
\end{proof}
\begin{lemma}\label{l:cipria} 
For every $x \in ]\alpha,\beta[$ the map
$t \mapsto \tilde Q(t,\gamma_{\bar t,x}(t))$ is constant on $]t_*(\gamma_{\bar t,x}), t^*(\gamma_{\bar t,x})[$.
\end{lemma}
\begin{proof}
 We recall Lemma~\ref{l:strudel} and apply Lemma~\ref{l:cialda} with $\tau =t_*=t_*(\gamma_{\bar t,\bar x})$, $a=0$ and $b= \bar t$ to get 
\be \label{e:mirtillo}
E^-_{\bar t} \cap [0,t_*(\gamma_{\bar t,\bar x})]= E^-_{t_*(\gamma_{\bar t,\bar x})}.
\eq
This implies that $t_{\min}$, $\underline t_{\min}$ and $\bar y$, $\underline y$ do not 
change when $t$ varies in $]t_*(\gamma_{\bar t,x}), t^*(\gamma_{\bar t,x})[$ and hence implies that $t~\mapsto~\tilde Q(t,\gamma_{\bar t,x}(t))$ is constant .
\end{proof}
\begin{figure}
\centering 
\includegraphics[width=0.6\columnwidth]{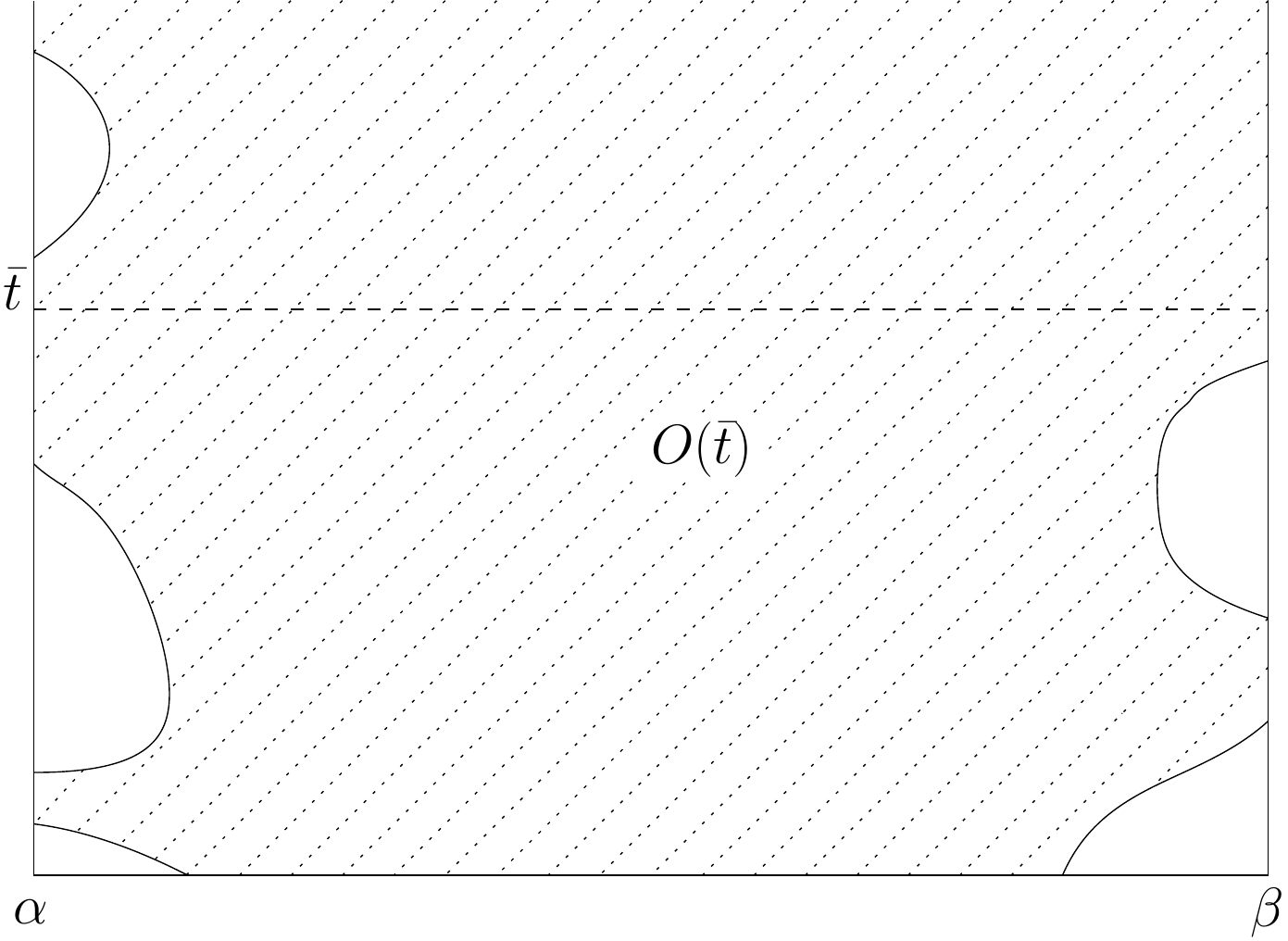}
\label{f:obart}
\caption{The set $O (\bar t)$ defined by~\eqref{e:obart}}
\end{figure}
\begin{lemma}\label{L_tildeQ}
For every $\bar t \in ]0, T[$ we set
\begin{equation} \label{e:obart}
O(\bar t):=\{(t,x) \in ]0,T[\times ]\alpha,\beta[: t_*(\gamma_{t,x})< \bar t < t^*(\gamma_{t,x})\},
\end{equation}
see Figure~\ref{f:obart}. Then there is a Lipschitz continuous function $g_{\bar t}:\R \to \R$ such that 
\begin{equation}\label{E_gbart}
\tilde Q (t,x)= g_{\bar t}(Q(t,x)) \qquad \mbox{ for every } (t,x) \in O(\bar t).
\end{equation}
Also, 
\begin{enumerate}
\item[1.] for a.e. $h \in Q(\bar t, ]\alpha,x_\alpha(\bar t)[)$ there is a unique $x \in ]\alpha,x_\alpha(\bar t)[$ such that 
$Q(\bar t,x)=h$ and 
\begin{equation}\label{E_dergalpha}
g'_{\bar t}(h)= \bar \theta (t_*(\gamma_{\bar t,x})).
\end{equation}
\item[2.] For a.e. $h \in Q(\bar t, ]x_\alpha(\bar t),x_\beta(\bar t)[)$ there is a unique $x \in ]x_\alpha(\bar t), x_\beta(\bar t)[$ such that 
$Q(\bar t,x)=h$ and 
\begin{equation}\label{E_derg_0}
g'_{\bar t}(h)= \theta_0 (\gamma_{\bar t,x}(0)).
\end{equation}
\item[3.] For a.e. $h \in Q(\bar t, ]x_\beta(\bar t), \beta[)$ there is a unique $x \in ]x_\beta(\bar t),\beta[$ such that 
$Q(\bar t,x)=h$ and 
\begin{equation}\label{E_dergbeta}
g'_{\bar t}(h)= \underline \theta (t_*(\gamma_{\bar t,x})).
\end{equation}
\end{enumerate}
\end{lemma}
\begin{proof}
We separately define the function $g_{\bar t}$ on the sets $Q(\bar t, ]\alpha, x_\alpha(\bar t)[)$, $Q(\bar t, ]x_\alpha(\bar t), x_\beta (\bar t)[)$ and $Q(\bar t, ]x_\beta(\bar t), \beta[)$ and then patch together the various definitions. The proof is organized into several steps. \\
{\sc Step 1:} we define the function $g_{\bar t}$ on the set $Q(\bar t, ]x_\alpha(\bar t), x_\beta (\bar t)[)$. First, we point out that, if $x_1,x_2 \in ]x_\alpha(\bar t),x_\beta(\bar t)[$, then
\inputencoding{latin1}
\begin{equation*}
\begin{split}
|\tilde Q(\bar t,x_2) - \tilde Q(\bar t,x_1)| \stackrel{\eqref{e:sacher}}{=} &~  \left| \int_0^{\gamma_{\bar t, x_2}(0)} [\rho]_0 \theta_0 (x) dx -  \int_0^{\gamma_{\bar t, x_1}(0)} [\rho]_0 \theta_0 (x) dx \right| \le \int_{\gamma_{\bar t,x_1}(0)}^{\gamma_{\bar t,x_2}(0)} [\rho]_0 |\theta_0|(x) dx \\
\le &~ \|\theta_0\|_{L^\infty} \int_{\gamma_{\bar t,x_1}(0)}^{\gamma_{\bar t,x_2}(0)} [\rho]_0 (x) dx =  \|\theta_0\|_{L^\infty} | Q(\bar t,x_2) - Q(\bar t,x_1)|.
\end{split}
\end{equation*}
To establish the last inequality at the first line we have used the fact that, since $\rho \ge 0$, then $[\rho]_0 \ge0$.  By relying on Lemma~\ref{l:pastiera} we conclude that there is a Lipschitz continuous function $g_{\bar t,0}: Q(\bar t, ]x_\alpha(\bar t),x_\beta(\bar t)[)\to \R$ such that $\tilde Q(\bar t,x)= g_{\bar t,0}(Q(\bar t,x))$,  
for every $x \in ]x_\alpha(\bar t),x_\beta(\bar t)[$.

We now establish \eqref{E_derg_0}. First, we recall property (f) in Lemma~\ref{L_2} and the definition~\eqref{e:gelato} of $x_\alpha(\bar t)$ and $x_\beta (\bar t)$ and we conclude that $Q(\bar t,]x_\alpha(\bar t),x_\beta(\bar t)[) \subseteq Q(0, ]\alpha, \beta[)$.  We denote by $\mathcal R_0$ the set of values $h \in Q(\bar t,]x_\alpha(\bar t),x_\beta(\bar t)[)$ such that i) there is a unique $x_h \in ]x_\alpha(\bar t),x_\beta(\bar t)[$ such that $Q (\bar t, x_h)=h$; ii)
there is a unique $z_h \in ]\alpha,\beta[$ such that $Q(0,z_h)=h$ and furthermore iii) $z_h$ is a Lebesgue point of $\theta_0$ with respect to the measure $[\rho]_0 \Leb^1$, namely $z_h \in \mathrm{supp} \left([\rho]_0 \Leb^1\right)$ and 
\begin{equation}\label{E_Lebesgue_0}
\lim_{\varepsilon \to 0^+}\frac{\displaystyle{\int_{z_h-\varepsilon}^{z_h+\varepsilon}[\rho]_0|\theta_0(y)-\theta_0(z_h)|\,dy}}{\displaystyle{\int_{z_h-\varepsilon}^{z_h+\varepsilon}[\rho]_0(y)\,dy}}=0.
\end{equation}
Note that in the previous expression the integral at the denominator does not vanish since $z_h \in \mathrm{supp} \left([\rho]_0 \Leb^1\right)$. Note furthermore that, owing to the monotonicity of $Q$ with respect to $x$, if $x_h$ and $z_h$ are as in points i) and ii) then $z_h = \gamma_{\bar t, x_h} (0)$.
Assume for a moment that we have shown that $\Leb^1(Q(\bar t,]x_\alpha(\bar t),x_\beta(\bar t)[)\setminus \mathcal R_0)=0$.
Fix $h \in \mathcal R_0$ and the corresponding $z_h\in ]\alpha,\beta[$ such that $Q(0,z_h)=h$: owing to  \eqref{E_Lebesgue_0} we have 
\begin{equation*}
g_{\bar t,0}'(h) = \lim_{\varepsilon \to 0^+}\frac{\tilde Q(0,z_h +\varepsilon) -\tilde Q(0,z_h -\varepsilon) }{Q(0,z_h+\varepsilon) - Q(0,z_h-\varepsilon)} = \theta_0(z_h) = \theta_0 (\gamma_{\bar t, x_h} (0)),
\end{equation*}
that is~\eqref{E_derg_0}. 

We are left to show that $\Leb^1(Q(\bar t,]x_\alpha(\bar t),x_\beta(\bar t)[)\setminus \mathcal R_0)=0$. Note that 
$
   \mathcal R_0 = \mathcal R_1 \cap \mathcal R_2 \cap \mathcal R_3,
$
where $\mathcal R_1$ is the set of $h$ such that there is a unique $x \in ]x_\alpha(\bar t),x_\beta(\bar t)[$ satisfying $Q (\bar t, x)=h$, $\mathcal R_2$
 is the set of $h$ such that there is a unique $z \in ]x_\alpha(\bar t),x_\beta(\bar t)[$ satisfying $Q (0, z)=h$ and 
$$
    \mathcal R_3  := Q(\bar t; ]x_\alpha(\bar t),x_\beta(\bar t)[) \cap  
     Q(0, L), \quad \text{where} \;
     L:= \{ z \; \text{is a Lebesgue point for $[\rho]_0 \Leb^1$ on $ ]\alpha, \beta[$}  
     \}. 
$$
To show that $\Leb^1(Q(\bar t,]x_\alpha(\bar t),x_\beta(\bar t)[)\setminus \mathcal R_0)=0$ it suffices to show that $\Leb^1(Q(\bar t,]x_\alpha(\bar t),x_\beta(\bar t)[)\setminus \mathcal R_i)=0$ for $i=1, 2, 3$. Since $Q(\bar t, \cdot)$ and $Q(0, \cdot)$ are both monotone functions, then $\Leb^1(Q(\bar t,]x_\alpha(\bar t),x_\beta(\bar t)[)\setminus \mathcal R_1)=0$ and $\Leb^1(Q(\bar t,]x_\alpha(\bar t),x_\beta(\bar t)[)\setminus \mathcal R_2)=0$. To establish the last inequality, we point out that $L$ is a $\Leb^1$-measurable set and that, owing to the Lebesgue Differentiation Theorem, $[\rho]_0 \Leb^1 ( ]\alpha, \beta[ \setminus L)=0$. Owing to the Coarea Formula we get 
$$
    0= \int_{]\alpha, \beta[ \setminus L} [\rho]_0 (x) dx= \int_{Q(0, ]\alpha, \beta[ \setminus L)} \mathcal H^0 (\{x: \;  Q(0, x) =h \} ) d h ,
$$
which yields $\Leb^1 (Q(0, ]\alpha, \beta[ \setminus L))=0$. To conclude, we recall that $Q(\bar t, ]x_\alpha (\bar t), x_\beta (\bar t)[) \subseteq Q(0, ]\alpha, \beta[)$, which implies $Q(\bar t, ]x_\alpha (\bar t), x_\beta (\bar t)[) \setminus \mathcal R_3 \subseteq Q(0, ]\alpha, \beta[) \setminus Q(0, L) \subseteq Q(0, ]\alpha, \beta[ \setminus L)$. \\
{\sc Step 2:} we define the function $g_{\bar t}$ on the 
set $]\alpha,x_\alpha(\bar t)[$. As in {\sc Step 1} we want to rely on Lemma~\ref{l:pastiera}, so we fix $x_1,x_2 \in ]\alpha,x_\alpha(\bar t)[$, $x_1 < x_2$. 
By the monotonicity property at point (h) in Lemma \ref{L_2}, we have $t_*(\gamma_{\bar t,x_2})\le  t_*(\gamma_{\bar t,x_1})$. Also, the point $\bar y$ defined in~\eqref{e:limone}  is the same for both $x_1$ and $x_2$. This yields
\begin{equation}\label{E_x_12alpha}
\begin{split}
&|\tilde Q(\bar t,x_2) - \tilde Q(\bar t,x_1)| = ~ \left| \int_{E^-_{\bar t} \cap [t_*(\gamma_{\bar t,x_2}), t_*(\gamma_{\bar t,x_1})] } \mathrm{Tr}[\rho b](t,\alpha^+) \bar \theta(t) dt\right|\\
&\stackrel{\eqref{E_trace_boundary}}{=} \left|\int_{E^-_{\bar t}\cap [t_*(\gamma_{\bar t,x_2}), t_*(\gamma_{\bar t,x_1})]} \partial_t Q(t,\alpha) \bar \theta(t)  dt \right| \\
\stackrel{\eqref{E_derivatives_envelope}}{=} & \left|\int_{E^-_{\bar t}\cap [t_*(\gamma_{\bar t,x_2}), t_*(\gamma_{\bar t,x_1})]} \partial_t
U^-(t) \bar \theta(t) dt\right| \leq  \|\bar \theta\|_{L^\infty}
 \int_{E^-_{\bar t}\cap [t_*(\gamma_{\bar t,x_2}), t_*(\gamma_{\bar t,x_1})]} - \partial_t
U^-(t) dt \\
  & = 
 \|\bar \theta\|_{L^\infty} \left|U^-(t_*(\gamma_{\bar t,x_2}))- U^-(t_*(\gamma_{\bar t,x_1}))\right| 
\stackrel{\text{Lemma~\ref{l:strudel}}}{=} ~ \|\bar \theta\|_{L^\infty} | Q(\bar t,x_2) -  Q(\bar t,x_1)|.
\end{split}
\end{equation}
In the previous expression we have written $U^-$ instead of $U^-(Q(\alpha,\cdot),  [0, \bar t])$ to simplify the notation. Owing to Lemma~\ref{l:pastiera}, this implies that one can define
 a Lipschitz continuous function $g_{\bar t,\alpha}: Q(\bar t, ]\alpha,x_\alpha(\bar t)[)\to~\R$ such that 
for every $x \in ]\alpha, x_\alpha(\bar t)[$ we have $\tilde Q(\bar t,x)= g_{\bar t,\alpha}(Q(\bar t,x))$. To establish~\eqref{E_dergalpha} we first recall that, owing to property ({f}) in 
Lemma~\ref{L_2}, $Q(t_*(\gamma_{\bar t,x}), \alpha) = Q(\bar t, x)$ for every $x \in ]\alpha, x_\alpha (\bar t)[$. Owing to Lemma~\ref{l:strudel}, this implies that $Q(\bar t, ]\alpha, x_\alpha (\bar t)[) \subseteq U^- (Q(\cdot, \alpha), [0, \bar t]) ([0, \bar t] )$. We now term $\mathcal R_\alpha(\bar t)$ the set of values $h \in Q(\bar t,]\alpha, x_\alpha(\bar t)[)$ such that i) there is a unique $x_h \in ]\alpha, x_\alpha (\bar t)[$ satisfying $Q(\bar t, x_h) = h$; ii) there is a unique $t_h \in ]0, \bar t[$ satisfying $U^- (Q(\cdot, \alpha), [0, \bar t]) (t_h)=h$ and furthermore iii) $t_h$ is a Lebesgue point for the measure $ - d U^- (Q(\cdot, \alpha), [0, \bar t])/ dt \Leb^1$.
Note that if $t_h$ is as in ii), then owing to Lemma~\ref{L_envelope} we have $t_h \in E^-_{\bar t}$ and hence by~\eqref{E_trace_boundary} point iii) could be equivalently riformulated by requiring that $t_h$ is a Lebesgue point for the measure $- \mathrm{Tr}[b \rho] (\cdot, \alpha^+) \Leb^1 \llcorner E^-_{\bar t}$. Also, if $x_h$ and $t_h$ are in points i) and ii), then $t_h = t_\ast (\gamma_{\bar t, x_h})$. By arguing as in {\sc Step 1} one  can show that $\mathcal \Leb^1\left(  Q(\bar t,]\alpha, x_\alpha(\bar t)[) \setminus \mathcal R_\alpha \right)=0$ and that for every $h \in \mathcal R_\alpha$ we have
 \begin{equation*}
 g'_{\bar t,\alpha}(h) = \lim_{\varepsilon \to 0}\frac{\tilde Q(t_h+\varepsilon,\alpha) -\tilde Q(t_h-\varepsilon,\alpha) }{ Q(t_h+\varepsilon,\alpha) - Q(t_h-\varepsilon,\alpha)} =\bar \theta(t_h)  = \bar \theta(t_*(\gamma_{\bar t,x_h})).
 \end{equation*}
By relying on the same argument one can show that there is a Lipschitz continuous function $g_{\bar t,\beta}: Q(\bar t, ]x_\beta(\bar t),\beta[)\to \R$ such that 
$\tilde Q(\bar t,x)= g_{\bar t,\beta}(Q(t,x))$  
for every $x \in ]x_\beta(\bar t),\beta[$. Also, 
 \begin{equation*}
 g'_{\bar t,\beta}(h) = \underline \theta(t_*(\gamma_{\bar t,x})), \quad \text{for a.e. $h \in Q(\bar t, ]x_\beta(\bar t),\beta[)$},
 \end{equation*}
 provided $x \in ]x_\beta(\bar t),\beta[$ satisfies $\tilde Q(t,x)=h$. \\
{\sc Step 3:} we patch together the definitions of $g_{\bar t}$ given at the previous steps. First, we recall that the function $Q(\bar t, \cdot)$ is Lipschitz continuous and monotone non-increasing, which implies 
\be \label{e:frappe} \begin{split}
  &  \sup  Q(\bar t, ]\alpha, x_\alpha(\bar t)[) =
    \lim_{x \to x_\alpha(\bar t)^-}  Q(\bar t, x) =
   \lim_{x \to x_\alpha(\bar t)^+}  Q(\bar t, x)=
    \inf  Q(\bar t, ]x_\alpha(\bar t), x_\beta (\bar t)[), 
   \\ &
     \sup  Q(\bar t, ]x_\alpha(\bar t), x_\beta (\bar t)[)=
     \lim_{x \to x_\beta(\bar t)^-}  Q(\bar t, x) =
   \lim_{x \to x_\beta(\bar t)^+}  Q(\bar t, x)=
    \inf  Q(\bar t, ]x_\beta (\bar t), \beta[)
\end{split}
\eq  
 if both $x_\alpha (\bar t), x_\beta (\bar t) \in ]\alpha, \beta[$. 
We set
\begin{equation*}
g_{\bar t}(h) := 
\begin{cases}
 g_{\bar t,\alpha}(h) &\mbox{if } h \in \overline{ Q(\bar t, ]\alpha, x_\alpha(\bar t)[)},\\
 g_{\bar t,0}(h) &\mbox{if } h \in \overline{ Q(\bar t, ]x_\alpha(\bar t),x_\beta(\bar t)[)},\\
 g_{\bar t,\beta}(h) &\mbox{if } h \in \overline{ Q(\bar t,]x_\beta(\bar t),\beta[)}.
\end{cases}
\end{equation*}
Owing to~\eqref{e:frappe}, in order to prove that $g_{\bar t}$ is well-defined it suffices to show that, if $x_\alpha (\bar t), x_\beta (\bar t) \in ]\alpha, \beta[$ (the other cases are actually simpler), then $\tilde Q(\bar t,\cdot)$ is continuous at both $x_\alpha(\bar t)$ and $x_\beta(\bar t)$. We only establish the continuity at $x_\alpha(\bar t)$, the other proof is analogous. We assume that $x_\alpha(\bar t) \in ]\alpha,\beta[$ and define the two curves $\tilde \gamma_\alpha^\pm:[0,T] \to [\alpha,\beta]$ by setting
\begin{equation*}
\tilde\gamma^+_\alpha (t) : = \lim_{\varepsilon \to 0^+} \gamma_{\bar t,x_\alpha(\bar t)+\varepsilon}(t) \qquad \mbox{and} \qquad
\tilde\gamma^-_\alpha (t) : = \lim_{\varepsilon \to 0^+} \gamma_{\bar t,x_\alpha(\bar t)-\varepsilon}(t).
\end{equation*}
We also set 
\begin{equation} \label{e:bluenavy}
\tilde t_*:= \lim_{\varepsilon \to 0^+} t_*(\gamma_{\bar t,x_\alpha(\bar t)-\varepsilon}).
\end{equation}
Owing to Lemma~\ref{l:strudel}, $t_*(\gamma_{\bar t,x}) \in E^-_{\bar t}$ for every $x< x_\alpha (\bar t)$. Since by Lemma~\ref{L_envelope}  $E^-_t$ is a closed set, then  
$\tilde t_* \in E^-_{\bar t}$ and this implies  $\tilde t_* \ge t_{\min}$, where $t_{\min}$ is the same as in~\eqref{e:limone}. 

By definition of $\tilde Q$ we have
\begin{equation}\label{E_tildeQxalpha}
\begin{split}
\lim_{x \to x_\alpha(\bar t)^-}\tilde Q(\bar t, x)
= \tilde Q(\bar t, x_\alpha(\bar t)^-)=&~  \int_\alpha^{\bar y}[\rho]_0 (x)\theta_0(x)dx + \int_{E^-_{\bar t} \cap ]0,\tilde t_*[} \mathrm{Tr}[\rho b](t,\alpha^+) \bar \theta(t)dt, \\
\lim_{x \to x_\alpha(\bar t)^+} \tilde Q(\bar t, x) = \tilde Q(\bar t, x_\alpha(\bar t)^+)=&~\int_\alpha^{\tilde \gamma^+_\alpha(0)}[\rho]_0 (x)\theta_0(x)dx 
\end{split}
\end{equation}
In {\sc Step 4} below we show that 
$Q (\tilde t_*,\alpha) = Q(t_{\min},\alpha)$. 
Since $\tilde t_*, t_{\min} \in E^-_{\bar t}$, this yields 
$U^-(Q(\cdot, \alpha),[0,\bar t])(\tilde t_*)=U^-(Q(\cdot, \alpha),[0,\bar t])(t_{\min})$. We then have 
\begin{equation*}
\begin{split}
 \mathrm{Tr}[\rho b](\cdot, \alpha^+) \llcorner \big(E^-_{\bar t}\cap ]0,\tilde t_*[\big) \stackrel{\text{Lemma~\ref{l:crema}}}{=} &~ - \partial_t Q(t,\alpha) \llcorner \big(E^-_{\bar t}\cap ]0 , \tilde t_*[\big) \\
 \stackrel{\text{Lemma~\ref{L_envelope}}}{=} &~ - \partial_t U^-(Q(\cdot, \alpha),[0,\bar t]) \llcorner \big(E^-_{\bar t}\cap ]0,\tilde t_*[\big) = 0
 \end{split}
\end{equation*}
and this implies $\int_{E^-_{\bar t} \cap ]0,\tilde t_*[} \mathrm{Tr}[\rho b](t,\alpha^+) \bar \theta(t)dt=0$. Owing to~\eqref{E_tildeQxalpha}, in order to show that 
$\tilde Q(\bar t, x_\alpha(\bar t)^-)=\tilde Q(\bar t, x_\alpha(\bar t)^+)$ we are left to prove that 
$$ \int_\alpha^{\bar y}[\rho]_0 (x)\theta_0(x)dx = \int_\alpha^{\tilde \gamma^+_\alpha(0)}[\rho]_0 (x)\theta_0(x)dx.$$ 
It suffices to show that $[\rho]_0 =0$ a.e. on the interval with endpoints $\bar y$ and $\tilde  \gamma^+_\alpha(0)$, which is equivalent to prove that $Q(0, \bar y)=Q(0, \tilde \gamma^+_\alpha(0))$. To see this, we first show that 
\be \label{e:sorbetto}
Q(t,\alpha)< Q(t_{\min}, \alpha), \quad 
\text{for every $t \in ]0, t_{\min}[$}. 
\eq
Assume by contradiction that $t_{\min}>0$ and there is 
$\tau \in ]0, t_{\min}[$ such that $Q(\tau,\alpha)\ge Q(t_{\min}, \alpha)$. Since by definition $t_{\min}= \min E^-_{\bar t}$, then $U^-(Q(\cdot, \alpha),[0,\bar t]) (t) = U^-(Q(\cdot, \alpha),[0,\bar t])(t_{\min})= Q(t_{\min}, \alpha)$ for every $t \in ]0, t_{\min}[$. Since $U^-(Q(\cdot, \alpha),[0,\bar t]) \ge 
Q(\cdot,\alpha)$, this implies $U^-(Q(\cdot, \alpha),[0,\bar t]) (\tau)=
Q(\tau,\alpha)$, which contradicts the minimality of $ t_{\min}$ and hence establishes~\eqref{e:sorbetto}. 
We now recall property  {(f)} in Lemma~\ref{L_2} and that, for every $t \in [0, T]$ the function $Q(t, \cdot)$ is monotone non-decreasing. Owing to~\eqref{e:sorbetto}, this implies that $\gamma_{t_{\min},\alpha +\varepsilon} (t) >\alpha$ for every $\varepsilon>0$  and every $ t \in ]0,t_{\min}[$.
By Point {(f)} in Lemma \ref{L_2} this yields $Q(0,\gamma_{t_{\min},\alpha +\varepsilon} (0))=Q(t_{\min},\alpha+\varepsilon)$ and by letting
$\varepsilon \to 0$ we arrive at
\begin{equation}\label{E_equality1}
Q(0,\bar y)=Q(t_{\min},\alpha).
\end{equation}
On the other hand, $ Q(\bar t,x_\alpha(\bar t)) = Q(0,\tilde \gamma^+_\alpha(0))=Q(\tilde t_*,\alpha)$. Since 
$Q(\tilde t_*,\alpha) = Q(t_{\min},\alpha)$ by {\sc Step 4}, then owing to \eqref{E_equality1} we have $Q(0, \bar y)=Q(0, \tilde \gamma^+_\alpha(0))$.
This concludes the proof of the continuity of $\tilde Q(\bar t,\cdot)$ at the point $x_\alpha(\bar t)$.
\\
{\sc Step 4:} we establish the equality $Q(\tilde t_*,\alpha) = Q(t_{\min},\alpha)$. Since $t_{\min},\tilde t_* \in E^-_{\bar t}$ and $\tilde t_* \ge t_{\min}$, then 
\begin{equation*}
Q (\tilde t_*,\alpha) = U^-(Q(\cdot, \alpha),[0,\bar t])(\tilde t_*) \le  U^-(Q(\cdot, \alpha),[0,\bar t])(t_{\min})= Q(t_{\min},\alpha). 
\end{equation*}
Assume by contradiction that  $Q (\tilde t_*,\alpha) < Q(t_{\min},\alpha)$. 
Since $U^-(Q(\cdot, \alpha),[0,\bar t])$ is a monotone non-decreasing function, there is $h \in ]Q (\tilde t_*,\alpha), Q(t_{\min},\alpha)[$ such that 
there is a unique $t_h \in ]t_{\min}, \tilde t_*[$ satisfying $U^-(Q(\cdot, \alpha),[0,\bar t])(t_h)=h$. Note that $t_h \in E^-_{\bar t}$ and that for 
every $t \in ]t_h,\bar t[$ we have $Q(t,\alpha)\leq 
U^-(Q(\cdot, \alpha),[0,\bar t])(t) < U^-(Q(\cdot, \alpha),[0,\bar t])(t_h)= Q(t_h, \alpha)=h$. By using property {(f)} in Lemma~\ref{L_2} and the monotonicity of $Q(t, \cdot)$ this implies that,  for every $\varepsilon>0$ and every $t \in [t_h,\bar t]$, we have $\gamma_{t_{h},\alpha + \varepsilon}(t)>\alpha$ and this in turn implies 
$$
 Q(\bar t,\gamma_{t_h,\alpha+\varepsilon}(\bar t))= Q(t_h,\alpha+\varepsilon)\ge h >  Q (\tilde t_*, \alpha) =
Q(\bar t, x_\alpha(\bar t)).
$$
 By the monotonicity of $Q(\bar t,\cdot)$, this yields  $x_\alpha(\bar t)< x_h:=\lim_{\varepsilon \to 0^+}\gamma_{t_h,\alpha + \varepsilon}(\bar t)$.
Note that,  by the definition of $x_{\alpha}(\bar t)$, for every $ x \in ]x_\alpha(\bar t),x_h[$ we have $t_*(\gamma_{\bar t,x})=0$. On the other hand, by the monotonicity property  {(h)} in Lemma \ref{L_2}, for every  $ x \in ]x_\alpha(\bar t),x_h[$ and every $\varepsilon>0$ we have $\gamma_{\bar t,x} (\cdot) \leq  \gamma_{t_h,\alpha + \varepsilon} (\cdot)$ and this implies 
$t_*(\gamma_{\bar t,x})\ge t_h>0$. This yields a  contradiction and establishes the equality
$Q (\tilde t_*,\alpha) = Q(t_{\min},\alpha)$. \\
{\sc Step 5:} we are left to show that \eqref{E_gbart} holds for every $(t,x) \in O(\bar t)$. So far we have shown that \eqref{E_gbart} holds for every $(t,x) \in ]0,T[\times ]\alpha,\beta[$ with $t=\bar t$. To conclude we recall that $Q$ is constant along $\gamma_{\bar t,x}$ owing to property {(f)} in Lemma~\ref{L_2} and we recall Lemma~\ref{l:cipria}. 
\end{proof}
\begin{remark}
It follows immediately from the statement of Lemma \ref{L_tildeQ} that for every $t_1,t_2 \in ]0,T[$ it holds $g_{t_1}=g_{t_2}$ in $Q(O(t_1)) \cap Q(O(t_2))$.
\end{remark}
\subsection{Proof of Theorem~\ref{t:genex}: existence} \label{ss:indaco}
The following result establishes the existence of a solution of~\eqref{IVPtheta2}
by exhibiting an explicit formula. Note that by relying on~\eqref{e:golfino} below we conclude that the function  $\tilde Q$ constructed in \S\ref{ss:giallo} is a potential function for $(\rho \theta, b \rho \theta)$. 
\begin{proposition}\label{P_existence}
Let $\tilde Q$ be as in \S\ref{ss:giallo} and $g_{ t}$ be as in Lemma~\ref{L_tildeQ}. 
Then the function
\begin{equation}\label{E_def_theta}
\theta(t,x):= \begin{cases}
g'_{t} (Q(t,x)) &\mbox{if } g_{t} \mbox{ is differentiable at }Q(t,x), \\
0 & \mbox{otherwise}.
\end{cases}
\end{equation}
is a solution of the initial-boundary value problem~\eqref{IVPtheta2} in the sense of Definition~\ref{d:ibvp}.
Also, 
\begin{equation}\label{E_theta_characteristics}
\theta(t,x) = 
\begin{cases}
\bar \theta( t_*(\gamma_{t,x})) & \mbox{if } \alpha<x<x_{\alpha}(t) \\
\theta_0(\gamma_{t,x}(0)) & \mbox{if } x_\alpha(t)<x<x_{\beta}(t)\\
\underline \theta( t_*(\gamma_{t,x})) & \mbox{if } x_\beta(t)<x<\beta
\end{cases}
\end{equation} 
for $\rho \Leb^2$-a.e. $(t,x) \in ]0,T[ \times ]\alpha,\beta[$.
\end{proposition}

\begin{proof} We organize the proof in four steps. \\
{\sc Step 1:} we establish~\eqref{e:suppin2}, namely we show that the function $\theta$ defined by~\eqref{E_def_theta} is a distributional solution of the equation on $]0, T[ \times ]\alpha, \beta[.$
By a standard argument involving partitions of unity, it suffices to show that, for every $(\bar t,\bar x) \in ]0,T[\times ]\alpha,\beta[$, the equation $\partial_t [\rho \theta] + \partial_x [b \rho \theta]=0$ holds in the sense of distributions in a neighborhood of $(\bar t,\bar x)$. To this end, fix $(\bar t,\bar x) \in ]0,T[\times ]\alpha,\beta[$ and $r>0$ such that the ball $B_r(\bar t,\bar x) \subseteq O(\bar t)$. Owing to Lemma \ref{L_tildeQ} we have $\tilde Q (t,x)= g_{\bar t} \circ Q(t,x)$ for every
$(t,x) \in  B_r(\bar t,\bar x)$. In particular, by the chain-rule for Lipschitz functions and definition~\eqref{E_def_theta} of $\theta$, we have
\begin{equation} \label{e:golfino}
\partial_x\tilde Q = \theta \partial_x Q= \rho \theta, \qquad  \mbox{and} \qquad \partial_t \tilde Q = \theta \partial_t Q= - \rho b \theta.
\end{equation}
and this yields $\partial_t [\rho \theta] + \partial_x [\rho b \theta] =0$ on $B_r(\bar t,\bar x)$. \\
{\sc Step 2:} we show that $[\rho \theta]_0=[\rho]_0 \theta_0$. First, we point out that~\eqref{e:golfino} means that $\tilde Q$ is a potential function for $(\rho \theta, b\rho \theta)$ and we recall~\eqref{E_trace_initial}. We conclude that establishing the equality 
 $[\rho \theta]_0=[\rho]_0 \theta_0$ amounts to show that $\partial_x \tilde Q(0,\cdot)=[\rho]_0 \theta_0$. To this end, we fix $x_1<x_2$ in $]\alpha,\beta[$, then  $x_1,x_2 \in ]x_\alpha(\varepsilon),x_\beta(\varepsilon)[$ provided $\ee>0$ is sufficiently small. This yields 
\begin{equation*}
\begin{split}
\tilde Q(0,x_2)-\tilde Q(0,x_1) = &~ \lim_{\varepsilon \to 0^+} \tilde Q(\varepsilon,x_2)-\tilde Q(\varepsilon,x_1)\\
= & \lim_{\varepsilon \to 0^+} \int_0^{\gamma_{\varepsilon,x_2}(0)} [\rho]_0 \theta_0(x)dx -  \int_0^{\gamma_{\varepsilon,x_1}(0)} [\rho]_0\theta_0(x)dx
= \int_{x_1}^{x_2}[\rho]_0 \theta_0(x)dx
\end{split}
\end{equation*}
and concludes {\sc Step 2.} \\
{\sc Step 3:} we show that  $\mathrm{Tr}[\rho b \theta] (t,\alpha^+) = \mathrm{Tr}[\rho b](t,\alpha^+)\bar \theta(t)$ for a.e. $t \in ]0,T[$ such that $\mathrm{Tr}[\rho b](t,\alpha^+)<0$.  We combine~\eqref{e:golfino} with~\eqref{E_trace_initial} and we conclude that 
$\mathrm{Tr}[\rho b \theta] (t,\alpha^+) = \partial_t \tilde Q(t,\alpha)$.
The computation in \eqref{E_x_12alpha} shows that, if $t_1,t_2$ are such that there are $\bar t> t_1,t_2$ and $x_1,x_2 \in ]\alpha,x_\alpha(\bar t)[$ satisfying
$t_1= t_*(\gamma_{\bar t,x_1})$ and $t_2= t_*(\gamma_{\bar t,x_2})$, then 
\begin{equation*}
\tilde Q(t_1,\alpha)-\tilde Q(t_2,\alpha)= \int_{t_2}^{t_1}\mathrm{Tr}[\rho b](t,\alpha^+)\bar \theta(t) dt.
\end{equation*} 
We set $F: = \{ t  \in ]0, T[: \; \mathrm{Tr}[\rho b](\tau,\alpha^+)<0 \}$ and we point out that, owing to the above equality, in order to complete the proof it suffices to show that i) for a.e. $ \tau \in F$  there are $\bar t>\tau$ and $x \in ]\alpha,x_\alpha(\bar t)[$ such that $t_*(\gamma_{\bar t,x})=\tau$ and ii)  the set $E^-_{\bar t} \cap F$
has density 1 at $\tau$.
We first prove that for a.e. $\tau \in F$ there is $\bar t \in \mathbb{Q} \cap ]0,T[$ as in i):
indeed for a.e. $\tau \in F$, the function $Q(\cdot, \alpha)$ is differentiable at $\tau$ and furthermore $\partial_t Q(\tau, \alpha) <0$.
In particular there is $\ee >0$ such that $Q(t,\alpha)< Q(\tau,\alpha)$ for a.e. $t \in ]\tau, \tau + \ee[$.
Let $\bar t \in \mathbb{Q}\cap ]\tau, \tau + \ee[$ and $x : = \lim_{\nu \to \alpha^+} \gamma_{\tau, \nu} (\bar t)$. We are left to show that $\tau = t_\ast (\gamma_{\bar t, x})$. By the monotonicity property stated in point (h) of Lemma~\ref{L_2} we have $ t_\ast (\gamma_{\bar t, x}) \ge \tau$. On the other hand, owing to property (f) in Lemma~\ref{L_2}, we have $Q (t, \gamma_{\bar t , x} (t))= Q(\tau, \alpha)$. Since $Q(t, \alpha) < Q(\tau, \alpha)$ for every $t \in ]\tau, \bar t]$, this implies that $ t_\ast (\gamma_{\bar t, x}) \leq \tau$ and concludes the proof of i).
Since a.e. $\tau \in F$ belongs to the countable union $\bigcup_{\bar t \in \mathbb{Q}\cap]0,T[}E^-_{\bar t} \cap F$, then a.e. $\tau \in F$ is a point of density 1 of $E^-_{\bar t}$ for at least a $\bar t \in \mathbb{Q}\cap]0,T[$ and this ends the proof of ii). \\
{\sc Step 4:} we establish \eqref{E_theta_characteristics}. We recall~\eqref{E_dergalpha}, \eqref{E_derg_0} and~\eqref{E_dergbeta}. We also fix $\bar t \in ]0, T[$ and recall the definition of the set $\mathcal R_0$ given in {\sc Step 1} of the proof of Lemma~\ref{L_tildeQ}. We term 
$\mathcal R_0^c : Q(\bar t, ]x_\alpha(\bar t), x_\beta (\bar t)[) \setminus \mathcal R_0$ and we point out that, owing to the coarea formula, 
$$
   \rho \Leb^1 
   \Big( ]x_\alpha(\bar t), x_\beta (\bar t)[ \cap Q(\bar t, \cdot)^{-1}
    (\mathcal R_0^c)  \Big)    =0.
$$
This implies that, for  $ \rho \Leb^1$-a.e. $x \in ]x_\alpha (\bar t), x_\beta (\bar t)[$, equation~\eqref{E_theta_characteristics} holds true at $t= \bar t$. By relying on analogous arguments, one can conclude the proof of~\eqref{E_theta_characteristics}. 
\end{proof}
%
\subsection{Proof of Theorem~\ref{t:genex}: uniqueness} \label{ss:lavanda}
In this paragraph we establish the uniqueness result in the statement of Theorem~\ref{t:genex}. 
To this end, we fix a solution $\theta$ of \eqref{IVPtheta2} and we recall the definition~\eqref{E_def_Q} of the potential function $Q_\theta$. 
The final goal is to establish Proposition \ref{P_formula}, which dictates that $Q_\theta=\tilde Q$, where $\tilde Q$ is the same as in \S\ref{ss:giallo}. Corollary~\ref{c:bigne} is then an easy consequence of Proposition \ref{P_formula} and gives the uniqueness result in Theorem~\ref{t:genex}. We first establish a preliminary result, which uses an argument in~\cite{Panov}.
\begin{lemma}\label{L_3}
	Let $\theta \in L^\infty (]0, T[ \times ]\alpha, \beta[)$ be a solution, in the sense of Definition~\ref{d:ibvp}, of the initial-boundary value problem~\eqref{IVPtheta2}. Fix $(\bar t,\bar x)\in ]0,T[\times ]\alpha,\beta[$, then there are $r>0$ and a Lipschitz continuous function $g_\theta: \R \to \R$ such that 
	\begin{equation}\label{E_g_theta}
	Q_\theta(t,x) = g_\theta(Q(t,x)) \qquad \forall (t,x) \in B_r(\bar t,\bar x).
	\end{equation}
	Also, 
	\begin{equation}\label{E_g_theta_curve}
	Q_\theta(t,\gamma_{\bar t,\bar x}(t))=g_\theta(Q(\bar t,\bar x)), \qquad \forall t \in [t_*(\gamma_{\bar t,\bar x}),t^*(\gamma_{\bar t,\bar x})]
	\end{equation}
	and
	\begin{equation}\label{E_trace_null}
	\mathrm{Tr}[(\rho\theta,\rho b \theta)](t,\gamma_{\bar t,\bar x}(t)^-)=\mathrm{Tr}[(\rho \theta,\rho b\theta)](t,\gamma_{\bar t,\bar x}(t)^+)=0,
      \quad \text{for a.e. $t \in [t_*(\gamma_{\bar t,\bar x}),t^*(\gamma_{\bar t,\bar x})]$.}
	\end{equation}
\end{lemma}
\begin{proof}
	We fix $(\bar t,\bar x)\in ]0,T[\times ]\alpha,\beta[$ and $r>0$ such that
	$]\bar t - r, \bar t + r[\subseteq ]0,T[$ and $I:=]\bar x - r - 2r\|b\|_{L^\infty}, \bar x + r + 2r\|b\|_{L^\infty}[ \subseteq ]\alpha,\beta[$. The rest of the proof is organized 
   into some steps. \\
   {\sc Step 1:}  we define the function $g_\theta$ on the set $Q(\bar t-r, I)$. 
   We point out that 
   \begin{equation*}
   \begin{split}
       |Q_\theta(\bar t-r, x_1) - Q_\theta(\bar t-r, x_2) | &
       \stackrel{\eqref{E_def_Q}}{=}
      \left|  \int_{x_1}^{x_2} \rho \theta (\bar t-r, x) dx  \right|
      \leq \| \theta \|_{L^\infty} 
        \left|  \int_{x_1}^{x_2} \rho  (\bar t-r, x) dx  \right|\\
     &  =  \| \theta \|_{L^\infty} 
       |Q(\bar t-r, x_1) - Q (\bar t-r, x_2) |
   \end{split}
   \end{equation*}
   and owing to Lemma~\ref{l:pastiera} this implies that there is a Lipschitz continuous function $g_\theta$ satisfying~\eqref{E_g_theta} for every $(\bar t-r, x)$ with $x \in I$. \\
      {\sc Step 2:} we prove that~\eqref{E_g_theta} holds for every $(t,x)$ in the trapezoidal region
	\begin{equation}\label{E_def_trapezio}
	E_{\bar t,\bar x}(r):= \{ (t,x)\in ]\bar t-r,\bar t +r[ \times ]\alpha,\beta[: x \in ]\gamma_l(t),\gamma_r(t)[\},
	\end{equation}
	where 
	\begin{equation*}
	\gamma_l(t):= \bar x - r - 2r\|b\|_{L^\infty} + \|b\|_{L^\infty}(t-\bar t +r)\qquad \mbox{and} \qquad\gamma_r(t):= \bar x + r + 2r\|b\|_{L^\infty} - \|b\|_{L^\infty}(t-\bar t +r).
	\end{equation*}
	Since $B_r(\bar t,\bar x) \subseteq E_{\bar t,\bar x}(r) $ this completes the proof of \eqref{E_g_theta}. \\
    {\sc Step 2A:} we show that $g_\theta (h)$ is defined for every $h =Q(t, x)$ with 
    $(t, x) \in E_{\bar t,\bar x}(r)$. To this end, we recall the proof of point {(a)} of Lemma \ref{L_2} and conclude that, since $\dot\gamma_l(t)=\|b\|_{L^\infty}$ and 
	$\dot\gamma_r(t)=-\|b\|_{L^\infty}$, then 
	\begin{equation*}
	\frac{d}{dt}Q(t,\gamma_l(t))\ge 0 \qquad \mbox{and} \qquad \frac{d}{dt}Q(t,\gamma_r(t))\le 0.
	\end{equation*}
     Since for every $t \in ]0, T[$ the function $Q(t, \cdot)$ is monotone non-decreasing, the above inequalities imply that the image of $E_{\bar t,\bar x}(r)$ through $Q$ is contained in the set $Q(\bar t - r, I)$. \\
     {\sc Step 2B:} we consider any Lipschitz continuous extension of $g_\theta$ to 
	the whole $\R$ and set $H := (Q_\theta - g_\theta \circ Q)^2$. We want to show \begin{equation} \label{e:sacher}
	\rho H (t, x) = 0 \qquad \mbox{for }\mathcal L^1\mbox{-a.e. }x \in ]\gamma_l(t),\gamma_r(t)[.
	\end{equation}
      To this end, we first show that $H$ is a distributional solution of the equation 
      \begin{equation}\label{E_eq_tildeQ}
	\partial_t (\rho H) + \partial_x(\rho b H) = 0 \qquad
     \text{on $]0, T[ \times ]\alpha, \beta[$}.
	\end{equation} 
       Fix a test function $\varphi \in C^\infty_c(]0,T[\times ]\alpha,\beta[)$: since $H$ is a Lipschitz continuous function, then we can use $H \varphi$ as a test function for the equation  $\partial_t \rho  + \partial_x(\rho b) = 0$. We obtain 
       \inputencoding{latin1}
	\begin{equation*}
	\begin{split}
	0 = &~\iint_{]0,T[ \times ]\alpha,\beta[}\rho (H \varphi)_t + \rho b (H \varphi)_x dtdx 
	= \iint_{]0,T[ \times ]\alpha,\beta[} H \left( \rho \varphi_t + \rho b \varphi_x\right) + \varphi \left(\rho H_t + \rho b H_x\right)dt dx \\
	= &~ \iint_{]0,T[ \times ]\alpha,\beta[} H \left( \rho \varphi_t + \rho b \varphi_x\right) dt dx \\
	&~ + \iint_{]0,T[ \times ]\alpha,\beta[} 2\varphi (Q_\theta - g_\theta \circ Q) \left( \rho(Q_\theta)_t  +\rho b (Q_\theta)_x - g'_\theta \circ Q(\rho (Q)_t  +\rho b (Q)_x)  \right) dt dx \\
	= &~ \iint_{]0,T[ \times ]\alpha,\beta[} H \left( \rho \varphi_t + \rho b \varphi_x\right) dt dx,
	\end{split}
	\end{equation*}
      that is~\eqref{E_eq_tildeQ}. Note that in the last equality we used that $(Q_\theta)_t = - b  (Q_\theta)_x$ and $ Q_t = - b  Q_x$.

      We now fix $\tau \in ]\bar t-r, \bar t+r[$ and apply Lemma~\ref{l:p32acm} with $C= (\rho H, \rho b H)$, $\Lambda = \{ (t, x): \; t \in ]\bar t-r, \tau[, \; x  \in ]\gamma_l (t), \gamma_r (t)[\}$ and we choose a test function $\psi$ such that $\psi \equiv 1$ on $\Lambda$. Owing to~\eqref{E_eq_tildeQ} we arrive at 
	\begin{equation}\label{E_balance_tildeQ}
	\begin{split}
	\int_{\gamma_l(\tau)}^{\gamma_r(\tau)}\rho H (\tau,x) dx = &~ \int_{\gamma_l(\bar t - r)}^{\gamma_r(\bar t-r)}\rho H (\bar t - r,x) dx -	\int_{\gamma_l} \mathrm{Tr}[(\rho H, \rho b H)] d\mathcal H^1  -
	\int_{\gamma_r} \mathrm{Tr}[(\rho H, \rho b H)] d\mathcal H^1.
	\end{split}
	\end{equation}
	Since $\rho H \ge 0$, $\dot \gamma_l=-\|b\|_{L^\infty}$, and $\dot \gamma_r=\|b\|_{L^\infty}$, then 
	\begin{equation}\label{E_sign_traces}
	 \mathrm{Tr}[(\rho H, \rho b H)] \ge 0 \quad 
     \text{on $\gamma_l$} \qquad \mbox{and} \qquad \mathrm{Tr}[(\rho H, \rho b H)] 
      \ge 0 \quad 
     \text{on $\gamma_r$}.
	\end{equation}
     By plugging~ \eqref{E_sign_traces} and the equality 
	 $H(\bar t - r,x)=0$ for every $x \in ]\gamma_l(\bar t - r),\gamma_r(\bar t - r)[$    
     into  	\eqref{E_balance_tildeQ} we arrive at the equality $H(\tau, x) =0$ 
     for a.e. $x \in ]\gamma_l(\tau),\gamma_r(\tau)[$ and by the 
     arbitrariness of $\tau$ this  yields~\eqref{e:sacher}. \\
	{\sc Step 2C:} we show that  $H^2 \equiv 0$ in $E_{\bar t,\bar x}(r)$.
	Since  $\rho H=0$ a.e. in $E_{\bar t,\bar x}(r)$, then
	\begin{equation*}
	\partial_t H^2=2H \partial_t H= - 4\rho H (Q_\theta - g_\theta\circ Q) ( b \theta -  b g'_\theta \circ Q)=0 \quad \text{a.e. on $E_{\bar t,\bar x}(r)$}. 
	\end{equation*}	
   Since $H (t-\bar r, x)=0$ for every $x \in  ]\gamma_l(\bar t - r),\gamma_r(\bar t - r)[$,
	then $H^2 \equiv 0$ in $E_{\bar t,\bar x}(r)$, i.e. $Q_\theta = g_\theta \circ Q$ in $E_{\bar t,\bar x}(r)$.
	This completes the proof of \eqref{E_g_theta}. \\
{\sc Step 3:} we complete the proof of the lemma. The identity \eqref{E_g_theta_curve} follows from \eqref{E_g_theta} by a standard covering argument. To establish~\eqref{E_trace_null}, we point out that, owing to~\eqref{E_g_theta} and to point (f) in Lemma~\ref{L_2}, the function $\tilde Q$ is constant along $\gamma_{\bar t, \bar x}$. By relying on Lemma \ref{L_1} we then arrive at~\eqref{E_trace_null}. 
\end{proof}
We are now ready to state the main result of this paragraph.
\begin{proposition}\label{P_formula}
Let $\theta$ be a solution of \eqref{IVPtheta2} in the sense of Definition~\ref{d:ibvp}, $\tilde Q$ as in \S\ref{ss:giallo} and $Q_\theta$ as in~\eqref{E_def_Q}.
Then
\begin{equation}\label{E_Qtheta}
\tilde Q = Q_\theta.
\end{equation}
\end{proposition}
\begin{proof}
We fix $\bar t\in ]0,T[$ and we distinguish between three cases. \\
{\sc Case 1:} $\bar x \in ]x_\alpha(\bar t),x_\beta(\bar t)[$. Owing to~\eqref{E_g_theta_curve}, $Q_\theta(\bar t,\bar x)= Q_\theta(0,\gamma_{\bar t,\bar x}(0))$. Since $Q_\theta(0,\alpha)=0$ and
$\partial_xQ_\theta= \rho\theta$, then 
\begin{equation*}
Q_\theta(\bar t,\bar x)= Q_\theta(\gamma_{\bar t,\bar x}(0))= \int_0^{\gamma_{\bar t,\bar x}(0)}[\rho\theta]_0 (x)dx
\stackrel{\eqref{e:dataibvp2}}{=}  \int_0^{\gamma_{\bar t,\bar x}(0)}[\rho]_0\theta_0 (x)dx 
\end{equation*}
and this yields~\eqref{E_Qtheta}. \\
{\sc Case 2:} $\bar x \in ]\alpha, x_\alpha(\bar t)[$. Owing to~\eqref{E_g_theta_curve},
$
Q_\theta(\bar t,\bar x)= Q_\theta(t_*(\gamma_{\bar t,\bar x}),\alpha). 
$ We consider the curve $\gamma_{t_{\min},\alpha +\varepsilon}$, where $t_{\min}$ is the same as in~\eqref{e:limone}: by combining Lemma~\ref{l:strudel} and the minimality property of $t_{\min}$ we conclude that $\gamma_{t_{\min},\alpha +\varepsilon}(t) > \alpha$ for every $t \in [0, t_{\min}]$. Since $\gamma_{t_{\min},\alpha +\varepsilon}$ is a Lipschitz continuous curve,  then owing to~\eqref{E_ass} $\gamma_{t_{\min},\alpha +\varepsilon}(t) < \beta $ for every $t \in [0, t_{\min}]$ and $\varepsilon>0$ sufficiently small. Summing up,  $\gamma_{t_{\min},\alpha +\varepsilon}(t) \in ]\alpha,  \beta[ $ for every $t \in [0, t_{\min}]$.  Owing to~\eqref{E_g_theta_curve}, 
$
Q_\theta(t_{\min},\alpha)=Q_\theta(0,\bar y). $
Since 
\begin{equation*}
Q_\theta(0,\bar y)= \int_0^{\bar y} [\rho]_0 (x)\theta_0(x)dx,
\end{equation*}
by recalling~\eqref{e:mirtillo} we conclude that to prove \eqref{E_Qtheta} it suffices to show that 
\begin{equation}\label{E_claim_case2}
 Q_\theta(t_*(\gamma_{\bar t,\bar x}),\alpha) - Q_\theta(t_{\min},\alpha)= \int_{E^-_{t_*(\gamma_{\bar t,\bar x})}}  \mathrm{Tr}[\rho b](t,\alpha^+) \bar \theta(t)dt.
\end{equation}
%
Owing to Lemma~\ref{L_3} there is a Lipschitz continuous function $g_\alpha:\R \to \R$ such that 
$
H_\theta:= g_\alpha \circ U^-(Q(\cdot,\alpha),[0,t_*(\gamma_{\bar t,\bar x})])
$
coincides with $Q_\theta(\cdot,\alpha)$ on $E^-_{t_*(\gamma_{\bar t,\bar x})}$.
Since $Q = U^-(Q(\cdot,\alpha),[0,t_*(\gamma_{\bar t,\bar x})])$ on $E^-_{t_*(\gamma_{\bar t,\bar x})}$,
then 
\begin{equation*}
\mathrm{Tr}[\rho b](t,\alpha^+)\stackrel{\eqref{E_trace_boundary}}{=} \partial_t Q(t,\alpha) \stackrel{\eqref{E_derivatives_envelope}}{=} \frac{d}{dt} \left[U^-(Q(\cdot,\alpha),[0,t_*(\gamma_{\bar t,\bar x})])\right](t) \le 0 
\quad \text{a.e. on $E^-_{t_*(\gamma_{\bar t,\bar x})}$}.
\end{equation*}
Since $g_\alpha$ is a Lipschitz continuous function, then 
\be \label{e:mora}
     \frac{d }{dt}H_\theta(t)=0 \quad \text{for a.e. $t \in E^-_{t_*(\gamma_{\bar t,\bar x})}$ such that
 $\mathrm{Tr}[\rho b](t,\alpha^+)= 0$.}
\eq
 Since $H_\theta$ coincides with $Q_\theta(\cdot,\alpha)$ on $E^-_{t_*(\gamma_{\bar t,\bar x})}$, then 
 \begin{equation}\label{E_trace_H}
\mathrm{Tr}[\rho b \theta](\cdot ,\alpha^+) = \partial_tQ_\theta(\cdot,\alpha)= \frac{d}{dt}H_\theta
\quad \text{a.e. on $\in E^-_{t_*(\gamma_{\bar t,\bar x})}$}.
 \end{equation}
 Owing to~\eqref{e:dataibvp2}, $\mathrm{Tr}[\rho b \theta](t,\alpha^+) = \mathrm{Tr}[\rho b](t,\alpha^+) \bar \theta(t)$ for  a.e. $t\in ]0,T[$ such that
 $\mathrm{Tr}[\rho b](t,\alpha^+)< 0$ and hence by combining \eqref{E_trace_H} and~\eqref{e:mora} we get 
 \begin{equation}\label{E_derH}
 \frac{d}{dt}H_\theta(t) = \mathrm{Tr}[\rho b](t,\alpha^+) \bar \theta(t) \qquad \text{a.e on $ E^-_{t_*(\gamma_{\bar t,\bar x})}$}.
 \end{equation}
 Moreover, since the derivative of $U^-(Q(\cdot,\alpha),[0,t_*(\gamma_{\bar t,\bar x})])$ vanishes a.e. on $[0,t_*(\gamma_{\bar t,\bar x})] \setminus E^-_{t_*(\gamma_{\bar t,\bar x})}$, then the same holds for $H_\theta$.
 We finally observe that by recalling Lemma~\ref{l:strudel} and~\eqref{e:mirtillo} we have $t_{\min}, t_*(\gamma_{\bar t,\bar x})\in  E^-_{t_*(\gamma_{\bar t,\bar x})}$ and therefore 
 \begin{equation*}
 \begin{split}
  Q_\theta(t_*(\gamma_{\bar t,\bar x}),\alpha) - Q_\theta(t_{\min},\alpha)= &~  H_\theta(t_*(\gamma_{\bar t,\bar x})) - H_\theta(t_{\min}) = \int_{t_{\min}}^{t_*(\gamma_{\bar t,\bar x})} \frac{d}{dt} H_\theta(t) dt \\
 =&~  \int_{E^-_{t_*(\gamma_{\bar t,\bar x})}} \frac{d}{dt} H_\theta(t) dt \stackrel{\eqref{E_derH}}{=} \int_{E^-_{t_*(\gamma_{\bar t,\bar x})}}  \mathrm{Tr}[\rho b](t,\alpha^+) \bar \theta(t) dt.
 \end{split}
 \end{equation*} 
 This establishes \eqref{E_claim_case2} and therefore concludes the analysis of {\sc Case 2.} \\
{\sc Case 3:} $\bar x \in ]x_\beta(\bar t),\beta[$. The argument is analogous to the one in {\sc Case 2} and is therefore omitted.
\end{proof}
\begin{corol}
\label{c:bigne} Under the same assumptions as in the statement of Theorem~\ref{t:genex}, assume that $\theta_1 \in L^\infty(]0, T[ \times ]\alpha, \beta[)$ and $\theta_2 \in L^\infty(]0, T[ \times ]\alpha, \beta[)$ are two solutions of~\eqref{IVPtheta2}, in the sense of Definition~\ref{d:ibvp}. Then $\rho \theta_1= \rho \theta_2$ a.e. on $]0, T[ \times ]\alpha, \beta[$. 
\end{corol} 
\begin{proof}
Let $\tilde Q$ be defined in \S\ref{ss:giallo}. Then, by Proposition \ref{P_formula}, we have
\begin{equation*}
Q_{\theta_1} = \tilde Q = Q_{\theta_2},
\end{equation*}
where $Q_{\theta_1}$ and $Q_{\theta_2}$ are the Lipschitz potentials associated to $\theta_1$ and $\theta_2$ respectively as in \eqref{E_def_Q}.
In particular 
\begin{equation*}
\rho \theta_1 = \partial_x Q_{\theta_1} = \partial_x Q_{\theta_2} = \rho \theta_2 \qquad \mbox{a.e. in }]0,T[\times ]\alpha,\beta[. \qedhere
 \end{equation*}
\end{proof}
\section{Proofs of Corollary~\ref{c:comparison} and of the trace renormalization property} \label{s:main2}
In this section we apply the analysis at the previous section and establish Corollary~\ref{c:comparison} and the trace renormalization property given by Theorem~\ref{t:fumetti}
\subsection{Proof of Corollary~\ref{c:comparison}}
By the uniqueness result in the statement of Theorem~\ref{t:genex}, every $\theta \in L^\infty (]0,T[ \times ]\alpha, \beta[)$ which solves~\eqref{IVPtheta2} in the sense of Definition~\ref{d:ibvp} coincides $\rho \Leb^2$-a.e. with the function defined by 
~\eqref{E_theta_characteristics}. This yields the comparison principle. 
\subsection{Trace renormalization property}\label{Ss_traces}
We first establish a preliminary result. 
\begin{lemma}\label{L_estimate_traces} Assume that $b \in L^\infty (]0, T[ \times ]\alpha, \beta[)$ is a nearly incompressible vector field with density $\rho$, and that $\theta\in L^\infty(]0,T[\times ]\alpha,\beta[)$ 
is a solution of~\eqref{IVPtheta2}, in the sense of Definition~\ref{d:ibvp}. Then 
\begin{equation}\label{E_estimate_traces}
\left.
\begin{array}{ll}
|\mathrm{Tr}[b\rho \theta](t,\alpha^+)|\le & \|\theta\|_{L^\infty}|\mathrm{Tr}[b\rho](t,\alpha^+)|, \\
|\mathrm{Tr}[b\rho \theta](t,\beta^-)|\le &\|\theta\|_{L^\infty}|\mathrm{Tr}[b\rho](t,\beta^-)|
\end{array}
\right.
\qquad \text{for a.e. $t \in ]0,T[$.}
\end{equation}
\end{lemma}
\begin{proof}
We only establish the first inequality in \eqref{E_estimate_traces}, the proof of the second one is analogous. What we actually show is that for every $\bar t \in ]0,T[$ there is a sequence $\varepsilon_n\to 0$ such that  
\begin{equation}\label{E_claim}
\left|  \int_{\bar t-\varepsilon_n}^{\bar t + \varepsilon_n} \mathrm{Tr}[(\rho \theta, \rho b \theta)](t,\alpha^+) dt  \right| \le \|\theta \|_{ L^\infty} 
 \int_{\bar t-\varepsilon_n}^{\bar t + \varepsilon_n} \big|\mathrm{Tr}[(\rho , \rho b )](t,\alpha^+)\big| dt \quad \text{for every $n\in \mathbb{N}$}.
\end{equation}
Then the first inequality in \eqref{E_estimate_traces} follows by the Lebesgue Differentiation Theorem.
Fix $\bar t \in ]0,T[$, let $\delta_n \to 0^+$ be a monotone decreasing sequence and consider the curve $\gamma_n:=\gamma_{\bar t, \alpha + \delta_n}$. We distinguish between two cases:
\begin{equation}\label{E_two_cases}
\begin{split}
&\liminf_{n\to \infty} \min \{t^*(\gamma_n) - \bar t, \bar t - t_*(\gamma_n)\} = 0, \\
&\liminf_{n\to \infty} \min \{t^*(\gamma_n) - \bar t, \bar t - t_*(\gamma_n)\} > 0.
\end{split}
\end{equation}
By possibly passing to suitable subsequences, we can assume that the two inferior limits in \eqref{E_two_cases} are actually limits. \\
{\sc Case 1:} the first condition in \eqref{E_two_cases} holds true. We set $\varepsilon_n:=  \min \{t^*(\gamma_n) - \bar t, \bar t - t_*(\gamma_n)\}$
and we apply Lemma~\ref{l:p32acm} with $C=(\rho,\rho b)$ and $C=(\rho \theta, \rho b \theta)$: note that in both cases $ \mathrm{Div} C=0$.  We apply~\eqref{e:acmnt} with $\Lambda: = E_n(t)$, where $E_n(t)$ is the region defined by setting 
\begin{equation*}
E_n(t):=\{ (t,x)\in ]\bar t-\varepsilon_n, \bar t + \varepsilon_n[ \times ]\alpha,\beta[ : x < \gamma_n(t)  \},
\end{equation*}
and with a test function $\psi$ such that $\psi \equiv 1$ on $E_n(t)$. 
Since $Q$ and $Q_\theta$ are constant along $\gamma_n$, then owing to Lemma~\ref{L_1} the normal traces of both $(\rho, \rho b)$ and $(\rho \theta, \rho b \theta)$ along $\gamma_n$ vanish. This yields 
\begin{equation}\label{E_rho}
\int_{\bar t- \varepsilon_n}^{\bar t + \varepsilon_n} \mathrm{Tr}[b\rho](t,\alpha^+)dt+\int_\alpha^{\gamma_n(\bar t + \varepsilon_n)} \rho(\bar t + \varepsilon_n,x)dx - 
\int_\alpha^{\gamma_n(\bar t - \varepsilon_n)} \rho(\bar t - \varepsilon_n,x)dx=0,
\end{equation}
and
\begin{equation}\label{E_rho_theta}
\int_{\bar t- \varepsilon_n}^{\bar t + \varepsilon_n} \mathrm{Tr}[b\rho\theta](t,\alpha^+)dt+\int_\alpha^{\gamma_n(\bar t + \varepsilon_n)} \rho\theta(\bar t + \varepsilon_n,x)dx - 
\int_\alpha^{\gamma_n(\bar t - \varepsilon_n)} \rho\theta(\bar t - \varepsilon_n,x)dx=0.
\end{equation}
Owing to property (e) in Lemma \ref{L_2}, either 
$\gamma_n(\bar t + \varepsilon_n)\in \{\alpha,\beta\}$ or $\gamma_n(\bar t - \varepsilon_n)\in \{\alpha,\beta\}$ (or both). Since $\gamma_n$ are $\|b\|_{L^\infty}$-Lipschitz continuous curves and $\varepsilon_n$ is converging to $0$ by assumption, for $n$ sufficiently large neither $\gamma_n(\bar t + \varepsilon_n)=\beta$ nor $\gamma_n(\bar t - \varepsilon_n) =\beta$ can be satisfied and hence either $\gamma_n(\bar t + \varepsilon_n)=\alpha$ or $\gamma_n(\bar t - \varepsilon_n) =\alpha$ (or both).
In particular at least one among the second and the third integral vanishes in both \eqref{E_rho} and \eqref{E_rho_theta}.  Owing to the inequality $\rho \ge 0$ we then have  
\begin{equation}\label{E_control}
\begin{split}
\left| \int_{\alpha}^{\gamma_n(\bar t + \varepsilon_n)} \rho\theta(\bar t + \varepsilon_n,x)dx -\right. &
\left. \int_{\alpha}^{\gamma_n(\bar t - \varepsilon_n)} \rho\theta(\bar t - \varepsilon_n,x)dx  \right|  \\
& \le ~
\|\theta\|_{L^\infty} \left| \int_{\alpha}^{\gamma_n(\bar t + \varepsilon_n)} \rho(\bar t + \varepsilon_n,x)dx -
\int_{\alpha}^{\gamma_n(\bar t - \varepsilon_n)} \rho(\bar t - \varepsilon_n,x)dx \right|.
\end{split}
\end{equation}
By combining \eqref{E_rho}, \eqref{E_rho_theta} and \eqref{E_control} we get
\begin{equation*}
\left| \int_{\bar t- \varepsilon_n}^{\bar t + \varepsilon_n} \mathrm{Tr}[b\rho\theta](t,\alpha^+)dt \right| \le \|\theta\|_{L^\infty} \left| \int_{\bar t- \varepsilon_n}^{\bar t + \varepsilon_n} \mathrm{Tr}[b\rho](t,\alpha^+)dt \right|,
\end{equation*}
which yields~\eqref{E_claim}. \\
{\sc Case 2:} we assume that $\liminf_{n\to \infty} \min \{t^*(\gamma_n) - \bar t, \bar t - t_*(\gamma_n)\} > 0$.
We denote by
\begin{equation}\label{E_def_bareps}
\bar \varepsilon:=\min \left\{ \frac{\beta-\alpha}{2\|b\|_{L^\infty}}, \liminf_{n\to \infty} \min \{t^*(\gamma_n) - \bar t, \bar t - t_*(\gamma_n)\} \right\}
\end{equation}
and assume without loss of generality that $\delta_n < (\beta-\alpha)/2$ for every $n\in \mathbb{N}$. Owing to the definition of $\bar \varepsilon$ in
\eqref{E_def_bareps} and to the fact that all the $\gamma_n$ are $\|b\|_{L^\infty}$-Lipschitz continuous curves, we conclude that $\gamma_n(t) \in ]\alpha,\beta[$ for every $t \in ]\bar t - \bar \varepsilon,\bar t + \bar \varepsilon[$. Owing to the Dini Theorem, $\gamma_n$ converges uniformly on $[\bar t - \bar \varepsilon,\bar t + \bar \varepsilon]$ to some Lipschitz continuous curve 
$\bar \gamma$, which satisfies $\bar \gamma(\bar t)=\alpha$.
We claim that for every $n\in \mathbb{N}$ there are $ h_n,  h_{\theta,n} \in \R$ such that 
\begin{equation}\label{E_h_n}
Q(t,\gamma_n(t))= h_n \qquad \mbox{and} \qquad Q_\theta(t,\gamma_n(t))= h_{\theta, n}, 
\quad 
\text{a.e. $t \in ]\bar t - \bar \varepsilon,\bar t + \bar \varepsilon[$}. 
\end{equation}
To establish the first property in~\eqref{E_h_n} we recall property (f) in Lemma~\ref{L_2}. To establish the second property  in~\eqref{E_h_n} we first recall that, by construction, $\gamma_n (t) \in O(\bar t)$ for every $t \in ]\bar t - \bar \varepsilon,\bar t + \bar \varepsilon[$, where the set $O(\bar t)$ is the same as in~\eqref{e:obart}. Next, we recall that $Q_\theta = \tilde Q$ owing to~\eqref{E_Qtheta}. 
Up to subsequences,  $h_n$ and $h_{\theta,n}$ converge to some
 $\bar h, \bar h_\theta\in \R$: since the curves $\gamma_n$ uniformly converge to $\bar \gamma$ and $Q$ and $Q_\theta$ are both continuous functions, the equalities   \eqref{E_h_n} imply that  $Q(t,\bar\gamma(t))=\bar h$ and $Q_\theta(t,\bar \gamma(t))=\bar h_\theta$ for every $t \in ]\bar t - \bar \varepsilon, \bar t + \bar \varepsilon[$. For every $\varepsilon\in ]0,\bar \varepsilon[$ we set 
\begin{equation*}
t_1(\varepsilon):=\inf\{t \in [\bar t- \varepsilon, \bar t]: \bar \gamma(t)=\alpha\}, \qquad \mbox{and} \qquad t_2(\varepsilon):=\sup\{t \in [\bar t, \bar t+ \varepsilon]: \bar \gamma(t)=\alpha\}.
\end{equation*}
Since $\bar \gamma (\bar t) = \alpha$, then 
\begin{equation*}
Q(t_2(\varepsilon),\alpha)=\bar h=  Q(\bar t, \alpha) =Q(t_1(\varepsilon),\alpha) \qquad \mbox{and} \qquad Q_\theta(t_2(\varepsilon),\alpha)=\bar h_\theta=Q(\bar t, \alpha) = Q_\theta(t_1(\varepsilon),\alpha),
\end{equation*}
and hence  
\begin{equation}\label{E_int=0}
\int_{\bar t}^{t_2(\varepsilon)}\mathrm{Tr}[\rho b](t,\alpha^+) dt = 0 = \int_{\bar t}^{t_2(\varepsilon)}\mathrm{Tr}[\rho b \theta](t,\alpha^+) dt, \quad 
\int_{t_1(\varepsilon)}^{\bar t} \mathrm{Tr}[\rho b](t,\alpha^+) dt = 0 = \int_{t_1(\varepsilon)}^{\bar t} \mathrm{Tr}[\rho b \theta](t,\alpha^+) dt. 
\end{equation}
We now apply the same argument as in {\sc Case 1} with $E_n (\bar t)$ replaced by 
\begin{equation*}
E_+(\bar t, \varepsilon):=\{ (t,x)\in ]t_2(\varepsilon), \bar t +  \varepsilon[ \times ]\alpha,\beta[ : x < \bar \gamma(t) \},
\end{equation*}
and get
\begin{equation}\label{E_estimate+}
\left| \int_{t_2(\varepsilon)}^{\bar t + \varepsilon}\mathrm{Tr}[\rho b \theta](t,\alpha^+) dt \right| \le \|\theta\|_{L^\infty} \left| \int_{t_2(\varepsilon)}^{\bar t + \varepsilon}\mathrm{Tr}[\rho b](t,\alpha^+) dt\right|,
\end{equation}
which owing to \eqref{E_int=0} yields 
\begin{equation*}
\left| \int_{\bar t}^{\bar t + \varepsilon}\mathrm{Tr}[\rho b \theta](t,\alpha^+) dt \right| \le \|\theta\|_{L^\infty} \left| \int_{\bar t}^{\bar t + \varepsilon}\mathrm{Tr}[\rho b](t,\alpha^+) dt\right|.
\end{equation*}
We repeat the same argument as in {\sc Case 1} with $E_n (\bar t)$ replaced by 
\begin{equation*}
E_-(\bar t, \varepsilon):=\{ (t,x)\in ]\bar t - \varepsilon, t_1(\varepsilon)[ \times ]\alpha,\beta[ : x < \bar \gamma(t)  \}, 
\end{equation*}
and by relying on \eqref{E_int=0} we get
\begin{equation}\label{E_estimate-}
\left| \int_{\bar t-\varepsilon}^{\bar t}\mathrm{Tr}[\rho b \theta](t,\alpha^+) dt \right| \le \|\theta\|_{L^\infty} \left| \int_{\bar t-\varepsilon}^{\bar t}\mathrm{Tr}[\rho b](t,\alpha^+) dt\right|.
\end{equation}
By combining~\eqref{E_int=0}, \eqref{E_estimate+}, and \eqref{E_estimate-} we get 
\begin{equation*}
\left|\int_{\bar t - \varepsilon}^{\bar t + \varepsilon}\mathrm{Tr}[\rho b \theta](t,\alpha^+) dt \right| \le \|\theta\|_{L^\infty}  \int_{\bar t-\varepsilon}^{\bar t+\varepsilon}\left|\mathrm{Tr}[\rho b](t,\alpha^+)\right| dt, 
\quad \text{for every $\varepsilon\in ]0,\bar \varepsilon[$}
\end{equation*}
and this concludes the proof of the lemma.
\end{proof}
Here is our main result yielding the trace renormalization property
\begin{theorem} \label{t:fumetti}
Assume that $b \in L^\infty (]0, T[ \times ]\alpha, \beta[)$ is a nearly incompressible vector field with density $\rho$, and that $\theta\in L^\infty(]0,T[\times ]\alpha,\beta[)$ 
is a solution of~\eqref{IVPtheta2}, in the sense of Definition~\ref{d:ibvp}. Then for every $q \in C^0(\R)$ we have 
\begin{equation}\label{E_trace}
\mathrm{Tr}[\rho b q(\theta)] (t, \alpha^+) =  
\begin{cases}
\mathrm{Tr}[\rho b] (t, \alpha^+)q\left(\displaystyle{\frac{\mathrm{Tr}[\rho b \theta] (t, \alpha^+)}{\mathrm{Tr}[\rho b ] (t, \alpha^+)}}\right) & \mbox{if }\mathrm{Tr}[\rho b] (t, \alpha^+) \ne 0, \\
0 &\mbox{otherwise}
\end{cases} 
\qquad \text{for a.e. $t \in ]0,T[$}
\end{equation}
An analogous statement holds for the trace $\mathrm{Tr}[\rho b q(\theta)] (\cdot, \beta^-) $.
\end{theorem}
The above result should be compared to~\cite[Theorem 4.2]{AmbrosioCrippaManiglia}: we have much weaker regularity assumptions on $b$ since we only require boundedness and remove the $BV$ regularity hypothesis, but we restrict to the one-dimensional setting. Note that several counterexamples (see for instance~\cite{CrippaDonadelloSpinolo}) show that in higher space dimension the trace renormalization property fails as soon as the $BV$ regularity deteriorates at the domain boundary. 
\begin{proof}[Proof of Theorem~\ref{t:fumetti}]
We separately establish~\eqref{E_trace} on the sets
\begin{equation*}
\begin{split}
F_\alpha^-&:=\{ t \in ]0,T[: \mathrm{Tr}[\rho b](t,\alpha^+)<0\}, \\
 F_\alpha^0&:=\{ t \in ]0,T[: \mathrm{Tr}[\rho b](t,\alpha^+)=0\}, \\
  F_\alpha^+&:=\{ t \in ]0,T[: \mathrm{Tr}[\rho b](t,\alpha^+)>0\}.
\end{split}
\end{equation*}
{\sc Step 1:} we establish~\eqref{E_trace} on $F_\alpha^-$.
Owing to Proposition~\ref{P_existence}, $q(\theta)$ is a solution of the initial boundary value problem 
with initial and boundary data $q(\bar \theta), q([\rho]_0), q( \underline \theta)$. In particular, 
\begin{equation}\label{E_trace_1}
\mathrm{Tr}[\rho bq (\theta)](t,\alpha^+) = \mathrm{Tr}[\rho b](t,\alpha^+) q(\bar \theta(t,\alpha^+)) \quad \text{for a.e. $t \in F_\alpha^-$}.
\end{equation}
Since 
\begin{equation}\label{E_trace_2}
\mathrm{Tr}[\rho b\theta](t,\alpha^+) = \mathrm{Tr}[\rho b](t,\alpha^+) \bar \theta(t,\alpha^+) \quad \text{for a.e. $t \in F_\alpha^-$},
\end{equation}
then by combining \eqref{E_trace_1} and \eqref{E_trace_2} we get \eqref{E_trace} a.e. on 
$ F_\alpha^-$. \\
{\sc Step 2:} by applying Lemma \ref{L_estimate_traces} to the solution $q(\theta)$ in place of $\theta$ we conclude that $\mathrm{Tr}[\rho b q(\theta)](t,\alpha^+)=0$ for a.e. $t \in  F_\alpha^0$, and this yields~\eqref{E_trace} a.e. on $F^0_\alpha$.\\
{\sc Step 3:} we are left to establish \eqref{E_trace} a.e. on $F^+_\alpha$. To this end, we point out 
that by setting  
\begin{equation*}
 \theta'(t,x):= \theta(T-t,x), \qquad  \rho'(t,x):= \rho (T-t,x), \qquad  b'(t,x)=-b(T-t,x),
\end{equation*}
we obtain
\begin{equation*}
\partial_t  \rho' + \partial_x ( \rho'  b')= 0, \qquad \partial_t ( \rho'  \theta') + \partial_x ( \rho' b' \theta')= 0, \qquad \mbox{in }\mathcal D'(]0,T[\times ]\alpha,\beta[).
\end{equation*}
Also, we have 
\begin{equation}\label{E_-t1}
\mathrm{Tr}[  \rho'  b'](t,\alpha^+) = - \mathrm{Tr}[\rho b](T-t,\alpha^+) \qquad \mbox{and} \qquad
\mathrm{Tr}[  \rho'  b'  \theta'](t,\alpha^+) = - \mathrm{Tr}[\rho b \theta](T-t,\alpha^+)
\end{equation}
and similarly
\begin{equation*}
\mathrm{Tr}[  \rho'  b'](t,\beta^-) = - \mathrm{Tr}[\rho b](T-t,\beta^-) \qquad \mbox{and} \qquad
\mathrm{Tr}[  \rho'  b'  \theta'](t,\beta^-) = - \mathrm{Tr}[\rho b \theta](T-t,\beta^-).
\end{equation*}
In particular $\theta'$ is a solution of the initial boundary value problem \eqref{IVPtheta2} with initial datum $\theta_0'= \theta(T,\cdot)$
and boundary data
\begin{equation*}
\bar \theta' = \begin{cases}
\displaystyle{\frac{\mathrm{Tr}[  \rho'  b'  \theta'](t,\alpha^+) }{\mathrm{Tr}[  \rho'  b'](t,\alpha^+)}} & \mbox{if }\mathrm{Tr}[  \rho'  b'](t,\alpha^+) \ne 0, \\
0 &\mbox{otherwise},
\end{cases}
\qquad \mbox{and} \qquad 
\underline \theta' = \begin{cases}
\displaystyle{
\frac{\mathrm{Tr}[  \rho'  b'  \theta'](t,\beta^-) }{\mathrm{Tr}[  \rho'  b'](t,\beta^-)}}
 & \mbox{if }\mathrm{Tr}[  \rho'  b'](t,\beta^-) \ne 0, \\
0 &\mbox{otherwise}.
\end{cases}
\end{equation*}
Owing to Lemma \ref{L_estimate_traces}, we have $\bar \theta', \underline \theta' \in L^\infty (]0,T[)$ and by applying {\sc Step 1} to $\theta'$ we obtain that for $\mathcal L^1$-a.e. $t \in ]0,T[$ such that $\mathrm{Tr}[\rho'b'](t,\alpha^+)<0$ it holds 
\begin{equation*}
\mathrm{Tr}[\rho' b'q (\theta')](t,\alpha^+) = \mathrm{Tr}[\rho' b'](t,\alpha^+) q(\bar \theta'(t,\alpha^+)) \quad \text{for a.e. $t \in ]0,T[$ such that $\mathrm{Tr}[\rho'b'](t,\alpha^+)<0$ }.
\end{equation*}
Owing to~\eqref{E_-t1}, this establishes~\eqref{E_trace} a.e. on $F^+_\alpha$. 
\end{proof}
\section{Propagation of $BV$ regularity and stability} \label{s:confetto}
\subsection{Proof of Proposition~\ref{p:the}} \label{ss:pthe}
\label{s:ibvp_gen}
In this section we establish the proof of Proposition~\ref{p:the} and as a byproduct we also provide an alternative proof of the existence statement in Theorem~\ref{t:genex}, see Remark~\ref{r:cicerchie}.  
We rely on an approximation argument inspired by~\cite[Proof of Proposition 3.13]{DL07}, however the approximation we use here is more complicated because we need to make sure that the normal traces of the approximation converge strongly to the normal trace of the limit. We proceed according to the following steps. \\
{\sc Step 1:} we define an extension of $\rho$ and $b \rho $. We set 
\be \label{e:AB}
    A(t, x): = 
   \left\{
   \begin{array}{ll}
         \rho (t, x) & (t,x) \in [0,T]\times ]\alpha, \beta[ \\
         \rho   (T, x) & (t,x) \in ]T,+\infty[\times ]\alpha, \beta[\\
         1 & x < \alpha, \; x > \beta \\
   \end{array} 
   \right.
   \quad 
   B (t, x) : = 
  \left\{
  \begin{array}{ll}
           b \rho(t, x) & (t,x) \in [0,T]\times]\alpha, \beta[ \\ 
           - \mathrm{Tr}[b \rho] (t, \alpha^+) & (t,x)\in [0,T]\times ]-\infty,\alpha[ \\
            \mathrm{Tr}[b \rho] (t, \beta^-) & (t,x)\in[0,T]\times]\beta,+\infty[ \\
            0 & t>T\,.
  \end{array}
  \right.
\eq
In the previous expression, $\rho(T, \cdot)$ denotes the normal trace of the vector field $(\rho, b \rho)$ at $t=T$.
By direct check, one can verify that, owing to the inequalities $\| \mathrm{Tr}[b \rho] (t, \alpha^+)\|_{L^\infty}, \| \mathrm{Tr}[b \rho] (t, \beta^-)\|_{L^\infty} \leq \| b \rho \|_{L^\infty}$, we have 
\be \label{e:focaccia}
     |B(t, x)| \leq \max \{ \| b \|_{L^\infty}, \| b \rho \|_{L^\infty} \} |A(t, x)|, 
    \quad \text{a.e on $]0, +\infty[ \times \R$}.
\eq
Note that 
$$
    \partial_t A + \partial_x B =0 \quad \text{on $]0, T[ \times ]-\infty, \alpha[$ and on $]0, T[ \times ] \beta, + \infty[$}.
$$
We now want to show 
\be \label{e:sututtoerre}
       \partial_t A + \partial_x B =0 \quad \text{on $]0, T[ \times \R$.}
\eq
To this end, we recall~\cite[Lemma 3.2]{CrippaDonadelloSpinolo} and we apply Lemma~\ref{l:p32acm} with $C=(A, B)$ and $\Lambda= ]0, T[ \times ]- \infty, \alpha[$ and ${\Lambda = ]0, T[ \times ]\beta, + \infty[}$ and define the normal traces $\mathrm{Tr}[B](\cdot, \alpha^-)$ and $\mathrm{Tr} [B] (\cdot, \beta^+)$, that is, very loosely speaking, the normal traces of $B$ ``from outside" the interval $]\alpha, \beta[$. It turns out that 
\be \label{e:annullo}
   \mathrm{Tr}[B](\cdot, \alpha^-) =  -
   \mathrm{Tr}[b\rho](\cdot, \alpha^+), \qquad 
   \mathrm{Tr} [B] (\cdot, \beta^+) = -
     \mathrm{Tr}[b \rho](\cdot, \beta^-).
\eq 
To establish~\eqref{e:sututtoerre} we fix $\psi \in C^\infty_c (]0, T[ \times \R)$ and point out that
\begin{equation*}
\begin{split}
         \int_0^T & \int_\R A \partial_t \psi + B \partial_x \psi dx dt  =
          \int_0^T \int_{-\infty}^\beta A \partial_t \psi + B \partial_x \psi dx dt +
           \int_0^T \int_{\alpha}^\beta A \partial_t \psi + B \partial_x \psi dx dt \\
         & \qquad \qquad +
           \int_0^T \int_{\beta}^{+ \infty} A \partial_t \psi + B \partial_x \psi dx dt \\
        & = \int_0^T \mathrm{Tr} [B](\cdot, \beta^-) \psi (\cdot, \beta) dt +
        \int_0^T \mathrm{Tr} [b \rho](\cdot, \beta^+) \psi (\cdot, \beta) dt +
         \int_0^T \mathrm{Tr} [b \rho](\cdot, \alpha^-) \psi (\cdot, \alpha) dt \\
      & \qquad \qquad +
          \int_0^T \mathrm{Tr} [B](\cdot, \alpha^+) \psi (\cdot, \alpha) dt
       \stackrel{\eqref{e:annullo}}{=} 0,
\end{split}
\end{equation*}
that is~\eqref{e:sututtoerre}. \\
{\sc Step 2:} we define the approximation argument. First, we introduce a cut-off function $\zeta$ such that 
\be 
\label{e:zeta}
        \zeta \in C^\infty (\R), \quad 0 \leq \zeta \leq 1, \quad 
        \zeta \equiv 1 \, \text{on $] - \infty, \alpha + \delta]$}, \quad 
         \zeta \equiv 0 \; \text{on $[  \beta- \delta, + \infty [$}, \quad \delta: = \frac{\beta - \alpha}{10}.
\eq
Next, we introduce two anisotropic convolution kernels $\omega \in C^\infty_c (]0, 1[)$ and $\xi \in C^\infty_c (]-1, 0[) $. We set 
\be \label{e:kernels}
   \omega_n(z): = n \omega(n z), \quad   
   \xi_n (z): = n \xi (nz), \quad 
  \gamma_n (t, x) : = n^2 \xi \left( n t \right) \omega \left( n x \right), \quad
   \eta_n (t, x) : = n^2 \xi \left(n t \right) \xi \left( n x \right) 
\eq
and 
\be \label{e:bienne}
    B_n : = \zeta [B  \ast \gamma_n ]+ (1- \zeta) [B \ast \eta_n], \qquad 
    \rho_n : = \zeta [A  \ast \gamma_n ]+ (1- \zeta) [A \ast \eta_n] + \frac{1}{n}, \qquad 
    b_n: = \frac{B_n}{\rho_n}. 
\eq
Owing to~\eqref{e:focaccia}, we have  
\be \label{e:pomodorini}
      \| b_n \|_{L^\infty} \leq \max \{ \| b \|_{L^\infty}, \| b \rho \|_{L^\infty} \} . 
\eq
We conclude by defining the approximation of the data: we fix $n \in \mathbb N$, recall that $\bar \theta$, $\underline \theta \in BV (]0, T[)$ and $\theta_0 \in BV (]\alpha, \beta[)$ and we introduce the functions $\bar \theta^\ast_n, \underline \theta^\ast_n: ]0, + \infty[ \to \R$, $\theta_{0,n}^\ast: \R \to \R$ by setting 
\be \label{e:liscia}
   \bar \theta^\ast_n (t) :=
    \left\{
    \begin{array}{ll}
              \theta_0(\alpha^+) & 0< t < 2/n \\
              \bar \theta (t) & 2/n < t < T \\
              \bar \theta (T^-)  & T< t \\
    \end{array}
    \right.   
   \quad 
    \underline \theta^\ast_n (t) :=
    \left\{
    \begin{array}{ll}
              \theta_0(\beta^-) & 0< t < 2/n \\
              \underline \theta (t) & 2/n < t < T \\
              \underline \theta (T^-)  & T< t \\
    \end{array}
    \right. 
\eq
and 
\be \label{e:liscia2}
    \theta_{0,n}^\ast (x) : = 
    \left\{
    \begin{array}{ll}
    \theta_0(\alpha^+) & x< \alpha \\
    \theta_0(x) & \alpha < x < \beta \\
    \theta_0 (\beta^-) & \beta < x. \\
    \end{array}
    \right.
\eq
We set 
\be \label{e:sedia}
   \theta_{0, n} : = \theta^\ast_{0,n} \ast \xi_n, \qquad \bar \theta_n : = \bar \theta^\ast_n \ast \xi_n, 
  \qquad \underline \theta_n : = \underline \theta^\ast_n \ast \xi_n,
\eq
where $\xi_n$ is the same as in~\eqref{e:kernels}. \\
{\sc Step 3:} we study the properties of the approximation. We have 
\be \label{e:rbm}
   \rho_n \to \rho \; \text{in $L^1 (]0, T[ \times ]\alpha, \beta[)$ as $n \to + \infty$}, \qquad \| \rho_n \|_{L^\infty} \leq 
C( \| \rho\|_{L^\infty} ) ,    
\eq
 and 
\be 
\label{e:rbm2}
        b_n \rho_n  \to b \rho \; \text{in $L^1 ( ]0, T[  \times]\alpha, \beta[)$ as $n \to + \infty$}, \qquad \| b_n \rho_n \|_{L^\infty} \leq C( \| b \rho \|_{L^\infty}). 
\eq
Also, 
\be \label{e:btm}
    \bar \theta_n \to \bar \theta, \; \underline \theta_n \to \underline \theta  \; \text{in $L^1 (]0, T[)$ as $n \to + \infty$}, \qquad \| \bar \theta_n \|_{L^\infty} \leq  \| \bar \theta \|_{L^\infty}, 
     \| \underline \theta_n \|_{L^\infty} \leq  \| \underline \theta \|_{L^\infty}
\eq
and 
\be \label{e:agenda}
     \mathrm{TotVar} \ \bar \theta_n \leq 
     \mathrm{TotVar} \ \bar \theta + |\theta_0 (\alpha^+) - \bar \theta (0^+)|, 
     \qquad 
     \mathrm{TotVar} \ \underline \theta_n \leq 
     \mathrm{TotVar} \ \underline \theta + |\theta_0 (\beta^-) - \underline \theta (0^+)|.
\eq
Finally, 
\be 
\label{e:btm2}
        \theta_{0, n} \to \theta_0 \; \text{in $L^1 (]\alpha, \beta[)$ as $n \to + \infty$}, \; 
     \quad \|  \theta_{0, n} \|_{L^\infty} \leq 
 \|  \theta_0 \|_{L^\infty}, \quad 
     \mathrm{TotVar} \ \theta_{0, n} \leq \mathrm{TotVar} \theta_{0}. 
\eq
We now recall that $\rho_n$ and $b_n$ are smooth functions and we establish the following convergence result:
\be \label{e:accan}
    h_n : = \partial_t \rho_n + \partial_x [b_n \rho_n ], \qquad h_n \to 0 \; \text{in $L^1(]0, T[ \times ]\alpha, \beta[)$}. 
\eq
We use the following chain of equalities: 
\begin{equation*}
\begin{split}
         h_n & =  \partial_t \rho_n + \partial_x [b_n \rho_n ] \stackrel{\eqref{e:bienne}}{=}
        \zeta [\partial_t A \ast \gamma_n]  + (1-\zeta) [\partial_t A \ast \eta_n] +
        \zeta [\partial_x B \ast \gamma_n]  + (1-\zeta) [\partial_x B \ast \eta_n] \\
       & \quad +
       \zeta' [B \ast \gamma_n - B \ast \eta_n] \\
       & = \zeta [(\partial_t A + \partial_x B) \ast \gamma_n] +
        (1-\zeta) [(\partial_t A + \partial_x B) \ast \eta_n] +  \zeta' [B \ast \gamma_n - B \ast \eta_n]
       \stackrel{\eqref{e:sututtoerre}}{=}  \zeta' [B \ast \gamma_n - B \ast \eta_n]
\end{split}
\end{equation*}
and this last term converges to $0$ in $L^1$ because both $ \zeta' B \ast \gamma_n$ and $ \zeta' B \ast \eta_n$ converge to $\zeta' B$ in $L^1$. 
To conclude, we establish the convergence results
\be 
\label{e:convtraccia}
      - b_n \rho_n (\cdot, \alpha) \to \mathrm{Tr}[b \rho] (\cdot, \alpha^+)  \; 
      \text{in $L^1 (]0, T[) $}, 
      \quad 
       b_n \rho_n (\cdot, \beta) \to \mathrm{Tr}[b \rho] (\cdot, \beta^-)  \; 
      \text{in $L^1 (]0, T[) $}
\eq
as $n \to + \infty$. The reason why we use the approximation construction described in {\sc Step 2} and in particular the anisotropic kernels is exactly to have~\eqref{e:convtraccia}. First, we point out that, if $x < \alpha + \delta$, then $B_n= B \ast \gamma_n$ and, owing to the space anisotropy of the convolution kernel $\gamma_n$ and to the particular structure of $B$, see~\eqref{e:AB}, we have
$$
     b_n \rho_n (\cdot, x)= B_n (\cdot, x)  = - \mathrm{Tr}[\rho b] (\cdot, \alpha^+) \ast \xi_n \quad \text{if $x\leq \alpha$},
$$
that is $b_n \rho_n$ is constant with respect to $x$ on $]- \infty, \alpha]$. 
In the above formula, the convolution is computed with respect to the time variable only. By evaluating the previous formula at $x = \alpha$ and taking the limit $n \to +\infty$ we establish the first convergence result in~\eqref{e:convtraccia}. The proof of the second convergence result in~\eqref{e:convtraccia} is analogous. \\
{\sc Step 4:} we construct the approximating sequences $\{ \theta_n \}_{ n \in \mathbb N}$  by solving the initial-boundary value problem
\be 
\label{e:smootht}
\left\{
       \begin{array}{ll}
              \partial_t \theta_n + b_n \partial_x \theta_n =0 \\
              \theta_n (t, \alpha) = \bar \theta_n (t) & t \in \Gamma^-_{\alpha n} \\
              \theta_n (t, \beta)= \underline \theta_n (t) & t \in \Gamma^-_{\beta n} \\
              \theta_n (0, \cdot) = \theta_{0, n} . \\
      \end{array}
\right.
\eq
In the previous expression, $\Gamma^-_{\alpha n}$ denotes the subset of $\{ (t, \alpha), \; t \in ]0, T[ \}$ such that the characteristic lines of $b_n$ starting at $\Gamma^-_{\alpha n}$ (which are well defined since $b_n$ is a smooth function) are entering the set $]0, T[ \times ]\alpha, \beta[$, whereas $\Gamma^-_{\beta n}$ denotes the subset of $\{ (t, \beta), \; t \in ]0, T[ \}$ such that the characteristic lines starting at $\Gamma^-_{\beta n}$  are entering the set $]0, T[ \times ]\alpha, \beta[$. Note that, since $\rho_n \ge 0$,
$$
    \Gamma^-_{\alpha n} \subseteq \{ t: b_n \rho_n (t, \alpha) \ge 0 \}, \qquad 
       \Gamma^-_{\beta n} \subseteq \{ t: b_n \rho_n (t, \beta) \leq 0 \}.
$$
The equation at the first line of~\eqref{e:smootht} is a transport equation with smooth coefficients and hence the solution of~\eqref{e:smootht} can be constructed by relying on the classical method of the characteristics.  Note furthermore that, owing to the $L^\infty$ bounds in~\eqref{e:btm} and~\eqref{e:btm2}, 
\be \label{e:linftym}
    \| \theta_n \|_{L^\infty} \leq \max \{ \| \bar \theta \|_{L^\infty} , \|  \theta_0 \|_{L^\infty}, \| \underline \theta \|_{L^\infty}\}. 
\eq 
We now establish a uniform total variation bound on $\theta_n$. First, we point out that $\theta_n$ is a smooth function and from the first line of~\eqref{e:smootht} we get 
\be \label{e:continuityb}
    \partial_t v_n + \partial_x [b_n v_n] =0, \qquad v_n := \partial_x \theta_n .
\eq
We now fix $g \in C^1(\R)$ and by multiplying the above equation by $g'(v_n)$
we arrive at 
\be 
\label{e:continuityb2}
       \partial_t g(v_n) + \partial_x [ b_n g (v_n) ] + \partial_x b_n [ g'(v_n) v_n - g (v_n)]  =0. 
\eq
By a standard approximation procedure on $g$, we conclude that $|v_n|$ is a distributional  solution of the continuity equation~\eqref{e:continuityb}. This yields 
\begin{equation*}
\begin{split}
      \frac{d}{dt} \int_\alpha^\beta |v_n| (t, x) dx  & = - \int_\alpha^\beta \partial_x [b_n |
v_n|] (t, x) dx =
     - b_n |v_n| (t, \alpha) + b_n |v_n| (t, \beta) \leq
     |b_n v_n| (t, \alpha) + |b_n v_n| (t, \beta) \\
     & \stackrel{\partial_t \theta_n + b_n \partial_x \theta_n =0}{=} 
     |\partial_t \theta_n | (t, \alpha) + |\partial_t \theta_n| (t, \beta)  
\end{split}
\end{equation*}
and by time-integrating the above equality and recalling~\eqref{e:agenda} and~\eqref{e:btm2} we arrive at 
\be \label{e:ceci}
    \mathrm{TotVar} \ \theta_n (t, \cdot) \leq 
    \mathrm{TotVar} \ \theta_0 +  
     \mathrm{TotVar} \ \underline \theta + 
     \mathrm{TotVar} \ \bar \theta + |\theta_0 (\alpha^+) - \bar \theta (0^+)| +
    |\theta_0 (\beta^-) - \underline \theta (0^+)|.
\eq
By using the first equation in~\eqref{e:smootht} and recalling~\eqref{e:pomodorini} we also get 
\be \label{e:lenticchie}
     |D \theta_n |\leq C (T, \| b \|_{L^\infty}, \| b \rho \|_{L^\infty}, \mathrm{TotVar} \ \theta_0 +  
     \mathrm{TotVar} \ \underline \theta,
     \mathrm{TotVar} \ \bar \theta, |\theta_0 (\alpha^+) - \bar \theta (0^+)|, 
    |\theta_0 (\beta^-) - \underline \theta (0^+)|  ).
\eq
In the above expression, $|D \theta_n|$ denotes the total variation of the distributional gradient of $\theta_n$ over $]0, T[ \times ]\alpha, \beta[$. \\
{\sc Step 5:} we pass to the limit. Owing to~\eqref{e:accan},  $\rho_n \theta_n$ satisfies 
\be 
\label{e:appeq}
    \partial_t [\rho_n \theta_n ] + \partial_x [ b_n \rho_n \theta_n] = h_n \theta_n.
\eq
We can argue as in~\cite[pp. 9-10]{CCS17} and conclude that $\theta_n$ converges to a solution $\theta$ of~\eqref{IVPtheta2} satisfying~\eqref{e:linftysol2}. Note that to apply~\cite[Lemma 4.2]{CCS17} we need~\eqref{e:convtraccia} but we do not need~\eqref{e:ceci} since weak$^\ast$ convergence of $\theta_n$ suffices to pass to the limit. However, by combining~\eqref{e:lenticchie} and~\eqref{e:linftym} and applying the Helly-Kolmogorov Compactness Theorem we conclude that $\theta_n$ strongly converges to $\theta$ in $L^1(]0, T[ \times ]\alpha, \beta[)$ and by passing to the limit in~\eqref{e:ceci} and~\eqref{e:lenticchie} we arrive at~\eqref{e:the} and~\eqref{e:caffelatte}. 
\begin{remark} \label{r:cicerchie}
The argument in \S\ref{ss:pthe} can be easily modified to get an alternative proof of the existence result in Theorem~\ref{t:genex}. Indeed, assume $\theta_0 \in L^\infty (]\alpha, \beta[)$, $\bar \theta, \underline \theta \in L^\infty (]0, T[)$ and set
\be \label{e:liscia3}
   \bar \theta^\ast_n (t) :=
    \left\{
    \begin{array}{ll}
             0 & 0< t < 2/n \\
              \bar \theta (t) & 2/n < t < T \\
              0 & T< t \\
    \end{array}
    \right.   
   \quad 
    \underline \theta^\ast_n (t) :=
    \left\{
    \begin{array}{ll}
            0 & 0< t < 2/n \\
              \underline \theta (t) & 2/n < t < T \\
             0  & T< t \\
    \end{array}
    \right. 
\eq
and 
\be \label{e:liscia4}
    \theta_{0,n}^\ast (x) : = 
    \left\{
    \begin{array}{ll}
      0 & x< \alpha  + 2/n\\
    \theta_0(x) & \alpha + 2/n < x < \beta - 2/n \\
    0 & \beta - 2/n< x. \\
    \end{array}
    \right.
\eq
We then define the functions $\theta_{0,n}, \bar \theta_n$ and $\underline \theta_n$ as in~\eqref{e:sedia} and the function $\theta_n$ as in~\eqref{e:smootht}. All the analysis in {\sc Step 4} and {\sc Step 5} of \S\ref{ss:pthe} carries over, except for the proof of the $BV$ bounds. However, as pointed out in {\sc Step 5} the weak$^\ast$ convergence of $\theta_n$ is enough to pass to the limit in the formulation of the initial-boundary value problem and this yields existence of a solution of~\eqref{IVPtheta2}, in the sense of Definition~\ref{d:ibvp}. 
\end{remark}
\subsection{Continuous dependence on the initial and boundary data}\label{Ss_continuity}
\begin{proposition} \label{p:caffe}
If $\bar \theta_n$, $\theta_{0,n}$ and $\underline \theta_n$ converge respectively to $ \bar \theta$, $\theta_0$ and $\underline \theta$ with respect to the weak* topology in $L^\infty$, then $\rho \theta_n$ converges to $\rho \theta$ with respect to the weak* topology in $L^\infty$.

Moreover if $\bar \theta_n$, $\theta_{0,n}$ and $\underline \theta_n$ converge respectively to $ \bar \theta$, $\theta_0$ and $\underline \theta$ with respect to the strong topology in $L^1$, then $\rho \theta_n$ converges to $\rho \theta$ with respect to the strong topology in $L^1$.
\end{proposition}
The proof follows by the same argument as in the proof of~\cite[Theorem 6.1]{CCS17} and of~\cite[Theorem 6.2]{CCS17}.
%
 \section{Positive velocity fields} \label{s:pos}
\begin{lemma}
\label{l:sign} Under the assumptions of Theorem~\ref{t:fanta}, we have 
$\mathrm{Tr} [b \rho ] (\cdot, \alpha^+) \leq 0$ and 
$\mathrm{Tr} [b \rho ] (\cdot, \beta^-)~\ge~0$
a.e. on $]0, T[$. 
\end{lemma}
\begin{proof}
We only prove $\mathrm{Tr} [b \rho ] (\cdot, \beta^-) \ge 0$ a.e. on $]0, T[$, since the proof of the other inequality is analogous. It suffices to show that 
\be \label{e:tavolo}
      \int_0^T \mathrm{Tr} [b \rho ] (t, \beta^-) \upsilon (t) dt \ge 0, \quad \text{for every $\upsilon \in C^\infty_c (]0, T[)$, $\upsilon \ge 0$}. 
\eq
To this end, we use~\eqref{e:suppout} and we choose a sequence of test functions $\{ \psi_n \}_{n \in \mathbb N}$ having the form ${\psi_n (t, x) : = \upsilon (t) \zeta_n (x)}$, where $\{ \zeta_n \}_{n \in \mathbb N} \subseteq C^\infty_c (\R)$ satisfies 
\be 
\label{e:scrivania}
     \zeta_n(x)=0 \; \text{if $x \leq \beta - \frac{1}{n}$}, \qquad 
    \zeta_n(\beta) =1, \qquad 
    \zeta'_n (x) \ge 0 \; \text{for every $x \leq \beta$}.  
\eq
By plugging $\psi_n$ into~\eqref{e:suppout} we arrive at
\begin{equation*}
\begin{split}
      \int_0^T & \! \! \int_\alpha^\beta \rho  (\upsilon' \zeta_n  + b \upsilon \zeta'_n ) dx dt  
       =
    \int_0^T \! \!  v (t) \zeta_n (\alpha) \mathrm{Tr} [b \rho] (t, \alpha^+)  dt + \int_0^T \! \!   v (t) \zeta_n ( \beta) \mathrm{Tr} [b \rho ] (t, \beta^-)  dt 
\end{split}
\end{equation*}
and by using~\eqref{e:scrivania} this yields 
$$ 
    \underbrace{ \int_0^T \! \! \int_{\beta -1/n}^\beta \rho  \upsilon' \zeta_n dx dt}_{P_n} +
    \underbrace{\int_0^T \! \! \int_\alpha^\beta b \rho    \upsilon \zeta'_n  dx dt}_{S_n}
   =  \int_0^T \! \!   \upsilon (t) \mathrm{Tr} [b \rho ] (t, \beta^-)dt.
$$
Note that $\lim_{n \to + \infty} P_n=0$ and that $S_n \ge 0$ owing to~\eqref{e:scrivania} and to the assumption $b \rho \ge 0$. This implies that the right hand side of the above equation is nonnegative and yields~\eqref{e:tavolo}. 
\end{proof}
\begin{remark}\label{r:libri} 
Owing to Lemma~\ref{L_estimate_traces}, $\mathrm{Tr}[b \rho \theta] (t, \alpha^+)=0$ for a.e. $t \in ]0, T[$ such that $\mathrm{Tr}[b \rho] (t, \alpha^+)=0$, and the same property holds at $\beta$. This implies that~\eqref{e:dataibvp2} are automatically satisfied at the points $t$ where $\mathrm{Tr}[b \rho] (t, \alpha^+)=0$ or  $\mathrm{Tr}[b \rho] (t, \beta^-)=0$. In other words, in~\eqref{IVPtheta2} we could equivalently assign the data $\bar \theta$ and $\underline \theta$ on the sets where $ \mathrm{Tr}[b \rho] (\cdot, \alpha^+)\leq 0$ and $ \mathrm{Tr}[b \rho] (\cdot, \beta^-)\leq 0$, respectively. This in particular yields that, owing to Lemma~\ref{l:sign}, Theorem~\ref{t:genex}, Corollary~\ref{c:comparison}, Proposition~\ref{p:the} extend to the initial-boundary value problem~\eqref{IVPtheta}. 
\end{remark}
\subsection{Proof of Theorem~\ref{t:fanta}} \label{ss:fanta}
We basically rely on the same approximation argument as in \S\ref{ss:pthe} and explicitely construct a solution $\theta$ and, for any given $x \in ]\alpha, \beta[$, a function $\tilde \theta_x$ as in the statement of Theorem~\ref{t:fanta}.
We proceed according to the following steps.\\
{\sc Step 1:} we introduce the approximation argument: we recall the definition of $A$ and $B$  in~\eqref{e:AB}. Since, owing to Lemma~\ref{l:sign}, $\mathrm{Tr}[\rho \theta] (\cdot, \alpha^+) \leq 0$ and $\mathrm{Tr}[\rho \theta] (\cdot, \beta^-) \ge 0$, then $B \ge 0$. Next, we recall~\eqref{e:kernels} and we introduce the following definitions:
\be 
\label{e:cucina}
     \rho_n : = A \ast \gamma_n + \frac{1}{n}, \qquad 
     B_n : = B \ast \gamma_n + \frac{1}{n}, 
    \qquad 
     b_n : = \frac{B_n}{\rho_n}.
\eq
Note that $b_n>0$ and that 
\be \label{e:frigo}
     \| \rho_n \|_{L^\infty} \leq \| \rho \|_{L^\infty} + 1, \qquad 
     \| b_n \rho_n \|_{L^\infty} \leq \| \rho b \|_{L^\infty}.
\eq
Also, by using~\eqref{e:sututtoerre} we have the equality $\partial_t \rho_n + \partial_x [b_n \rho_n ] =0$ on $]0, T[ \times \R$. 
We term $\theta_n$  the solution of the initial-boundary value problem
\be \label{e:tappeto}
\left\{
\begin{array}{ll}
      \partial_t \theta_n + b_n \partial_x \theta_n =0 \\
    \theta_n (\cdot, \alpha) = \bar \theta_n, \qquad \theta_n(0, \cdot) = \theta_{0, n} 
\end{array}
\right.
\eq
and point out that the above problem is well-posed since $b_n$ is smooth and strictly positive. The functions $\theta_{0,n}$ e $\bar \theta_n$ are the same as in~\eqref{e:sedia}. Note that, by construction, $\kappa \leq \theta_n \leq K$ if both $\theta_{0}$ and $\bar \theta$ are comprised between $\kappa$ and $K$.  \\
{\sc Step 2:} we fix $x \in ]\alpha, \beta]$ and establish some a priori estimates.  First, we point out that $\theta_n$ is a smooth function. Next, 
by using the inequality $b_n > 0$ we get $ |\partial_t \theta_n | =
     b _n|\partial_x \theta_n | $. 
We set $v_n: = \partial_x \theta_n$ and recall that $|v_n|$ is a distributional  solution of the continuity equation~\eqref{e:continuityb}. By space-integrating we get 
\begin{equation} \label{e:tfc} 
\begin{split}
    \frac{d}{dt} \int_\alpha^{x} |\partial_x \theta_n | (t, y) dy &=
    \frac{d}{dt} \int_\alpha^{ x} |v_n | (t, y) dy = -
     \int_\alpha^{x} \partial_x [ b_n | v_n | ] =
     b_n |v_n| (t, \alpha) - b_n |v_n| (t, x), 
\end{split}
\end{equation}
which yields 
\begin{equation*} \begin{split}
    |\partial_t \theta_n |(t, {x}) & = 
   b_n |\partial_x \theta_n | (t, { x}) \stackrel{\eqref{e:tfc}}{=}    b_n |\partial_x \theta_n | (t, \alpha) - \frac{d}{dt} \int_\alpha^{{x}} 
    |\partial_x \theta_n | (t, y) dy \\ &
    =  |\partial_t \theta_n| (t, \alpha) 
   - \frac{d}{dt} \int_\alpha^{{ x}}  |\partial_x \theta_n| (t, y) dy 
  \stackrel{\eqref{e:tappeto}}{=}   |\bar \theta_n'|
     -\frac{d}{dt} \int_\alpha^{{ x}}  |\partial_x \theta_n| (t, y) dy \,.
\end{split}
\end{equation*} 
By time-integrating the above expression we arrive at 
\be \label{e:tvm2}
    \int_0^T  |\partial_t \theta_n |(t,  {x}) dt \leq 
     \int_\alpha^\beta |\theta'_{0,n} |  dx  + \int_0^T  |\bar \theta_n'| dt
    \stackrel{\eqref{e:agenda},\eqref{e:btm2}}{\leq}
   \mathrm{TotVar} \,  \theta_0+ 
     | \theta_0(\alpha^+) -   \bar \theta (0^+)| +
    \mathrm{TotVar} \, \bar \theta, 
\eq
for every $x \in ]\alpha, \beta]$. \\ 
{\sc Step 3:} we pass to the limit. By arguing as in {\sc Step 3} of \S\ref{ss:pthe} we get that
the sequence $\theta_n$ converges up to subsequences to a limit function $\theta$ weakly$^\ast$ in $L^\infty (]0, T[ \times ]\alpha, \beta[)$ as $n \to + \infty$, and that $\theta$ is a solution of the initial-boundary value problem~\eqref{IVPtheta} in the sense of Definition~\ref{d:ibvp}. Also, we fix $x \in ]\alpha, \beta[$, recall~\eqref{e:frigo} and conclude that, up to subsequences, the functions $b_n \rho_n(\cdot, x)$ and $b_n \rho_n \theta_n (\cdot, x)$ converge weakly$^{\ast}$ in the sense of measures to some limit functions. By passing to the limit in the definition of normal trace~\eqref{e:acmnt} we conclude that the limit functions are exactly $\mathrm{Tr}[b \rho] (\cdot, x)$ and $\mathrm{Tr}[b\rho \theta](\cdot, x)$, respectively. By recalling~\eqref{e:tvm2} and relying on the Helly-Kolmogorov Compactness Theorem we conclude that, up to subsequences, the function $\theta_(\cdot, x)$ strongly converges in $L^1(]0, T[)$ to some limit function $\tilde \theta_x$ satisfying~\eqref{e:bvatx}. By passing to the limit in the equality $b_n \rho_n \theta_n (\cdot, x) = b_n \rho_n (\cdot, x) \theta_n (\cdot, x)$ we arrive at~\eqref{e:bvatx}. We are left to establish~\eqref{e:traghetto}: it suffices to recall that, as pointed out in {\sc Step 1},  $\kappa \leq \theta_n (\cdot, x) \leq K$ provided both $\theta_{0}$ and $\bar \theta$ are comprised between $\kappa$ and $K$. 
\subsection{A counterexample to~\eqref{e:bvatx} for sign-changing $b$} \label{sss:sprite}
We set $T =1$, $\alpha =-4$, $\beta =4$ and consider the domain $]0, 1[ \times ]-4, 4[$.
We set 
\be \label{e:mocaccino}
     b(t, x) : = a'(t), \qquad a(t) = (1-t)^2 \sin \left( \frac{\pi}{1-t} \right) 
\eq
and observe that $b \in L^\infty (]0, 1[ \times ]-4, 4[)$ is a nearly incompressible vector field with density $\rho \equiv 1$. We set 
\be \label{e:capoinb}
     \theta_0 (x) : = 
\left\{ 
     \begin{array}{ll} 
     -1 & -4 < x <0 \\
      1 & 0 < x < 4 \\
    \end{array}  
\right.
       \qquad 
     \bar \theta(t) : = -1 \qquad \text{and} \qquad
     \underline \theta (t) : =  1
\eq 
and consider the initial-boundary value problem~\eqref{IVPtheta2}. Since $b$ is a regular vector field,  we can apply the classical method of characteristics: the solution $\theta$ is constant along the curves with slope $b$. 
If $x_0 \in ]-4, 4[$, then the characteristic curve starting at $x_0$  is 
\be \label{e:cioccolata}
     X(t, x_0)= x_0 + (1-t)^2 \sin \left( \frac{\pi}{1-t} \right)
\eq
and hence the solution of the initial-boundary value problem~\eqref{IVPtheta2},~\eqref{e:capoinb} is 
\be
\label{e:guarana}
       \theta(t, x) = 
       \left\{
       \begin{array}{ll}
       -1 & x < X(t, 0) \\
       1   & x > X(t, 0). \\
      \end{array}
     \right.
\eq
Note that the above $\theta \in L^\infty (]0, 1[ \times ]-4, 4[)$ is the unique solution given by Theorem~\ref{t:genex} because $\rho=1>0$. 
We now focus on the vertical segment $x=0$. Since $\rho \equiv 1$ and $b$ is a $C^\infty$ vector field, then 
$\mathrm {Tr} [\rho b](t, \cdot) = b(t)$ for a.e. $t \in ]0, 1[$. Since $b(t) \neq 0$ for a.e. $t \in ]0, 1[$, then the value of $\tilde \theta_0$ satisfying~\eqref{e:bvatx} is uniquely determined for a.e. $t \in ]0, 1[$ and by using the explicit expression of $\theta$ in~\eqref{e:guarana} we get $\tilde \theta_0 = \theta(\cdot, 0)$. Owing to the explicit expressions~\eqref{e:cioccolata} and~\eqref{e:guarana}, this yields $\mathrm{TotVar}\, \tilde \theta_0 = + \infty$ and provides a counterexample to~\eqref{e:bvatx} in the case of the sign-changing velocity $b$ given by~\eqref{e:mocaccino}.
\subsection{Proof of Theorem~\ref{t:genex} in the case $b \ge 0$} \label{ss:gazzosa}
The existence 
follows from the same argument as in \S\ref{ss:fanta}, provided we define the initial and boundary datum as in~\eqref{e:liscia3} and~\eqref{e:liscia4}. We are left to establish uniqueness. 
We fix $\theta_1, \theta_2$ distributional solutions of~\eqref{IVPtheta} and we show that $\rho \theta_1=\rho \theta_2$ a.e. on $]0, T[ \times ]\alpha, \beta[$. Owing to the fact that the equation is linear it suffices to show that, if $\bar \theta \equiv 0$ and $\theta_0 \equiv 0$, then any solution of~\eqref{IVPtheta} satisfies $\rho \theta =0$ a.e. on $]0, T[ \times ]\alpha, \beta[$. To this end we adapt the argument in the proof of~\cite[Theorem 2]{Panov}. Let $Q_\theta: [0, T] \times [\alpha, \beta] \to \R$ be the Lipschitz continuous potential function defined by~\eqref{E_def_Q}.  
We now proceed according to the following steps. \\
{\sc Step 1:} {\color{blue}} we show that 
\be \label{e:zeropot}
Q_\theta(0, x)=0 \; \text{for every $x \in [\alpha, \beta]$}, \qquad 
 Q_\theta(t, \alpha)=0 \; \text{for every $t \in [0, T]$.}
\eq
 We recall that by assumption the function $\theta$ satisfies~\eqref{e:suppout} with $[\rho \theta]_0=0$ owing to~\eqref{e:dataibvp} and to the equality $\theta_0 \equiv0$. This implies that, for every $\zeta \in  C^\infty_c (]\alpha, \beta[)$ and every $\nu_n \in C^\infty (]- \infty, T[)$ we have 
$$
    0= \int_0^T \! \! \int_\alpha^{\beta} \rho \theta (\zeta \nu_n' + b  \zeta' \nu_n ) dx dt \stackrel{\eqref{E_def_Q}}{=}
    \int_0^T \! \! \int_\alpha^{\beta} \partial_x Q_\theta (\zeta \nu_n' + b  \zeta' \nu_n ) dx dt\,.
$$
We now choose the sequence $\{ \nu_n \}_{n \in \mathbb N}$ by setting $\nu_n (t) : = \nu (tn)$, where $\nu$ is a standard cut-off function such that $\nu \equiv 1$ on $]-1, 0]$, $\nu \equiv 0$ on $[1, + \infty[$ and $\nu'_n \leq 0$ on $[0, + \infty[$. By letting $n \to + \infty$ we arrive at 
$$
    \int_\alpha^{\beta} \partial_x Q_\theta (0, \cdot) \zeta dx =0, \quad 
    \text{for every $\zeta \in C^\infty_c (]\alpha, \beta[)$}
$$
and this implies that $Q_\theta(0, x)$ is a constant function. By recalling~\eqref{E_def_Q} we conclude that $Q_\theta(0, x) =0$ for every $x \in [\alpha, \beta]$. The proof of the second equality 
in~\eqref{e:zeropot} is analogous. \\ 
{\sc Step 2:} we set $z (q) : = q^2 / (1 + q^2)$ and show that the function $z(Q_\theta)$ satisfies 
\be \label{e:zeta}
    \partial_t [\rho z(Q_\theta)] + \partial_x [b \rho z (Q_\theta) ] =0, \quad 
    \mathrm{Tr}[b \rho z(Q_\theta) ]( \cdot, \alpha^+)  =0, \quad 
    [\rho z (Q_\theta)]_0 =0, 
     \quad 
    \mathrm{Tr}[b\rho z(Q_\theta)]( \cdot, \beta^-)  \ge 0. 
\eq
To establish the first three equalities in~\eqref{e:zeta} it suffices to show that 
\be \label{e:zeta2}
       \begin{split}
    \int_0^T \! \! \int_\alpha^\beta \rho z(Q_\theta) (\partial_t \xi + b \partial_x \xi) dx dt  = 0 \qquad 
       \text{for every $\xi \in C^\infty_c (]- \infty, T[ \times ]-\infty, \beta[)$.}
\end{split}
\eq
To this end we recall Remark \ref{r:carta} and formula \eqref{e:suppout2}
and apply a standard approximation argument on the test function to conclude that
\be \label{e:xilo}
 \begin{split}
    \int_0^T \! \! \int_\alpha^\beta \rho  (\partial_t \psi + b \partial_x \psi) dx dt & = 
    \int_0^T \! \!  \psi \mathrm{Tr} [b \rho ] (\cdot, \alpha^+)  dt -
    \int_\alpha^\beta [\rho]_0 \psi(0, \cdot) dx 
\end{split}
\eq 
for every Lipschitz continuous function $\psi: \R^2 \to \R$ which is compactly supported in $]-\infty, T[ \times ]- \infty, \beta[$. We now 
fix $\xi \in C^\infty_c (]-\infty, T[ \times ]-\infty, \beta[)$ and we set 
$$
    \psi (t, x) = \left\{
    \begin{array}{ll}
    \xi z (Q_\theta)(t, x) & (t, x) \in ]- \infty, T[ \times ]\alpha, \beta[ \\
    0 & \text{elsewhere.}
    \end{array}
   \right.
$$ 
Note that $\psi$ is a Lipschitz continuous function owing to~\eqref{e:zeropot} and to the equality $z(0)=0$. We plug the above expression into~\eqref{e:xilo}, we point out that owing to~\eqref{e:zeropot} $\psi (\cdot, \alpha) \equiv 0$, $\psi(0, \cdot) \equiv 0$ and get 
\begin{equation*}
 \begin{split}
    0 = \int_0^T \! \! \int_\alpha^\beta \rho  (\partial_t \psi + b \partial_x \psi) dx dt & = 
    \int_0^T \! \! \int_\alpha^\beta \rho  
   \big(\partial_t \xi z(Q_\theta) + \xi \partial_t [z(Q_\theta)] + b \partial_x \xi z(Q_\theta) +b \xi \partial_x [z(Q_\theta)] \big) dx dt  \\ 
   & \stackrel{\eqref{E_def_Q}}{=}
    \int_0^T \! \! \int_\alpha^\beta \rho  z(Q_\theta)
   (\partial_t \xi + b \partial_x \xi   ) dx dt, 
\end{split}
\end{equation*}
 that is~\eqref{e:zeta2} owing to the arbitrariness of the function $\xi$. To establish the last inequality in~\eqref{e:zeta} it suffices to apply Lemma~\ref{l:sign} and recall that $b \rho z(Q_\theta) \ge 0$. \\
{\sc Step 3:} we establish the equality $\rho z (Q_\theta) =0$ a.e. on $]0, T[ \times ]\alpha, \beta[$. We set $\psi (t, x)= \nu_n(t) \zeta (x)$, where $\nu_n \in C^\infty_c (]-\infty, T[)$, $\nu_n \ge 0$ and $\zeta$ satisfies $\zeta (x)=1$ for every $x \in [\alpha, \beta]$. We combine~\eqref{e:zeta} with~\eqref{e:suppout} and get 
$$
      \int_0^T \! \! \int_\alpha^\beta \rho z(Q_\theta)  \partial_t \nu_n dx dt
       = \int_0^T \nu_n \mathrm{Tr} [b \rho z(Q_\theta)] (\cdot, \beta^-) dt \ge 0. 
$$
Fix $t \in ]0, T[$ and choose a suitable sequence of test functions $\{ \nu_n\}_{n \in \mathbb N}$ converging to the characteristic function of $]-\infty, t[$. By recalling Lemma~\ref{l:dafermos}, the above inequality yields 
$$
    \int_\alpha^\beta \rho z(Q_\theta)  (t, \cdot) dx \leq 0
$$
in turn implying $\rho z(Q_\theta) = 0$ a.e. $]0, T[\times ]\alpha, \beta[$.   \\
{\sc Step 4:} we conclude the proof. The equality $\rho z(Q_\theta)\equiv 0$ yields $\rho \theta z (Q_\theta) =0= \rho b \theta z(Q_\theta)$. We choose a function $Z: \R \to \R$ such that $Z'=z$ and we point out that $Z$ is strictly monotone. Also, 
$$
    \partial_x [Z(Q_\theta)] = z(Q_\theta) \partial_x Q_\theta = z(Q_\theta) \rho \theta =0, \qquad 
    \partial_t [Z(Q_\theta)] = z(Q_\theta) \partial_t z(Q_\theta)= - z(Q_\theta) b \rho \theta =0.
$$
This implies that $Z(Q_\theta)$ is constant and, by the strict monotonicity of $Z$, that $Q_\theta$ is constant. Since $Q_\theta(0, \alpha) =0$, this yields $Q_\theta \equiv 0$ and by~\eqref{E_def_Q} we eventually get $\rho \theta \equiv 0.$
\hspace{5cm}$\Box$
\begin{remark} \label{r:rosso}
As pointed out before, the uniqueness proof discussed in \S\ref{ss:gazzosa} is based on an argument due to Panov~\cite{Panov}. The reason why it fails in the case of a sign-changing velocity field $b$ (and hence the reason why we had to rely on the more technical construction discussed in \S\ref{s:main}) is because we cannot a-priori prove that $Q_\theta( \cdot, \alpha) \equiv 0$. Indeed, the heuristic reason why this inequality holds true in the case $ b \ge 0$ is because $\partial_t \theta= - b \rho \theta$ and $\theta (\cdot, \alpha) \equiv 0$, which yields $\partial_t \theta =0$ at $x = \alpha$. 
In the case of a sign-changing coefficient the equality $\theta (\cdot, \alpha) \equiv 0$ no longer holds true and hence we cannot conclude that $Q_\theta( \cdot, \alpha)\equiv 0$.
\end{remark}

\section*{Acknowledgments}
S.D. and L.V.S. are partially supported by the INDAM-GNAMPA project 2020 \emph{Modelli differenziali alle derivate parziali per fenomeni di interazione}. E.M. is supported by the SNF Grant 182565. Part of this work was done while S.D. was affiliated to IMATI-CNR, Pavia. 
\bibliographystyle{plain}
\bibliography{1dtra}
\end{document}